\documentclass[11pt,twoside]{amsart}
\usepackage{amsfonts,amsmath,amstext,amsbsy,amsthm,amscd,amsmath,amssymb}
\usepackage[utf8]{inputenc}
\usepackage{hyperref}
\usepackage[T1]{fontenc}

\newcommand{\norme}[1]{\left\Vert #1\right\Vert}
\newtheorem{Lemme}{Lemma}[section]
\newtheorem{Prop}{Proposition}[section]  

\newtheorem{Rmq}{Remark}[section]
\newtheorem{Thm}{Theorem}[section]

\newtheorem{Cor}{Corollary}[section]
\theoremstyle{remark}
\newtheorem*{Proof}{Proof}
\usepackage{indentfirst}
\newcommand{\be}{\begin{equation}}
\newcommand{\ee}{\end{equation}}
\newcommand{\ba}{\begin{array}}
\newcommand{\ea}{\end{array}}
\newcommand{\bea}{\begin{eqnarray}}
\newcommand{\eea}{\end{eqnarray}}
\newcommand{\bee}{\begin{eqnarray*}}
\newcommand{\eee}{\end{eqnarray*}}

\newcommand{\B} {\mathbb{B}}
 
\newcommand{\N} {\mathbb{N}}
\newcommand{\R} {\mathbb{R}}    
\newcommand{\Z} {\mathbb{Z}} 

\newcommand{\cC} {\mathcal{C}}     
\newcommand{\cD} {\mathcal{D}}   
     
\newcommand{\cH} {\mathcal{H}}     
\newcommand{\cL} {\mathcal{L}}      
     
\newcommand{\cR} {\mathcal{R}}

\def \with {\quad\!\hbox{with}\!\quad}
\def \andf {\quad\!\hbox{and}\!\quad}

\def\dn{\delta\!n}
\def\du{\delta\!u}
\def\dv{\delta\!v}
\def\dR{\delta\!R}
\def\dU{\delta\!U}
\def\dV{\delta\!V}

\def\d{\partial}

\def\wt{\widetilde}

\def\ddj{\dot\Delta_j}

\begin{document}
\title[Partially dissipative one-dimensional hyperbolic systems]{Partially dissipative one-dimensional
hyperbolic systems in the critical regularity setting, and applications
}
\author{Timoth\'ee Crin-Barat and Rapha\"el Danchin}

\subjclass[2010]{35Q35; 76N10}
\keywords{Hyperbolic systems, critical regularity, time decay, partially dissipative}

\begin{abstract}
   Here we develop a method for investigating 
   global strong  solutions of  partially dissipative 
   hyperbolic systems in the critical regularity setting.
         Compared to the recent  works by  Kawashima and Xu, we use
   \emph{hybrid} Besov spaces with different regularity exponent in low and high 
   frequency. This allows to consider more general data  and  to track the exact dependency on the dissipation parameter for the solution. 
  Our approach enables us to go  beyond the $L^2$ framework in the treatment of  the low frequencies of the solution, which 
  is totally new, to the best of our knowledge. 
   
  Focus is on the one-dimensional setting
 (the multi-dimensional case will be considered in a forthcoming paper)
 and, for expository purpose, the first part of the paper is 
 devoted to a toy model that may be seen as a simplification of the compressible 
 Euler system with damping. 
 More elaborated systems (including the compressible Euler system with 
 general increasing pressure law) are considered at the end  of the paper. 
   \end{abstract}

\maketitle

\section*{Introduction}

The study of the global existence issue for so-called partially dissipative hyperbolic systems of
balance laws goes back to the seminal work  of Kawashima \cite{Kawa1}. Recall that  a general  $n$-component systems of balance laws in $\mathbb{R}^d$ reads:
\begin{equation}
\frac{\partial w}{\partial t} + \sum_{j=1}^d\frac{\partial F_j(w)}{\partial x_j}=Q(w). \label{GEQ}
\end{equation}
Here  the unknown $w=w(t,x)$ with $t\in\R^+$ and $x\in\mathbb{R}^d$  is  valued in  an open convex subset $\mathcal{O}_w$
of $\mathbb{R}^n$ and $Q, F_j :\mathbb{R}^n\rightarrow \mathcal{O}_w$ are given $n$-vector valued smooth functions on~$\mathcal{O}_w$. 
\smallbreak
It is well known that  classical  systems of conservation laws (that is with $Q(w)=0$) 
supplemented with 
smooth data  admit local-in-time strong solutions that may develop singularities (shock waves) in finite time even if the initial data are small  perturbations of a constant solution  (see for instance the works by Majda in \cite{Majda} and Serre in \cite{Serre}).  
A  sufficient condition for  global existence  for small perturbations of a constant
solution $\bar w$ of \eqref{GEQ} is 
the \emph{total dissipation hypothesis}, namely the damping (or dissipation)  term $Q(w)$ acts directly on   each component of the system, making
the whole solution  to tend to $\bar w$  exponentially fast. 
However, in most evolutionary systems coming from physics, that condition   is not verified, and even though global-in-time strong solutions do exist, exponential decay is very unlikely.
 A more reasonable 
assumption is that dissipation acts 
 only on some  components of the system. After suitable change of coordinates, we may write:\begin{equation}
Q(w)=\left(\begin{matrix} 0\\q(w) \end{matrix}\right) \label{PartDissip}
\end{equation} where $0\in \mathbb{R}^{n_1}, q(w)\in \mathbb{R}^{n_2}$, $n_1,n_2\in\mathbb{N}$ and $n_1+n_2=n.$ 
This so-called \emph{partial dissipation hypothesis} arises in many applications such as gas dynamics or numerical simulation of conservation laws by relaxation scheme. A well
known example is the damped compressible Euler system for isentropic flows that we will be investigated at the end
of the paper. For this system, the works by Wang and Yang \cite{Wang} and Sideris, Thomases and Wang \cite{Sideris} pointed out that the dissipative mechanism, albeit  only present in the velocity equation, can prevent the formation of singularities that would occur if $Q\equiv0.$
\smallbreak
Looking for conditions 
on the systems of the form \eqref{GEQ}-\eqref{PartDissip}  guaranteeing   global 
existence of strong solutions for small perturbations of a constant solution $\bar w$ 
  goes back to the paper of Shizuta and Kawashima \cite{SK}, the thesis of Kawashima \cite{Kawa1} and, more recently, to the paper of Yong \cite{Yong}.
  Their researches reveal  the importance of a rather explicit   linear stability criterion, 
  that is nowadays called  the  (SK) (for \textit{Shizuta-Kawashima})  condition 
  and of  the existence  of an entropy that  provides a suitable symmetrisation of the system.
  Roughly speaking, (SK) condition ensures that the 
  partial damping provided by \eqref{PartDissip} acts on  all the components
  of the solution, although indirectly, so that all the solutions of \eqref{GEQ} emanating 
  from small perturbations of $\bar w$  eventually tend to $\bar w,$ while
 the paper by Yong provides tools to get quantitative estimates on the solutions when $Q(\bar{w})=0$.
 As observed by  Bianchini, Hanouzet and Natalini \cite{BHN}, in many   situations, 
   a careful analysis of the Green kernel of the linearized system about $\bar w$
   allows to get explicit (and optimal) algebraic rates of convergence in $L^p$ of smooth global solutions to $\bar w.$
 Let us finally  mention  that a more general  approach
 has been proposed  by  Beauchard and Zuazua in \cite{BZ}, that 
 allows to handle  partially dissipative systems that need not satisfy the (SK) condition. 
\smallbreak
 Recently,  Kawashima and Xu in \cite{XK1} and \cite{XK2} extended all the prior works
on partially dissipative hyperbolic systems satisfying the (SK) and entropy conditions (including  the compressible Euler system with a damping term)
 to `critical'   non-homogeneous Besov spaces of  $L^2$ type. 
 To obtain their results, they symmetrized the system thanks to the entropy hypothesis, applied a frequency localization argument relying on the Littlewood-Paley decomposition and used new properties concerning Chemin-Lerner's spaces.
They took advantage of the  equivalence between the condition (SK) and the existence of a compensating function so as to 
 to exhibit global-in-time $L^2$ integrability properties of all the components of the system. 
\medbreak
The present paper focuses on the particular  situation where  the space dimension is $d=1$ and the number of components is $n=2$ (the multi-dimensional case will be  investigated in a forthcoming  paper  \cite{CBD2} for the  whole class of partially dissipative systems verifying the (SK) condition).
Our goal is to propose a method and a functional framework with \emph{different} regularities for low and high frequencies. 
For the high frequencies, we do not really have the choice as it is known that the optimal regularity for local well-posedness
in the context of general quasilinear hyperbolic systems,  is given by the `critical'   Besov space $\B^{\frac32}_{2,1}.$ The novelty here is that we propose to look 
at the low frequencies of the solution in another space, not necessarily related to $L^2.$  
The advantage is not only that we will be able to consider a larger class of initial data that
may be less decaying at infinity,  but also  that  one can  easily keep  track of  the dependency of the solution 
with respect to the dissipation coefficient, and thus have some informations on the 
large dissipation asymptotics. 
Various considerations lead us to think that a suitable  space for low frequencies is 
the homogeneous Besov space $\dot\B^{\frac1p}_{p,1}$ (with, possibly, $p>2$) that 
corresponds to the critical embedding in $L^\infty.$  
\smallbreak
For expository purpose, we spend most of the paper implementing
our method on a simple `toy model'  that may be seen as a simplification of the
one-dimensional compressible Euler system with damping,  and pressure law $P(\rho)=\frac12\rho^2$, 
namely
$$ \left\{\begin{array}{lll}\d_tu+v\d_xu+\d_xv=0 \quad&\hbox{in}\quad& \R^+\times\R,\\
 \d_tv+v\d_xv+\d_xu +\lambda v=0 \quad&\hbox{in}\quad& \R^+\times\R,\\
(u,v)|_{t=0}=(u_0,v_0).  \quad&\hbox{on}\quad&\R\end{array}\right.\leqno(TM_\lambda)
$$ 
Above, the unknown   $u$  may be seen as the discrepancy to the reference density normalized to $1$,
(then,  the first equation is a simplification of the mass balance), while  the   unknown $v$ stands for the velocity 
of the fluid, and the second  equation corresponds to the evolution of velocity with a friction term of magnitude $\lambda>0$ (which could also be interpreted as a relaxation parameter).  
\smallbreak
In order to have a robust method that can be adapted to more involved systems,  we shall not 
compute explicitly the solution of
 the linearized system $(TM_\lambda)$ about $(0,0),$  but rather use modified energy arguments 
(different from those of S. Kawashima in his thesis \cite{Kawa1}) and suitable change of unknowns. 
More specifically,  we will introduce a `modified' velocity that plays the same role as  the `viscous effective flux'  in the works of Hoff \cite{Hoff} and, more recently, of Haspot  \cite{Boris} dedicated to the compressible Navier-Stokes equations.

Our approach will enable us to obtain more accurate estimates and a weaker smallness condition than
 in prior works (in particular \cite{Kawa1,Yong,BZ,XK1D}). We will see that it is enough to 
 assume that the low frequencies of the data have Besov regularity for some Lebesgue index  that may be greater than $2.$
 Also, we will improve the decay that was obtained for the compressible Euler system with damping in \cite{XK1D} and,  adapting an idea from Xu and Xin in \cite{XuXin} for the compressible Navier-Stokes system
will enable us to  discard the additional smallness assumption on the low frequencies that is usually required to obtain the decay estimates. 
  \medbreak
 The rest of the paper unfolds as follow. 
 In  Section \ref{s:main},  we present our main results for $(TM_\lambda),$   namely the global existence of a solution corresponding to  small data with optimal estimates with respect to 
 the dissipation coefficient, and time decay estimates. 
 In the next section, we focus on the particular case of data with regularity in Besov spaces built on $L^2,$ and prove 
  global existence in this case, as well as the time decay estimates. The method we here propose is different than the one for the general 
 case, and is more easily extendable to the multi-dimensional setting. 
 In Section \ref{s:Lp}, we propose another method that allows to get our global existence result for a larger class of data, not necessarily in $L^2$ type spaces. 
The next two sections are devoted to  adapting our  results, first for the isentropic Euler system with damping, 
and next for a  general class of one-dimensional systems of two conservations laws, with partial damping. 
Some technical lemmas are proved in Appendix. 

\medbreak\noindent{\bf Acknowledgments.}  
 The second author is partially  supported  by the ANR project INFAMIE (ANR-15-CE40-0011).


\section{Main results}\label{s:main}

Before stating the main results,  we need to introduce a few notations.
First, throughout the paper, we fix a homogeneous  Littlewood-Paley decomposition $(\ddj)_{j\in\Z}$
that is defined  by 
$$\ddj\triangleq\varphi(2^{-j}D)\with \varphi(\xi)\triangleq \chi(\xi/2)-\chi(\xi)$$
where $\chi$ stands for a  smooth function  with range in $[0,1],$ supported in  $]-4/3,4/3[$ and
such that $\chi\equiv1$ on $[-3/4,3/4].$ 
We further set 
$$\dot S_j\triangleq \chi(2^{-j}D) \quad\hbox{for all }\ j\in\Z$$
and define  $\mathcal{S}'_h$ to be the set of tempered distributions $z$  such 
that $\dot S_jz\to0$ uniformly when $j\to-\infty.$
\smallbreak
Following  \cite{HJR}, we introduce the  homogeneous Besov semi-norms:
$$\|z\|_{\dot\B^s_{p,r}}\triangleq \bigl\| 2^{js}\|\ddj z\|_{L^p(\R)}\bigr\|_{\ell^r(\Z)},$$
then  define the homogeneous Besov spaces $\dot\B^s_{p,r}$ (for any $s\in\R$ and $(p,r)\in[1,\infty]^2$)
 to be the subset of $z$ in  $\mathcal{S}'_h$ such that $\|z\|_{\dot\B^s_{p,r}}$ is finite. 
\smallbreak
To any element $z$ of $\mathcal{S}'_h,$ we  associate its low and 
high frequency parts with respect to some fixed  threshold  $J_0\in\Z,$ through 
 $$ z^{\ell}\triangleq \sum_{j\leq J_0}\dot{\Delta}_jz= \dot S_{J_0+1}z\andf  z^h\triangleq\sum_{j> J_0}\dot{\Delta}_jz=({\rm Id}-\dot S_{J_0+1})z.$$
In order to emphasize the dependency of the notation with respect to the threshold 
parameter $J_0,$ we use sometimes  the notation $z^{\ell,J_0}$ and $z^{h,J_0}.$ 
Likewise, we  set\footnote{For technical reasons, we need
a small overlap between low and high frequencies.}
if $r<\infty$ 
$$ \norme{z}^{\ell,J_0}_{\dot{\mathbb{B}}^{s}_{p,r}}\triangleq \biggl(\sum_{j\leq J_0}(2^{js}\norme{\dot{\Delta}_jz}_{L^p})^r\biggr)^{\frac{1}{r}} \andf
\norme{z}^{h,J_0}_{\dot{\mathbb{B}}^{s}_{p,r}}\triangleq \biggl(\sum_{j\geq J_0}(2^{js}\norme{\dot{\Delta}_jz}_{L^p})^r\biggr)^{\frac{1}{r}}\cdotp$$ 
Whenever the value of $J_0$ is clear from the context, we shall only write $\norme{z}^\ell_{\dot{\mathbb{B}}^{s}_{p,r}}$
\smallbreak
For any Banach space $X,$ index $\rho$ in $[1,\infty]$  and time $T\in[0,\infty],$ we use the notation 
$\|z\|_{L_T^\rho(X)}\triangleq  \bigl\| \|z\|_{X}\bigr\|_{L^\rho(0,T)}.$
If $T=+\infty$, then we  just  write $\|z\|_{L^\rho(X)}.$
In the case where $z$ has $n$ components $z_j$ in $X,$ we 
slightly abusively keep the notation  $\norme{z}$ to 
mean $\sum_{j\in\{1,\cdots,n\}} \norme{z_j}_X$. 
\smallbreak
Throughout the paper, $C>0$ designates a generic harmless constant, the 
value of which depends on the context. We use the notation 
$p'$ for the conjugate Lebesgue exponent of $p.$ Finally,  we denote by $(c_j)_{j\in\Z}$
nonnegative sequences such that  $\sum_{j\in\Z} c_j=1.$
\medbreak
We can now state our main global existence  result for $(TM_\lambda).$
\begin{Thm} \label{ThmExistLp}
Let  $2\leq p\leq4$. There exist $k=k(p)\in\Z$ and $c_0=c_0(p)>0$ such that for $J_\lambda\triangleq\left\lfloor\rm log_2\lambda\right\rfloor+k,$ if we assume that
$u_0^{\ell,J_\lambda},v_0^{\ell,J_\lambda}\in{\dot{\mathbb{B}}^{\frac{1}{p}}_{p,1}}$ and $u_0^{h,J_\lambda},v_0^{h,J_\lambda}\in{\dot{\mathbb{B}}^{\frac{3}{2}}_{2,1}}$ with
$$ \norme{(u_0,v_0)}^{\ell,J_\lambda}_{\dot{\mathbb{B}}^{\frac{1}{p}}_{p,1}}+ \lambda^{-1}\norme{(u_0,v_0)}^{h,J_\lambda}_{\dot{\mathbb{B}}^{\frac{3}{2}}_{2,1}} \leq c_0,
$$ then System $(TM_\lambda)$ admits a unique global solution $(u,v)$ in the space $E_p^{J_\lambda}$ defined by 
\begin{eqnarray*}
&&u^{\ell,J_\lambda}\in \mathcal{C}_b(\R^+;\dot{\mathbb{B}}^{\frac{1}{p}}_{p,1})\cap L^1(\mathbb{R}^+,\dot{\mathbb{B}}^{\frac{1}{p}+2}_{p,1}), \;\;\;\; u^{h,J_\lambda}\in \mathcal{C}_b(\R^+;\dot{\mathbb{B}}^{\frac{3}{2}}_{2,1})\cap L^1(\mathbb{R}^+,\dot{\mathbb{B}}^{\frac{3}{2}}_{2,1}), \\ &&
v^{\ell,J_\lambda}\in \mathcal{C}_b(\R^+;\dot{\mathbb{B}}^{\frac{1}{p}}_{p,1})\cap L^1(\mathbb{R}^+,\dot{\mathbb{B}}^{\frac{1}{p}+1}_{p,1}), \;\;\;\; v^{h,J_\lambda}\in \mathcal{C}_b(\R^+;\dot{\mathbb{B}}^{\frac{3}{2}}_{2,1})\cap L^1(\mathbb{R}^+,\dot{\mathbb{B}}^{\frac{3}{2}}_{2,1})  \\&& 
\lambda v+\d_xu \in  L^1(\mathbb{R}^+,\dot{\mathbb{B}}^{\frac{1}{p}}_{p,1}) \andf v\in L^2(\mathbb{R}^+,\dot{\mathbb{B}}^{\frac{1}{p}}_{p,1}).
\end{eqnarray*}
Moreover we have the following a priori estimate:
$$X_{p,\lambda}(t)\lesssim\norme{(u_0,v_0)}^{\ell,J_\lambda}_{\dot{\mathbb{B}}^{\frac{1}{p}}_{p,1}}+ \lambda^{-1}\norme{(u_0,v_0)}^{h,J_\lambda}_{\dot{\mathbb{B}}^{\frac{3}{2}}_{2,1}}
\quad\hbox{for all } t\geq 0,$$  where 
$$\displaylines{
X_{p,\lambda}(t)\triangleq\norme{(u,v)}^{\ell,J_\lambda}_{L^\infty_t(\dot{\mathbb{B}}^{\frac{1}{p}}_{p,1})}+\lambda^{-1}\norme{(u,v)}^{h,J_\lambda}_{L^\infty_t(\dot{\mathbb{B}}^{\frac{3}{2}}_{2,1})}
+\lambda^{-1}\norme{u}^{\ell,J_\lambda}_{L^1_t(\dot{\mathbb{B}}^{\frac{1}{p}+2}_{p,1})}+\norme{(u,v)}^{h,J_\lambda}_{L^1_t(\dot{\mathbb{B}}^{\frac{3}{2}}_{2,1})}
\hfill\cr\hfill
+\norme{\lambda v+\d_xu}_{L^1_t(\dot{\mathbb{B}}^{\frac{1}{p}}_{p,1})}+\lambda^{\frac12}\norme{v}_{L^2_t(\dot{\mathbb{B}}^{\frac{1}{p}}_{p,1})}.}$$
\end{Thm}

\begin{Rmq}  Somehow, the function $\lambda v+\d_xu$ may be seen as a damped mode of the system, 
which explains its better  time integrability.   This   is actually
 the key to closing  the estimates globally in time,  and this enables us to  prove similar results for more general systems (see Sections \ref{s:Euler} and \ref{s:general}).
\end{Rmq}
\begin{Rmq}
 Kawashima and Xu in \cite{XK1} obtained a result in  critical
 nonhomogeneous Besov spaces  built on $L^2$ 
 for a  class of system containing $(TM_\lambda).$
In their functional setting however, it seems difficult  to track  the exact dependency of the smallness condition and of 
the estimates with respect to the damping parameter $\lambda.$
Furthermore, whether  a $L^p$ approach may be performed for the whole classe of systems
that is considered therein, is unclear. 
\end{Rmq}
\begin{Rmq}  In the $L^2$ case, the method we here propose is robust enough to be adapted to higher dimension and to 
systems with more components,  see \cite{Bifluid} and \cite{CBD2}.
\end{Rmq} 
\begin{Rmq}  In  Section \ref{s:Euler} a statement similar to the above 
one is obtained  for  the isentropic compressible Euler system with a damping term in the velocity equation. To our knowledge, it is the first result 
(partially) in the $L^p$ setting for this system. Obtaining a similar result in higher dimension is  a work in progress.
\end{Rmq}
The above  theorem gives us for free some insight on 
the diffusive relaxation limit of $(TM_\lambda)$ 
in the case of fixed initial 
data\footnote{We actually expect our method
to be appropriate for investigating the connections 
between  the compressible Euler system and  the porous media equation,  in the spirit
of \cite{Junca,LC,XK1E}. This is a work in progress.}.
\begin{Cor} \label{CorLambda}Under the hypotheses of  Theorem \ref{ThmExistLp}, we 
 have $u\to u_0$ and $v\to 0$ when $\lambda$ goes to infinity. 
 More precisely, 
$$\norme{v}_{L^2(\dot{\mathbb{B}}^{\frac{1}{p}}_{p,1})}\leq Cc_0\lambda^{-1/2}\andf
\|u(t)-u_0\|_{\dot\B^0_{p,1}}\leq C c_0\biggl(\frac t\lambda\biggr)^{\frac1{2p}}\cdotp$$
\end{Cor}
\begin{proof} The first inequality follows from the estimate for $X_{p,\lambda}$
in Theorem \ref{ThmExistLp}. 
For the second inequality,  we observe that by interpolation in Besov
spaces and H\"older inequality, 
$$\|\d_xv\|_{L^r(\dot\B^0_{p,1})}\lesssim \|\d_xv\|_{L^1(\dot\B^{\frac1p}_{p,1})}^{1-\frac1p}
\|\d_xv\|_{L^2(\dot\B^{\frac1p-1}_{p,1})}^{\frac1p}\with\frac1r\triangleq 1-\frac1{2p}\cdotp$$
Since 
$$\|\d_xv\|_{L^1(\dot\B^{\frac1p}_{p,1})}\lesssim \|v\|_{L^1(\dot\B^{\frac1p+1}_{p,1})}^\ell+ \|v\|_{L^1(\dot\B^{\frac32}_{2,1})}^h,$$
 Theorem \ref{ThmExistLp} gives us
$$\|\d_xv\|_{L^r(\dot\B^0_{p,1})}\leq Cc_0\lambda^{-\frac1{2p}}\cdotp$$
Similarly,  we have
$$\|\d_xu\|_{L^{\wt r}(\dot\B^0_{p,1})}\lesssim \|\d_xu\|_{L^1(\dot\B^{\frac1p+1}_{p,1})}^{\frac12-\frac1{2p}}
\|\d_xu\|_{L^\infty(\dot\B^{\frac1p-1}_{p,1})}^{\frac12+\frac1{2p}}\with\frac1{\wt r}\triangleq\frac12-\frac1{2p}\cdotp$$
Hence, using that the product maps
$\dot\B^0_{p,1}\times\dot\B^{\frac1p}_{p,1}$ to $\dot\B^0_{p,1}$ and Theorem \ref{ThmExistLp}, 
we deduce that 
$$\|v\d_xu\|_{L^{r}(\dot\B^0_{p,1})} \lesssim \|v\|_{L^{2}(\dot\B^{\frac1p}_{p,1})}
\|\d_xu\|_{L^{\wt r}(\dot\B^0_{p,1})}
\leq Cc_0\lambda^{-\frac1{2p}}\cdotp$$
Since $\d_tu= -\d_xv-v\d_xu,$ we get the desired inequality 
for $u(t)-u_0,$ by time integration and H\"older inequality. 
\end{proof}

\medbreak 
Our second main result concerns the optimal decay estimates of the solution constructed in 
 the first  theorem. For now, we only consider the case $p=2.$
 \begin{Thm} \label{ThmDecayTM} Under the hypotheses of  Theorem \ref{ThmExistLp} with $p=2,$
 there exists a Lyapunov functional associated to the solution 
  $(u,v)$ constructed there, which is equivalent to $\|(u,v)\|_{\dot\B^{\frac12}_{2,1}\cap\dot\B^{\frac32}_{2,1}}.$ 
\medbreak
If,  additionally,  $(u_0,v_0)\in\dot{\mathbb{B}}^{-\sigma_1}_{2,\infty}$ for some
$\sigma_1\in\left(-\frac{1}{2},\frac{1}{2}\right]$ then, there exists 
a constant $C$ depending only on $\sigma_1$ and such that 
$$\norme{(u,v)(t)}_{\dot{\mathbb{B}}^{-\sigma_1}_{2,\infty}}\leq C\norme{(u_0,v_0)}_{\dot{\mathbb{B}}^{-\sigma_1}_{2,\infty}},\quad\forall t\geq 0. $$ 
Furthermore,    there exists a constant $\kappa_0$ depending only on $\sigma_1,$ $\lambda$ and on the norm of the 
data (and that may be taken equal to one in certain regimes, see the remark below) such that, if 
$$\langle  t \rangle\triangleq 1+\kappa_0 t,\quad \alpha_2\triangleq \sigma_1+\frac12
\andf C_{0,\lambda}\triangleq 
\lambda^{1+\alpha_2}\|(u_0,v_0)\|^{\ell,J_\lambda}_{\dot{\mathbb{B}}^{-\sigma_1}_{2,\infty}} + 
\|(u_0,v_0)\|^{h,J_\lambda}_{\dot{\mathbb{B}}^{\frac32}_{2,1}},$$
then we have the following decay estimates: 
$$
\begin{aligned}
\lambda^{\frac{3}{2}-\sigma}\norme{\langle  \lambda t \rangle^{\alpha}(u,v)(t)}^{\ell,J_\lambda}_{\dot{\mathbb{B}}^{\sigma}_{2,1}}&\leq C C_{0,\lambda},\quad  \sigma\in[-\sigma_1,1/2],\quad \alpha\triangleq\frac{\sigma+\sigma_1}2,\\
\norme{\langle\lambda t\rangle^{\alpha_2}(u,v)(t)}^{h,J_\lambda}_{\dot{\mathbb{B}}^{\frac{3}{2}}_{2,\infty}}
&\leq C C_{0,\lambda},\\
\lambda^{\sigma_1+\frac{3}{2}}\norme{\langle \lambda t\rangle^{\alpha_1}
v}^{\ell,J_\lambda}_{\dot{\mathbb{B}}^{-\sigma_1}_{2,\infty}}&\leq CC_{0,\lambda},\qquad
\alpha_1\triangleq\frac12\Bigl(\frac12+\sigma_1\Bigr)\cdotp\end{aligned}$$
\end{Thm}
\begin{Rmq}   Our proof reveals that\footnote{The exact value is $\kappa_0=
\biggl(\frac{\lambda\|(u_0,v_0)\|^{\ell,J_\lambda}_{\dot{\mathbb{B}}^{1/2}_{2,1}} + 
\|(u_0,v_0)\|^{h,J_\lambda}_{\dot{\mathbb{B}}^{3/2}_{2,1}} } {\lambda^{1+\alpha_2}\|(u_0,v_0)\|^{\ell,J_\lambda}_{\dot{\mathbb{B}}^{-\sigma_1}_{2,\infty}} + \|(u_0,v_0)\|^{h,J_\lambda}_{\dot{\mathbb{B}}^{3/2}_{2,1}}}\biggr)^{\frac{2}{\sigma_1+1/2}}\cdotp$}
 $\kappa_0\approx1$ whenever the first term of $C_{0,\lambda}$ is controlled by the second one
(which amounts to saying that the low frequencies of the data are dominated by the high frequencies).
\end{Rmq}
\begin{Rmq}
The fact that $v$ undergoes direct dissipation and not $u$ explains 
why  the decay of  the low frequencies of  $v$ is stronger than that of  $u.$
\end{Rmq}
\begin{Rmq}
 In light of  the embedding $L^1\hookrightarrow \dot{\mathbb{B}}^{-\frac{1}{2}}_{2,\infty},$ the above statement  with $\sigma_1=1/2$ encompasses
 the classical $L^1$ condition of  \cite{MatsumuraNishida}. 
 Actually, choosing suitable  exponents  allows to recover all the conditions used in \cite{BHN} for getting decay estimates.
\end{Rmq}


\section{The case  \texorpdfstring{$p=2$}{TEXT}}\label{s:L2}

The present section is dedicated to the case $p=2$ in Theorem \ref{ThmExistLp}
and to the proof  of Theorem \ref{ThmDecayTM}.
The reason for looking first at  $p=2$  is that one can 
exhibit a  Lyapunov functional for $(TM_\lambda)$ that allows
to treat the low and high frequencies of the solution together. 
Throughout this section, we focus on  the proof of a priori estimates  for smooth 
solutions to $(TM_\lambda),$ the reader
being referred to the  next section for the
rigorous proof of existence and  uniqueness, in  the general case. 
\smallbreak
Before starting, let us observe  that
$(u,v)$ is a solution to $(TM_\lambda)$ if and only if the couple $(\tilde u,\tilde v)$ defined
by 
\begin{equation}\label{eq:rescaling} 
({u},{v})(t,x)\triangleq(\tilde u,\tilde v)(\lambda t,\lambda x)
\end{equation} satisfies $(TM_1).$ 
Therefore, it suffices to establish Theorems \ref{ThmExistLp} and \ref{ThmDecayTM} 
for $\lambda=1,$  scaling back giving the desired inequalities, owing to the use of \emph{homogeneous}
Besov norms.

In the rest of this section, and in the following one, we shall use the short notation $(TM)$ to designate $(TM_1).$

\subsection{Global a  priori estimates for the linearized toy model}

Here we are concerned with the proof of a priori estimates for the following
linearization of $(TM)$:  
 $$ \left\{\begin{array}{lll}\d_tu+w\d_xu+\d_xv=0 \quad&\hbox{in}\quad& \R^+\times\R,\\
 \d_tv+w\d_xv+\d_xu +v=0 \quad&\hbox{in}\quad& \R^+\times\R,\\
(u,v)|_{t=0}=(u_0,v_0)  \quad&\hbox{on}\quad&\R,\end{array}\right.\leqno(LTM)
$$ 
where the given function $w:\R\times\R\to\R$ is smooth. 
\smallbreak
In the following computations, we assume that we are given a smooth solution $(u,v)$ 
of $(LTM)$ on $[0,T]\times\R,$ and  denote, for all $j\in\Z,$
\begin{equation}\label{eq:notationj}
u_j\triangleq \ddj u\andf v_j\triangleq \ddj v.
\end{equation}
Inspired by the work of the second author  in \cite{NSCL2,Handbook}, we consider  the following functional:
\begin{equation}\label{eq:Lj}\cL_j\triangleq \sqrt{ \|u_j\|_{L^2}^2 + \|v_j\|_{L^2}^2 + \|\d_xu_j\|_{L^2}^2 + \|\d_xv_j\|_{L^2}^2 + \int_\R v_j\,\d_xu_j}.\end{equation}
Applying operator $\ddj$ to $(LTM),$ simple computations lead to
$$\displaylines{\frac12\frac d{dt}\|(u_j,v_j)\|_{L^2}^2 + \|v_j\|_{L^2}^2 + \int_\R\Bigl( \ddj(w\d_xu)\,\ddj u + \ddj(w\d_x v)\,\ddj v\Bigr) =0, \cr
\frac12\frac d{dt}\|(\d_xu_j,\d_xv_j)\|_{L^2}^2 + \|\d_xv_j\|_{L^2}^2 + \int_\R\bigl(\ddj \d_{x}(w\d_xu))\,\ddj \d_xu +  \ddj(\d_x(w\d_x v))\,\ddj \d_xv\Bigr)  =0,} $$
$$\displaylines{\frac d{dt}\int_\R v_j\,\d_xu_j +\int_\R v_j\,\d_xu_j 
+ \|\d_xu_j\|_{L^2}^2 -  \|\d_xv_j\|_{L^2}^2+
\int_\R\Bigl( \ddj\d_{x}(w\d_xu)\ddj v+ \ddj(w\d_xv)\d_x\ddj u\Bigr)=0.}$$
Using the fact that 
\begin{equation}\label{eq:com}
\ddj(w\d_xz)=w\d_xz_j+[\ddj,w]\d_xz\quad\hbox{for }\ z=u,v\end{equation}
 and  integrating by parts, we see that
$$\int_\R \ddj(w\d_xz)\,\ddj z=-\frac12\int_\R \d_xw \,|z_j|^2+\int_\R [\ddj,w]\d_xz\, z_j.$$
Hence,  using  the classical commutator estimate \eqref{eq:com1} recalled
in the Appendix and the embedding $\dot\B^{\frac12}_{2,1}\hookrightarrow L^\infty,$
 we get   an absolute  constant $C>0$ such that for all $j\in\Z,$
$$\biggl|\int_\R\ddj(w\d_x z)\,\ddj z\biggr| \leq C c_j 2^{-\frac j2}\|\d_x w\|_{\dot \B^{\frac12}_{2,1}}\|z\|_{\dot \B^{\frac12}_{2,1}}\|\ddj z\|_{L^2}
\with\sum_{j\in\Z} c_j=1. $$
Likewise, we have, thanks to an  integration by parts,  
$$\int_\R \ddj \d_{x}(w\d_xz)\,\ddj \d_xz=  \frac12\int_\R \d_xw\,(\d_xz_j)^2
+\int_\R\d_x[\ddj,w]\d_xz\,\d_xz_j.$$
Hence, using \eqref{eq:com2} and  $\dot\B^{\frac12}_{2,1}\hookrightarrow L^\infty,$
$$\biggl|\int_\R \ddj \d_{x}(w\d_xz)\,\ddj \d_xz\biggr|\leq Cc_j 2^{-\frac j2} \|\d_x w\|_{\dot \B^{\frac12}_{2,1}}\|z\|_{\dot \B^{\frac12}_{2,1}}\|\d_xz_j\|_{L^2}.$$
Finally, integrating by parts reveals that 
$$\int_\R\Bigl( \ddj\d_{x}(w\d_xu)\ddj v+ \ddj(w\d_xv)\d_x\ddj u\Bigr)
=\int_\R [\ddj,w]\d_xv\,\d_xu_j-\int_\R[\ddj,w]\d_xu\,\d_xv_j.$$
Hence, using \eqref{eq:com1}, 
$$\displaylines{\quad
\biggl|\int_\R\Bigl( \ddj\d_{x}(w\d_xu)\ddj v+ \ddj(w\d_xv)\d_x\ddj u\Bigr)\biggr|\hfill\cr\hfill
\leq Cc_j2^{-\frac j2}\bigl(\|u\|_{\dot\B^{\frac12}_{2,1}}\|\d_xv_j\|_{L^2}
+\|v\|_{\dot\B^{\frac12}_{2,1}}\|\d_xu_j\|_{L^2}\bigr)\|\d_xw\|_{\dot\B^{\frac12}_{2,1}}\bigr)\cdotp}$$
In order to conclude the proof of estimates for $\cL_j,$ one can observe that there exist two 
absolute constants $C$ and $c$ such that
\begin{eqnarray}\label{eq:Xj}
&C^{-1}\|(u_j,v_j,\d_xu_j,\d_xv_j)\|_{L^2}^2\leq\cL^2_j\leq C \|(u_j,v_j,\d_xu_j,\d_xv_j)\|_{L^2}^2\\
&\andf
 \|v_j\|_{L^2}^2+\frac12\biggl(\|\d_xu_j\|_{L^2}^2+\|\d_xv_j\|_{L^2}^2+\int_\R v_j\,\d_xu_j\,dx\biggr)\geq c \min(1,2^{2j}) \cL_j^2.\nonumber\end{eqnarray}
Consequently, putting together with the above inequalities,  we obtain 
\begin{equation}\label{eq:Ljj}
\frac12\frac{d}{dt} \cL_j^2 +c\min(1,2^{2j})\cL_j^2\leq Cc_j 2^{-\frac j2} \|\d_x w\|_{\dot \B^{\frac12}_{2,1}}\|(u,v,\d_xu,\d_xv)\|_{\dot \B^{\frac12}_{2,1}}\cL_j.\end{equation}
Then,  integrating  on [0,t] for $t\in[0,T]$ and using Lemma \ref{SimpliCarre} yields
\begin{multline} \label{eq:toy1}
2^{\frac j2}\|(u_j,v_j,\d_xu_j,\d_xv_j)(t)\|_{L^2} + \min(1,2^{2j})2^{\frac j2} \int_0^t   \|(u_j,v_j,\d_xu_j,\d_xv_j)\|_{L^2}\\
\leq C\biggl(2^{\frac j2}\|(u_j,v_j,\d_xu_j,\d_xv_j)(0)\|_{L^2} + \int_0^t c_j   
 \|(u,v,\d_xu,\d_xv)\|_{\dot \B^{\frac12}_{2,1}}\|\d_xw\|_{\dot \B^{\frac12}_{2,1}}\biggr)\cdotp\end{multline}
Since  (direct) damping is present in the equation for $v,$ one can expect  $v$ to have 
 better decay and  time integrability properties than $u.$   In fact, as explained at the beginning of Section \ref{s:Lp}, it is even better to consider
 the function $z\triangleq v+\d_xu$ that satisfies:
$$\d_t z + z + w \d_xz=-\d^2_{xx}v-\d_xw\,\d_xu.$$
Now, applying Operator $\ddj$ to the above equation, then using the basic energy method gives: 
$$\frac12\frac d{dt}\|z_j\|_{L^2}^2+\|z_j\|_{L^2}^2=-\int_\R z_j\,\d^2_{xx}v_j -\int_\R\ddj(w\d_xz)\,z_j-\int_\R\ddj(\d_xw\,\d_xu)z_j.$$
The last term may be handled thanks to the decomposition \eqref{eq:com}, integration
by parts (as above) and Inequality \eqref{eq:com1} with $s=1/2.$ This gives
$$\frac12\frac d{dt}\|z_j\|_{L^2}^2+\|z_j\|_{L^2}^2\leq \|z_j\|_{L^2}\bigl(\|\d^2_{xx}v_j\|_{L^2}
+C2^{-\frac j2} c_j\|\d_xw\|_{\dot\B^{\frac12}_{2,1}}\|z\|_{\dot\B^{\frac12}_{2,1}}+ \|\ddj(\d_xw\,\d_xu)\|_{L^2}\bigr)\cdotp$$
After time integration (use Lemma \ref{SimpliCarre}), we end up  for all $t\in[0,T]$ with 
 $$\displaylines{2^{\frac j2} \|z_j(t)\|_{L^2}+ 2^{\frac j2}\int_0^t\|z_j\|_{L^2}\leq 2^{\frac j2}\|z_j(0)\|_{L^2} \hfill\cr\hfill+2^{\frac j2}\int_0^t\|\d^2_{xx}v_j\|_{L^2}+ C\int_0^tc_j\bigl(\|\d_xw\|_{\dot\B^{\frac12}_{2,1}}\|z\|_{\dot\B^{\frac12}_{2,1}}+\|\d_xw\,\d_xu\|_{\dot\B^{\frac12}_{2,1}}\bigr)\cdotp}$$
Hence, summing up on $j\leq0$ and using  the stability of the space $\dot\B^{\frac12}_{2,1}$ by product yields
\begin{equation}\label{eq:zzz}
\|z(t)\|^\ell_{\dot\B^{\frac12}_{2,1}}+\int_0^t\|z\|^\ell_{\dot\B^{\frac12}_{2,1}}\leq \|z_0\|^\ell_{\dot\B^{\frac12}_{2,1}}
+C\int_0^t\|v\|_{\dot\B^{\frac52}_{2,1}}^\ell
+C\int_0^t\|\d_x w\|_{\dot \B^{\frac12}_{2,1}} \bigl(\|z\|_{\dot\B^{\frac12}_{2,1}}+\|\d_xu\|_{\dot\B^{\frac12}_{2,1}}\bigr)\cdotp
\end{equation}
Let us sum  \eqref{eq:toy1} on $j\in\Z$,  
then add   \eqref{eq:zzz}   multiplied by a small enough constant. Using \eqref{eq:Xj} and 
denoting  $X(t)\triangleq  \|(u,v,\d_xu,\d_xv)(t)\|_{\dot \B^{\frac12}_{2,1}},$
we eventually get
\begin{equation}\label{eq:toy4}
X(t)+\int_0^t\bigl(\|(u,v)\|^\ell_{\dot \B^{\frac52}_{2,1}}
+\|(u,v)\|^h_{\dot \B^{\frac32}_{2,1}}+ \|v+\d_xu\|^\ell_{\dot\B^{\frac12}_{2,1}}\bigr) \leq C\biggl(X(0)+\int_0^t\|\d_x w\|_{\dot \B^{\frac12}_{2,1}}X\biggr)\cdotp
\end{equation}
Let us revert to our toy model, assuming that $w=v.$ Then, denoting by $Y(t)$ the left-hand side of \eqref{eq:toy4}, we get
$$
Y(t)\leq C\bigl(X(0)+Y^2(t)\bigr).$$
As $Y(0)=X(0),$ a continuity argument
ensures that   there exists  $c_0>0$ such that if 
\begin{equation}\label{eq:petit} 
X(0)\approx \|(u_0,v_0)\|_{\dot \B^{\frac12}_{2,1}\cap\dot \B^{\frac32}_{2,1}}\leq c_0,\end{equation}
then we have  $$Y(t)\leq 2 CX(0)\quad\hbox{for all }\ t\in[0,T].$$


\subsection{Proof of Theorem \ref{ThmDecayTM}}

For getting  decay estimates 
\emph{without any additional smallness condition}, the first step is to prove that 
the extra negative regularity for low frequencies is preserved
through the time evolution.  
 This is stated in the following lemma which is an adaptation to our setting of a result first 
 proved by J. Xu and Z. Xin 
 in \cite{XuXin} for the compressible Navier-Stokes system.   
\begin{Lemme} \label{LemmaDecayBf2}Let $\sigma_1\in]-\frac{3}{2},\frac{1}{2}]$. If, 
in addition to the hypotheses of Theorem \ref{ThmExistLp}, 
$\norme{(u_0,v_0)}_{\dot{\mathbb{B}}^{-\sigma_1}_{2,\infty}}$ is bounded then, for 
all $t\geq0,$ we have
$$\norme{(u,v)(t)}_{\dot{\mathbb{B}}^{-\sigma_1}_{2,\infty}}\leq C\norme{(u_0,v_0)}_{\dot{\mathbb{B}}^{-\sigma_1}_{2,\infty}}.$$
\end{Lemme}
\begin{proof}
Applying $\ddj$ to $(TM)$ yields
$$
\left\{\begin{array}{l}
\d_tu_j+\d_xv_j=-v\d_xu_j+[v,\ddj]\d_xu,\\
\d_tv_j+\d_xu_j+v_j=-v\d_xv_j+[v,\ddj]\d_xv.
\end{array}\right.
$$
Hence, an energy method, followed by time integration (use Lemma \ref{SimpliCarre}) gives
$$\displaylines{\|(u_j,v_j)(t)\|_{L^2} + \int_0^t\|v_j\|_{L^2} \leq \|(u_j,v_j)(0)\|_{L^2}
\hfill\cr\hfill+\frac12\int_0^t\|\d_xv\|_{L^\infty}\|(u_j,v_j)\|_{L^2} +\int_0^t\|[v,\ddj]\d_xu\|_{L^2}
+\int_0^t\|[v,\ddj]\d_xv\|_{L^2}.}$$
Omitting  the second term in the left-hand side, and using   the commutator
estimate \eqref{eq:com3}  that is valid provided $-1/2\leq-\sigma_1<3/2,$  we get
$$\|(u,v)(t)\|_{\dot\B^{-\sigma_1}_{2,\infty}}\leq \|(u_0,v_0)\|_{\dot\B^{-\sigma_1}_{2,\infty}}
+ C \int_0^t\|\d_xv\|_{\dot\B^{\frac12}_{2,1}} \|(u,v)\|_{\dot\B^{-\sigma_1}_{2,\infty}}.$$
Hence, by Gronwall lemma, 
$$\|(u,v)(t)\|_{\dot\B^{-\sigma_1}_{2,\infty}}\leq \|(u_0,v_0)\|_{\dot\B^{-\sigma_1}_{2,\infty}}
\exp\biggl(C\int_0^t\|\d_xv\|_{\dot\B^{\frac12}_{2,1}}\biggr)\cdotp
$$
Since the term in the exponential is small (as $X(0)$ is small), we get the lemma. 
\end{proof}

The second ingredient is that one can work out  from the computations we did in the previous paragraph,
  a  \emph{Lyapunov functional}
that is equivalent to the norm of $(u,v)$ in  $\dot\B^{\frac12}_{2,1}\cap\dot\B^{\frac32}_{2,1}.$
To proceed,  observe that, on the one hand,  Inequality \eqref{eq:Ljj} implies that for all $t\geq0,$
$$\displaylines{\quad
\cL(t)+c\int_0^t\cH\leq \cL(0) + C\int_0^t\|\d_xv\|_{\dot\B^{\frac12}_{2,1}}\cL
\hfill\cr\hfill\with \cL\triangleq \sum_{j\in\Z} 2^{\frac j2}\cL_j\andf
\cH\triangleq \sum_{j\in\Z}2^{\frac j2}\min(1,2^{2j})\cL_j\quad}$$
and that, on the other hand,  \eqref{eq:zzz} gives us 
$$\|(v+\d_xu)(t)\|^\ell_{\dot\B^{\frac12}_{2,1}}+\int_0^t\|v+\d_xu\|^\ell_{\dot\B^{\frac12}_{2,1}}\leq \|v_0+\d_xu_0\|^\ell_{\dot\B^{\frac12}_{2,1}}
+C\int_0^t\|v\|_{\dot\B^{\frac52}_{2,1}}^\ell
+C\int_0^t\|\d_x v\|_{\dot \B^{\frac12}_{2,1}}\cL.$$
Hence, there exist  $\eta>0$ and $c'>0$ such that, denoting $\wt\cL\triangleq 
\cL+\eta \|v+\d_xu\|^\ell_{\dot\B^{\frac12}_{2,1}},$ we have 
$$
\wt\cL(t)+c'\int_0^t\wt \cH\leq\wt\cL(0)+ C\int_0^t\|\d_xv\|_{\dot\B^{\frac12}_{2,1}}\cL
\with \wt \cH\triangleq \cH + \eta \|v+\d_xu\|^\ell_{\dot\B^{\frac12}_{2,1}}.$$
Observe that $\wt\cH\gtrsim\|\d_xv\|_{\dot\B^{\frac12}_{2,1}}.$  Since the previous step ensures that $\cL\lesssim X(0),$
one can conclude that the last term of the above inequality may be absorbed by the second term of the left-hand side 
provided $X(0)$ is small enough. So finally,  taking $c'$ smaller if need be,  we discover that
$$\wt\cL(t)+c'\int_0^t\wt \cH\leq\wt\cL(0).$$
Clearly,  one can start the proof from  any time $t_0\geq0$ and get in a similar way:
$$\wt\cL(t_0+h)+c'\int_{t_0}^{t_0+h} \wt\cH\leq\wt\cL(t_0),\quad h\geq0.$$
This of course ensures that $\wt\cL$ is nonincreasing on $\R^+$ (hence differentiable almost
everywhere) and that  for all $t_0\geq0$ and $h>0,$ 
$$
\frac{\wt\cL(t_0+h)-\wt\cL(t_0)}h+c'\frac1h\int_{t_0}^{t_0+h}\wt\cH\leq0.
$$
Consequently, passing to the limit $h\to0$ gives
\begin{equation}\label{eq:lyapunov}
\frac d{dt}\wt\cL+c'\wt\cH \leq 0\quad\hbox{a. e.  on }\ \R^+.\end{equation}
We thus come to the conclusion that:
\begin{Lemme} \label{LemmeDecay2}There exist two functionals $\wt\cL$ and $\wt\cD$  satisfying
$$ \wt\cL\simeq   \|(u,v)\|_{\dot\B^{\frac12}_{2,1}\cap\dot\B^{\frac32}_{2,1}}\andf 
\wt\cD\simeq \|u\|^\ell_{\dot\B^{\frac52}_{2,1}}+\|u\|^h_{\dot\B^{\frac32}_{2,1}}+\|v\|_{\dot\B^{\frac32}_{2,1}},$$
and such that if $\|(u_0,v_0)\|_{\dot\B^{\frac12}_{2,1}\cap\dot\B^{\frac32}_{2,1}}$ is small enough then \eqref{eq:lyapunov} is satisfied.
\end{Lemme}
One can now tackle the proof of decay estimates. 
Let us denote 
$$C_0\triangleq \|(u_0,v_0)\|^\ell_{\dot\B^{-\sigma_1}_{2,\infty}}+\|(u_0,v_0)\|^h_{\dot\B^{\frac32}_{2,1}}.$$
As a first,  observe that interpolation for homogeneous Besov norms  gives us: 
 $$ \norme{(u,v)}_{\dot{\mathbb{B}}^{\frac{1}{2}}_{2,1}}^\ell \lesssim \left(\norme{(u,v)}^\ell_{\dot{\mathbb{B}}^{-\sigma_1}_{2,\infty}}\right)^{\theta_0}\left(\norme{(u,v)}^\ell_{\dot{\mathbb{B}}^{\frac{5}{2}}_{2,1}}\right)^{1-\theta_0} \with\theta_0\triangleq\frac{2}{5/2+\sigma_1}\cdotp$$
 Therefore, owing  to   Lemma \ref{LemmaDecayBf2}, there exists $c>0$ such that 
   $$ \norme{(u,v)}_{\dot{\mathbb{B}}^{\frac{5}{2}}_{2,1}}^\ell\geq c\,   C_0^{-\frac{\theta_0}{1-\theta_0}}\bigl(\norme{(u,v)}^\ell_{\dot{\mathbb{B}}^{\frac{1}{2}}_{2,1}}\bigr)^{\frac{1}{1-\theta_0}}.$$
   Note that our definition of $C_0$ and the estimates we proved for $(u,v)$ in the previous paragraph also ensure that 
   $$
   \|(u,v)\|_{\dot\B^{\frac32}_{2,1}}^h\gtrsim C_0^{-\frac{\theta_0}{1-\theta_0}}    \bigl(\|(u,v)\|_{\dot\B^{\frac32}_{2,1}}^h\bigr)^{\frac1{1-\theta_0}}\cdotp
   $$
   Hence, thanks to the above lemma, we have, 
$$\frac{d}{dt} \wt\cL +c C_0^{-\frac{\theta_0}{1-\theta_0}} \wt\cL^{\frac{1}{1-\theta_0}}\leq 0\with \wt\cL\simeq \|(u,v)\|_{\dot\B^{\frac12}_{2,1}\cap\dot\B^{\frac32}_{2,1}}.$$
Integrating, this gives us
$$\wt\cL(t)\leq \bigl(1+\kappa_0 \,t\bigr)^{1-\frac1{\theta_0}}\wt\cL(0)
\with \kappa_0\triangleq  c\frac{\theta_0}{1-\theta_0}\biggl(\frac{\wt\cL(0)}{C_0}\biggr)^{\frac{\theta_0}{1-\theta_0}}.$$
Rewriting $\theta_0$ in terms of $\sigma_1$ and using that
 $\|(u_0,v_0)\|_{\dot\B^{\frac12}_{2,1}\cap\dot\B^{\frac32}_{2,1}}\lesssim C_0,$ 
 we get
\begin{equation}\label{eq:sigma1}\|(u,v)(t)\|_{\dot\B^{\frac12}_{2,1}\cap\dot\B^{\frac32}_{2,1}}\leq C
 (1+ t)^{-\alpha_1}\|(u_0,v_0)\|_{\dot\B^{\frac12}_{2,1}\cap\dot\B^{\frac32}_{2,1}}\with
 \alpha_1\triangleq \frac12\Bigl(\sigma_1+\frac12\Bigr)\cdotp\end{equation}
In order to get the decay rate in $\dot{\mathbb{B}}^{\sigma}_{2,1}$ for all $\sigma\in\left(-\sigma_1,1/2\right)$, we just need the interpolation inequality 
$$\norme{(u,v)}_{\dot{\mathbb{B}}^{\sigma}_{2,1}}\leq\norme{(u,v)}_{\dot{\mathbb{B}}^{-\sigma_1}_{2,\infty}}^{\theta_1}
\norme{(u,v)}_{\dot{\mathbb{B}}^{\frac{1}{2}}_{2,1}}^{1-\theta_1} \with \theta_1\triangleq \frac{1/2-\sigma}{1/2+\sigma_1}\in(0,1).$$
In the end, we get (since $\kappa_0\lesssim1$):
 $$\norme{(u,v)(t)}_{\dot{\mathbb{B}}^{\sigma}_{2,1}}\leq C(1+\kappa_0 t)^{-\frac{\sigma+\sigma_1}{2}}
\|(u_0,v_0)\|_{\dot\B^{-\sigma_1}_{2,\infty}}^{\theta_1}\|(u_0,v_0)\|_{\dot\B^{\frac12}_{2,1}\cap\dot\B^{\frac32}_{2,1}}^{1-\theta_1}.$$
In order to  improve  the decay for the damped component $v,$ let us start from 
$$\partial_t v + v = -\frac{1}{2}\partial_x(v^2) -\partial_xu.$$
As  $v_0^\ell$ is in $\dot\B^{-\sigma_1}_{2,\infty}$ for some $\sigma_1\in]-\frac12,\frac12],$ we get
\begin{equation}\label{eq:new}\norme{v(t)}^\ell_{\dot{\mathbb{B}}^{-\sigma_1}_{2,\infty}}\leq e^{-t}\norme{v_0}^\ell_{\dot{\mathbb{B}}^{-\sigma_1}_{2,\infty}}+\int_0^te^{-(t-\tau)}\norme{(\partial_xv^2,\partial_xu)}^\ell_{\dot{\mathbb{B}}^{-\sigma_1}_{2,\infty}}\,d\tau.\end{equation}
It is important to  observe that, as $1-\sigma_1\geq 1/2,$
\begin{equation}\label{eq:dxz}\norme{\d_xz}^\ell_{\dot{\mathbb{B}}^{-\sigma_1}_{2,\infty}}\lesssim \|z\|^\ell_{\dot\B^{\frac12}_{2,1}}.\end{equation}
Hence, 
multiplying \eqref{eq:new}  by $\langle t \rangle ^{\alpha_1}$  and using the product
laws in Besov spaces recalled in Proposition \ref{LP} yields:
$$\norme{\langle t \rangle ^{\alpha_1}v(t)}^\ell_{\dot{\mathbb{B}}^{-\sigma_1}_{2,\infty}}\leq \norme{v_0}^\ell_{\dot{\mathbb{B}}^{-\sigma_1}_{2,\infty}}
+\int_0^t\langle t \rangle ^{\alpha_1}e^{-(t-\tau)}\norme{u}^\ell_{\dot{\mathbb{B}}^{1-\sigma_1}_{2,\infty}}\,d\tau
+\int_0^t\langle t \rangle ^{\alpha_1}e^{-(t-\tau)}\norme{v}_{\dot{\mathbb{B}}^{\frac12}_{2,1}}^2\,d\tau,$$
and one can conclude as above that 
$$\langle t\rangle^{\alpha_1}\|v(t)\|_{\dot\B^{-\sigma_1}_{2,\infty}}\lesssim  \|(u_0,v_0)\|_{\dot\B^{-\sigma_1}_{2,\infty}}+ \|(u_0,v_0)\|_{\dot\B^{\frac12}_{2,1}\cap\dot\B^{\frac32}_{2,1}}.$$
Let us finally  exhibit the (optimal) decay rate of the high frequencies for the norm in $\dot{\mathbb{B}}^{\frac{3}{2}}_{2,1}$.
Recall that for $j\geq0,$ we have 
$$
\frac{d}{dt}\mathcal{L}^2_j+c\cL_j^2 \lesssim\cL_j^2\norme{\partial_xv}_{L^\infty}+c_j2^{-\frac{j}{2}} X\norme{\partial_xv}_{\mathbb{\dot{B}}^{\frac{1}{2}}_{2,1}}\cL_j.$$
Hence, using Lemma~\ref{SimpliCarre},  multiplying by $2^{\frac j2},$ summing up on $j\geq0$ and remembering that 
$$\sum_{j\geq0}2^{\frac j2}\cL_j\simeq \|(u,v)\|^h_{\dot\B^{\frac32}_{2,1}},$$  we get
\begin{equation}\label{eq:decayhf}
\norme{(u,v)(t)}^h_{\dot{\mathbb{B}}^{\frac{3}{2}}_{2,1}}\lesssim e^{-ct}\norme{(u_0,v_0)}^h_{\dot{\mathbb{B}}^{\frac{3}{2}}_{2,1}}+\int_0^te^{-c(t-\tau)}\norme{v}_{\dot{\mathbb{B}}^{\frac{3}{2}}_{2,1}}\norme{(u,v)}_{\dot{\mathbb{B}}^{\frac{3}{2}}_{2,1}}.
\end{equation}
Multiplying both sides by $\langle t \rangle^{2\alpha_1},$ we get
$$\displaylines{\norme{\langle t\rangle^{2\alpha_1}(u,v)(t)}^h_{\dot{\mathbb{B}}^{\frac{3}{2}}_{2,1}}\lesssim \langle t \rangle^{2\alpha_1}e^{-ct}\norme{(u_0,v_0)}^h_{\dot{\mathbb{B}}^{\frac{3}{2}}_{2,1}}
\hfill\cr\hfill+\int_0^t\biggl(\frac{\langle t \rangle}{\langle\tau\rangle}\biggr)^{2\alpha_1} e^{-c(t-\tau)} \bigl(\langle\tau\rangle^{\alpha_1}\norme{v}_{\dot{\mathbb{B}}^{\frac{3}{2}}_{2,1}}\bigr)
 \bigl(\langle\tau\rangle^{\alpha_1}\norme{(u,v)}_{\dot{\mathbb{B}}^{\frac{3}{2}}_{2,1}}\bigr).}$$
Taking advantage of  \eqref{eq:sigma1} for bounding the norms in the time integral, one ends up with 
the desired decay estimate for $\norme{(u,v)(t)}^h_{\dot{\mathbb{B}}^{\frac{3}{2}}_{2,1}}.$
This completes the proof of Theorem \ref{ThmDecayTM}.
\begin{Rmq} In the same way, making the slightly stronger assumption that $v_0^\ell\in\dot\B^{-\frac12}_{2,1},$ we get $$\langle t\rangle^{\alpha_1}\|v(t)\|_{\dot\B^{-\frac12}_{2,1}}\lesssim   \|v_0\|_{\dot\B^{-\frac12}_{2,1}}+ \|(u_0,v_0)\|_{\dot\B^{\frac12}_{2,1}\cap\dot\B^{\frac32}_{2,1}}.$$
\end{Rmq}


\section{Proof of Theorem \ref{ThmExistLp}}\label{s:Lp}

An explicit computation in the Fourier space of the solution 
to  $(LTM)$  with $w=0$ reveals that:
\begin{itemize}
\item In low frequencies, the matrix of the system corresponding to frequency $\xi$
has two real eigenvalues that tend to be equal to $1$ and to $\xi^2,$ for $\xi$ going to $0$;
\item In high frequencies, two complex conjugated eigenvalues coexist, 
that are, asymptotically, equal to $\frac12(\xi^2\pm i\xi).$ 
\end{itemize}
Consequently, one can expect that the low frequency part of System $(TM)$ 
is solvable in some $L^p$ type functional framework  with, possibly,  $p\not=2,$ whereas
going beyond the $L^2$ framework in high frequency is bound to fail. 
A similar dichotomy has been observed for the compressible Navier-Stokes 
equations (see in particular \cite{NSCLP,CMZ,Boris}) but  the behavior 
of the low and high frequencies in our situation is  exchanged. 

In order to extend the results of the previous section to the $L^p$ framework 
for low frequencies, we shall adapt  \cite{Boris} to our setting,  introducing an `effective velocity' that  reads $z=v+\partial_xu$
and may be seen as  an approximate  dissipative eigenmode of the system, in the low frequency regime. 
 
\medbreak

The bulk of the proof  consists in establishing estimates 
in the functional framework of Theorem  \ref{ThmExistLp} for $(LTM).$ This will be carried out in the first two subsections of this part. 
Then, we will prove the existence part of the theorem and, finally, 
the uniqueness of a solution.

\subsection{Low frequencies estimates in \texorpdfstring{$L^p$}{TEXT}}

The main result of this section reads as follows. 
\begin{Prop} \label{APLP}
Let $(u,v)$ be a smooth solution of  $(LTM)$  on $[0,T].$ Then,  for all $1\leq p<\infty,$ we have 
\begin{multline}
\norme{(u,v)(t)}_{\dot{\mathbb{B}}^{\frac{1}{p}}_{p,1}}^\ell
+\int_0^t\norme{u}_{\dot{\mathbb{B}}^{\frac{1}{p}+2}_{p,1}}^\ell+\int_0^t\norme{v+\d_xu}_{\dot{\mathbb{B}}^{\frac{1}{p}}_{p,1}}^\ell \\
\leq C\biggl(\norme{(u_0,v_0)}^\ell_{\dot{\mathbb{B}}^{\frac{1}{p}}_{p,1}}
+\int_0^t\norme{(u,v,\d_xu)}_{\dot{\mathbb{B}}^{\frac{1}{p}}_{p,1}}
\norme{w}_{\dot{\mathbb{B}}^{\frac{1}{p}+1}_{p,1}}\biggr)\cdotp\label{EQLP}  \end{multline}
\end{Prop}
\begin{proof} 
Let us set   $z\triangleq v+\partial_xu.$  We observe that the couple $(u,z)$  satisfies  
\begin{equation} \left\{ \begin{aligned}
&\partial_tu-\partial_{xx}^2u+w\partial_xu=-\partial_xz,\\ 
&\partial_tz +z +w\partial_xz=\partial_{xxx}^3u-\partial_{xx}^2z-\partial_{x}w\,\partial_{x}u.
 \end{aligned} \right. \label{TMLP}
\end{equation}
In low frequencies, we expect  the linear terms of the right-hand side  to be negligible,
so that we will look at  the first equation as a heat equation with a convection term,
and at the second one as  a damped transport equation.
\smallbreak
Now, applying $\dot{\Delta}_j$ to the first equation of \eqref{TMLP} yields
$$\begin{aligned}
\partial_tu_j-\partial_{xx}^2u_j&=-\dot{\Delta}_j(w\partial_xu)-\partial_xz_j \\ 
&=-w\partial_xu_j-\partial_xz_j+[w,\dot{\Delta}_j]\partial_xu.
\end{aligned}$$
Multiplying
 by $|u_j|^{p-2}u_j$ and integrating in space,  we get
 $$\displaylines{\frac{1}{p}\frac{d}{dt}\norme{u_j}_{L^p}^p-\int_\mathbb{R}\partial_{xx}^2u_j|u_j|^{p-2}u_j\hfill\cr\hfill=-\int_\mathbb{R}\partial_xz_j|u_j|^{p-2}u_j -\int_{\mathbb{R}}w\partial_xu_j\,|u_j|^{p-2}u_j
 +\int_\R[w,\dot{\Delta}_j]\partial_xu\,|u_j|^{p-2}u_j.}$$
 Hence,  integrating by parts, using Cauchy-Schwarz inequality and
 Proposition \ref{p:bernstein} gives
 $$\frac{1}{p}\frac{d}{dt}\norme{u_j}_{L^p}^p+c_p2^{2j}\norme{u_j}_{L^p}^p \leq\frac1p\norme{\partial_xw}_{L^\infty}\norme{u_j}^p_{L^p}+\left(\norme{\partial_xz_j}_{L^p}+\norme{ [w,\dot{\Delta}_j]\d_xu}_{L^p}\right)\norme{u_j}_{L^p}^{p-1}.$$
 Multiplying by $2^{\frac{j}{p}}$, summing up on $j\leq J_0$ and using Lemma \ref{SimpliCarre}, we obtain
$$\displaylines{
\norme{u(t)}_{\dot{\mathbb{B}}^{\frac{1}{p}}_{p,1}}^\ell+c_p\int_0^t\norme{u}_{\dot{\mathbb{B}}^{\frac{1}{p}+2}_{p,1}}^\ell\hfill\cr\hfill\leq  \norme{u_0}_{\dot{\mathbb{B}}^{\frac{1}{p}}_{p,1}}^\ell+\int_0^t\norme{z}_{\dot{\mathbb{B}}^{\frac{1}{p}+1}_{p,1}}^\ell 
+\frac1p\int_0^t\norme{\partial_xw}_{L^\infty}\norme{u}_{\dot{\mathbb{B}}^{\frac{1}{p}}_{p,1}}^\ell+\sum_{j\leq J_0}2^{\frac{j}{p}}\int_0^t\norme{[w,\dot{\Delta}_j]\partial_xu}_{L^p}.}$$
 The commutator term may be bounded according to  Inequality \eqref{eq:com1} with $s=1/p.$ 
 Hence, remembering that  $\dot{\mathbb{B}}^{\frac{1}{p}}_{p,1}\hookrightarrow L^\infty$, we end up with
\begin{equation}
\norme{u(t)}_{\dot{\mathbb{B}}^{\frac{1}{p}}_{p,1}}^\ell+c_p\int_0^t\norme{u}_{\dot{\mathbb{B}}^{\frac{1}{p}+2}_{p,1}}^\ell\leq \norme{u_0}_{\dot{\mathbb{B}}^{\frac{1}{p}}_{p,1}}^\ell+\int_0^t\norme{z}_{\dot{\mathbb{B}}^{\frac{1}{p}+1}_{p,1}}^\ell +C\int_0^t\norme{w}_{\dot{\mathbb{B}}^{\frac{1}{p}+1}_{p,1}}\norme{u}_{\dot{\mathbb{B}}^{\frac{1}{p}}_{p,1}}. \label{EQU}
\end{equation}
Let us next look at the second equation of  \eqref{TMLP}. 
We have for all $j\in\Z,$  
\begin{equation*}
\partial_tz_j +z_j +w\partial_xz_j=\partial_{xxx}^3u_j-\partial_{xx}^2z_j-\dot{\Delta}_j(\partial_xw\,\partial_xu)+[w,\ddj]\partial_{x}z.
\end{equation*}
Multiplying by $z_j|z_j|^{p-2}$ and adapting what we did for the for the first equation of $(LTM),$  we obtain
$$\displaylines{
\norme{z(t)}_{\dot{\mathbb{B}}^{\frac{1}{p}}_{p,1}}^\ell+\int_0^t\norme{z}_{\dot{\mathbb{B}}^{\frac{1}{p}}_{p,1}}^\ell\leq  \norme{z_0}_{\dot{\mathbb{B}}^{\frac{1}{p}}_{p,1}}^\ell+\int_0^t\norme{z}_{\dot{\mathbb{B}}^{\frac{1}{p}+2}_{p,1}}^\ell+\int_0^t\norme{u}_{\dot{\mathbb{B}}^{\frac{1}{p}+3}_{p,1}}^\ell + \frac1p\int_0^t\norme{\partial_xw}_{L^\infty}\norme{z}^\ell_{\dot{\mathbb{B}}^{\frac{1}{p}}_{p,1}}
\hfill\cr\hfill +\int_0^t\sum_{j\leq J_0}2^{\frac{j}{p}}\norme{[w,\ddj]\d_xz}_{L^p}+\int_0^t\norme{\partial_xw\,\partial_xu}^\ell_{\dot{\mathbb{B}}^{\frac{1}{p}}_{p,1}}.}$$
Combining Proposition \ref{C1}, the commutator estimate \eqref{eq:com1}, the embedding $\dot{\mathbb{B}}^{\frac{1}{p}}_{p,1}\hookrightarrow L^\infty$ and  Proposition \ref{LP}, we discover that
\begin{multline}
\norme{z(t)}_{\dot{\mathbb{B}}^{\frac{1}{p}}_{p,1}}^\ell+\int_0^t\norme{z}_{\dot{\mathbb{B}}^{\frac{1}{p}}_{p,1}}^\ell\leq \norme{z_0}_{\dot{\mathbb{B}}^{\frac{1}{p}}_{p,1}}^\ell
+\int_0^t\norme{z}_{\dot{\mathbb{B}}^{\frac{1}{p}+2}_{p,1}}^\ell+\int_0^t\norme{u}_{\dot{\mathbb{B}}^{\frac{1}{p}+3}_{p,1}}^\ell\\+ \int_0^t \norme{w}_{\dot{\mathbb{B}}^{\frac{1}{p}+1}_{p,1}}\norme{z}_{\dot{\mathbb{B}}^{\frac{1}{p}}_{p,1}}
+\int_0^t\norme{\partial_xw}_{\dot{\mathbb{B}}^{\frac{1}{p}}_{p,1}}\norme{\partial_xu}_{\dot{\mathbb{B}}^{\frac{1}{p}}_{p,1}}. \label{EQZ}
\end{multline}
At this stage, the key observation is that, owing to Bernstein inequality, there
exists an absolute constant $C$ such that for any  couple $(\sigma,\sigma')\in\R^2$ with $\sigma\leq\sigma',$  we have
\begin{equation}\label{eq:J0}
\|f\|_{\dot\B^{\sigma'}_{p,1}}^\ell \leq C2^{J_0(\sigma'-\sigma)} \|f\|_{\dot\B^{\sigma}_{p,1}}^\ell. 
\end{equation}
Consequently, if $J_0$ is chosen small enough, then after 
adding up \eqref{EQU} and \eqref{EQZ}, we just get
$$\displaylines{\norme{(u,z)(t)}_{\dot{\mathbb{B}}^{\frac{1}{p}}_{p,1}}^\ell+
\int_0^t\bigl(\norme{u}_{\dot{\mathbb{B}}^{\frac{1}{p}+2}_{p,1}}^\ell+\norme{z}_{\dot{\mathbb{B}}^{\frac{1}{p}}_{p,1}}^\ell\bigr) \lesssim \norme{(u_0,z_0)}_{\dot{\mathbb{B}}^{\frac{1}{p}}_{p,1}}^\ell
\hfill\cr\hfill+\int_0^t  \norme{w}_{\dot{\mathbb{B}}^{\frac{1}{p}+1}_{p,1}}\bigl(
 \norme{u}_{\dot{\mathbb{B}}^{\frac{1}{p}+1}_{p,1}}+ \norme{u}_{\dot{\mathbb{B}}^{\frac{1}{p}}_{p,1}}
 + \norme{z}_{\dot{\mathbb{B}}^{\frac{1}{p}}_{p,1}}\bigr)\cdotp}$$
 Because
 $$
  \norme{z}_{\dot{\mathbb{B}}^{\frac{1}{p}}_{p,1}}\lesssim 
   \norme{u}_{\dot{\mathbb{B}}^{\frac{1}{p}+1}_{p,1}}+
  \norme{v}_{\dot{\mathbb{B}}^{\frac{1}{p}}_{p,1}}\andf
    \norme{v}_{\dot{\mathbb{B}}^{\frac{1}{p}}_{p,1}}^\ell\lesssim 
      \norme{z}_{\dot{\mathbb{B}}^{\frac{1}{p}}_{p,1}}^\ell+ 
   \norme{u}_{\dot{\mathbb{B}}^{\frac{1}{p}}_{p,1}}^\ell,
$$
we conclude to the desired inequality. 
\end{proof}


\subsection{High frequencies estimates in \texorpdfstring{$L^2$}{TEXT}}

Our second task is to bound the high frequencies of the solution of $(LTM).$
Although the functional framework for high frequencies  is the same as before, one cannot 
repeat exactly the computations therein since the terms $(w\d_xu)^h$ and $(w\d_xv)^h$ contain 
 a little amount of low frequencies of $w,$ $u$ and $v,$ that are only in spaces  of the type $\dot{\mathbb{B}}^{s}_{p,1}$ with $p>2$ (and thus not in some $\dot{\mathbb{B}}^{s'}_{2,1}$).
 To overcome the difficulty, we have to study more  carefully the  commutators
 coming into play in the proof  (see Lemma \ref{CP}). 
  \begin{Prop} \label{PropHfLp}
Let $(u,v)$ be of solution of $(LTM)$ with $u_0^\ell,v_0^\ell\in{\dot{\mathbb{B}}^{\frac{1}{p}}_{p,1}}$ and $u_0^h,v_0^h\in{\dot{\mathbb{B}}^{\frac{3}{2}}_{2,1}}$
for some $2\leq p\leq 4.$ Define $p^*$ by the relation $1/p+1/p^*=1/2.$ 
Then, the following a priori estimate holds for some constant $C$ depending only on $J_0$:
$$\displaylines{
\norme{(u,v)(t)}^h_{\dot{\mathbb{B}}^{\frac{3}{2}}_{2,1}}+\int_0^t\norme{(u,v)}^h_{\dot{\mathbb{B}}^{\frac{3}{2}}_{2,1}}\lesssim\norme{(u_0,v_0)}^h_{\dot{\mathbb{B}}^{\frac{3}{2}}_{2,1}} + \int_0^t\bigl(\norme{\partial_xw}_{L^\infty}\norme{(u,v)}^h_{\dot{\mathbb{B}}^{\frac{3}{2}}_{2,1}}\hfill\cr\hfill
+\norme{w}^\ell_{\dot{\mathbb{B}}^{1+\frac{1}{p^*}}_{p^*,1}}\norme{(\partial_xu,\partial_xv)}_{\dot{\mathbb{B}}^{\frac{1}{p}-1}_{p,1}}+\norme{(\partial_xu,\partial_xv)}_{L^\infty}\norme{w}^h_{\dot{\mathbb{B}}^{\frac{3}{2}}_{2,1}}+\norme{(\partial_xu,\partial_xv)}^\ell_{\dot{\mathbb{B}}^{1-\frac{1}{p}}_{p,1}}\norme{\d_xw}^\ell_{\dot{\mathbb{B}}^{-\frac{1}{p^*}}_{p^*,1}}\bigr)\cdotp}$$
\end{Prop}
\begin{proof}
We localize  System $(LTM)$ by means of $\ddj,$ getting 
$$
 \left\{ \begin{matrix}\partial_tu_j + \partial_xv_j+\dot S_{j-1}w\,\partial_xu_j= R^1_j\\[1ex] \partial_tv_j+\partial_xu_j+v_j+\dot S_{j-1}w\,\partial_xv_j=R^2_j \end{matrix} \right.
$$ with 
 $$R^1_j\triangleq \dot S_{j-1}w\,\d_xu_j- \ddj(w\,\d_xu) \andf  
 R^2_j\triangleq\dot S_{j-1}w\,\d_xv_j- \ddj(w\,\d_xv).$$
The remainder terms $R^1_j$ and $R^2_j$ will be bounded according to Lemma \ref{CP}. 
To handle the left-hand side of the above localized system, we introduce
the following  functional, designed for high frequencies:
$$\wt\cL^2_j\triangleq\|(\d_xu_j,\d_xv_j)\|_{L^2}^2+\int_\R v_j\,\d_xu_j,$$
 and get
 $$\displaylines{\frac12\biggl(\frac{d}{dt}\wt\cL_j^2+\wt\cL_j^2\biggr)
 +\int_\R\d_x\bigl(\bigl(\dot S_{j-1}w\,\d_xu_j\bigr)\d_xu_j
+\d_x\bigl(\dot S_{j-1}w\,\d_xv_j\bigr)\d_xv_j\bigr)\hfill\cr\hfill
 +\!\int_\R\!\bigl(\dot S_{j-1}w\d_x v_j\d_xu_j+\d_x\bigl(\dot S_{j-1}w\d_xu_j\bigr)v_j\bigr)
 = \!\int_\R\!\bigl(\d_xR_j^1\d_xu_j+\d_xR_j^2\d_xv_j+\frac12\bigl(\d_xR_j^1v_j+R_j^2\d_{x}u_j\bigr)\bigr)\cdotp}
 $$
 Using integration by parts and multiplying by $2$ then yields
 $$\displaylines{\frac{d}{dt}\wt\cL_j^2\!+\!\wt\cL_j^2\!+\!\int_\R\d_x\dot S_{j-1}w\,\bigl((\d_xu_j)^2\!+\!(\d_xv_j)^2\bigr) = \int_\R\bigl(R_j^2\,\d_{x}u_j-R_j^1\,\d_xv_j-2R_j^1\,\d^2_xu_j-2R_j^2\,\d^2_xv_j\bigr)\cdotp} $$
{}From this, using Cauchy-Schwarz inequality, that
$\wt\cL_j\simeq \|(\d_xu_j,\d_xv_j)\|_{L^2}$ and Lemma~\ref{SimpliCarre},
we get for all $t\geq0$ and $j\geq J_0,$
$$\wt\cL_j(t)+\int_0^t\wt\cL_j\leq C\biggl(\int_0^t\|\d_xw\|_{L^\infty}\wt\cL_j+2^j\int_0^t\|R_j^1,R_j^2\|_{L^2}\biggr),$$ 
 with $C$ depending only on $J_0.$ Hence, multiplying by 
 $2^{\frac j2}$ and summing up on $j\geq J_0,$
\begin{multline}\label{eq:gmn}
\norme{(u,v)(t)}^h_{\dot{\mathbb{B}}^{\frac{3}{2}}_{2,1}}+\int_0^t\norme{(u,v)}^h_{\dot{\mathbb{B}}^{\frac{3}{2}}_{2,1}}\lesssim\norme{(u_0,v_0)}^h_{\dot{\mathbb{B}}^{\frac{3}{2}}_{2,1}}\\+\int_0^t\norme{\partial_xw}_{L^\infty}\norme{(u,v)}^h_{\dot{\mathbb{B}}^{\frac{3}{2}}_{2,1}}+\int_0^t\sum_{j\geq J_0}2^{\frac{3}{2}j}\norme{(R^1_j,R^2_j)}_{L^2}.
\end{multline}
At this stage, taking advantage of  Lemma \ref{CP}  with $s=3/2$ to bound the sum, 
we conclude to the desired inequality. 
\end{proof}



\subsection{Global a priori estimates for the toy model}
We are now ready  to establish the following proposition which is the key 
to the proof of the existence part of the theorem. 
\begin{Prop}\label{p:bound}
 Let $(u,v)$ be a smooth solution of $(TM)$  on $[0,T].$
 Then, still assuming that $2\leq p\leq4,$ there exists a constant $C$
 and an integer $J_0$ (corresponding to the threshold between low and high frequencies)
 such that for all $t\in[0,T],$ we have 
 $$X_p(t)\leq C\bigl(X_{p,0} + X_p^2(t)\bigr)$$
 with $X_{p,0}\triangleq \norme{(u_0,v_0)}^\ell_{\dot{\mathbb{B}}^{\frac{1}{p}}_{p,1}}
 +\norme{(u_0,v_0)}^h_{\dot{\mathbb{B}}^{\frac{3}{2}}_{2,1}}$ and 
 $$\displaylines{X_p(t)\triangleq\norme{(u,v)}^\ell_{L^\infty_t(\dot{\mathbb{B}}^{\frac{1}{p}}_{p,1})}+\norme{(u,v)}^h_{L^\infty_t(\dot{\mathbb{B}}^{\frac{3}{2}}_{2,1})}\hfill\cr\hfill+
\norme{u}^\ell_{L^1_t(\dot{\mathbb{B}}^{\frac{1}{p}+2}_{p,1})}
+\norme{v+\d_xu}_{L^1_t(\dot{\mathbb{B}}^{\frac{1}{p}}_{p,1})}+\norme{v}_{L^2_t(\dot{\mathbb{B}}^{\frac{1}{p}}_{p,1})}+\norme{(u,v)}^h_{L^1_t(\dot{\mathbb{B}}^{\frac{3}{2}}_{2,1})}.}$$
\end{Prop}
\begin{proof}  As a first, let us observe that $\norme{v}_{L^2_t(\dot{\mathbb{B}}^{\frac{1}{p}}_{p,1})}$
is dominated by the other terms of $X_p(t)$ (let us denote them by $\wt X_p(t)$).  This is clearly the case of the high frequency part 
since, by Bernstein inequality, H\"older inequality and the embedding  $\dot\B^{s}_{2,1}\hookrightarrow\dot\B^{s+\frac1p-\frac12}_{p,1}$
with $s=3/2,$  
$$\norme{v}_{L^2_t(\dot{\mathbb{B}}^{\frac{1}{p}}_{p,1})}^h\lesssim  \norme{v}_{L^2_t(\dot{\mathbb{B}}^{1+\frac{1}{p}}_{p,1})}^h
\lesssim  \norme{v}_{L^2_t(\dot{\mathbb{B}}^{\frac32}_{2,1})}^h\lesssim
\sqrt{\norme{v}_{L^1_t(\dot{\mathbb{B}}^{\frac{3}{2}}_{2,1})}^h\norme{v}_{L^\infty_t(\dot{\mathbb{B}}^{\frac{3}{2}}_{2,1})}^h}.$$
For the low frequency part, we write that
$$\norme{v}_{L^2_t(\dot{\mathbb{B}}^{\frac{1}{p}}_{p,1})}^\ell\leq \|\d_xu\|_{L^2_t(\dot{\mathbb{B}}^{\frac{1}{p}}_{p,1})}^\ell
+\|z\|_{L^2_t(\dot{\mathbb{B}}^{\frac{1}{p}}_{p,1})}^\ell\with z\triangleq v+\d_xu.$$
 By H\"older inequality and interpolation, we have
$$\|\d_xu\|_{L^2_t(\dot\B^{\frac1p}_{p,1})}^\ell\lesssim 
\Bigl(\|u\|^\ell_{L^\infty_t(\dot\B^{\frac1p}_{p,1})}\|u \|^\ell_{L^1_t(\dot\B^{\frac1p+2}_{p,1})}\Bigr)^{\frac12}
\andf \|z\|_{L^2_t(\dot\B^{\frac1p}_{p,1})}^\ell\lesssim 
\Bigl(\|z\|^\ell_{L^\infty_t(\dot\B^{\frac1p}_{p,1})}\|z \|^\ell_{L^1_t(\dot\B^{\frac1p}_{p,1})}\Bigr)^{\frac12}\cdotp
$$
As  the low frequencies of $z$ in  $L_t^\infty(\dot\B^{\frac1p}_{p,1})$ may be bounded
by $\wt X_p(t),$ we proved that 
\begin{equation}\label{eq:vL2}
\norme{v}_{L^2_t(\dot{\mathbb{B}}^{\frac{1}{p}}_{p,1})}\lesssim\wt X_p(t)\quad\hbox{for all }\ t\in\R^+.
\end{equation}
Let us also notice that by Sobolev embedding and Bernstein inequality, 
$$\norme{v+\d_xu}_{L^1_t(\dot{\mathbb{B}}^{\frac{1}{p}}_{p,1})}^h\lesssim
\norme{v}_{L^1_t(\dot{\mathbb{B}}^{\frac{3}{2}}_{2,1})}^h+\norme{\d_xu}_{L^1_t(\dot{\mathbb{B}}^{\frac{1}{2}}_{2,1})}^h.$$
Therefore, adding up the inequalities from Propositions \ref{APLP} and \ref{PropHfLp} 
 for $w=v$ and observing that $2\leq p\leq p^*,$ we get 
$$
X_p(t)\lesssim X_{p,0} +\int_0^t\bigl(\|(u,v)\|_{\dot\B^{\frac1p}_{p,1}}^\ell+\|(u,v)\|_{\dot\B^{\frac32}_{2,1}}^h\bigr)\bigl(\|v\|^\ell_{\dot\B^{\frac1p+1}_{p,1}}+\|v\|^h_{\dot\B^{\frac32}_{2,1}}\bigr)+\int_0^t\|v\|^\ell_{\dot\B^{\frac1p}_{p,1}}\|(\d_xu,\d_xv)\|_{\dot\B^{\frac1p}_{p,1}}.$$
Since 
$$
\|v\|^\ell_{\dot\B^{\frac1p+1}_{p,1}}\leq \|v+\d_xu\|^\ell_{\dot\B^{\frac1p+1}_{p,1}}
+C\|u\|^\ell_{\dot\B^{\frac1p+2}_{p,1}},
$$
we conclude that the inequality of Proposition \ref{p:bound} is satisfied. 
\end{proof}


\subsection{Proof of the existence part of Theorem \ref{ThmExistLp}}

The proof relies on the following classical result about the local existence of strong solutions  for hyperbolic symmetric systems of type
$$\left\{\begin{array}{l} \d_t U+\sum_{k=1}^d A_k(U)\d_kU+A_0(U)=0,\\[1ex]
U|_{t=0}=U_0,\end{array}\right.\leqno(QS)$$
where $A_k,$ $k=0,\dots,d$  are smooth functions from $\R^n$ to the space of $n\times n$  matrices, 
that are symmetric if $k\not=0.$ 
\begin{Thm}{\cite[Chap.~4]{HJR}}  \label{ThmExistLocalLp}
Let $U_0$  be in  the nonhomogeneous Besov space $\mathbb{B}^{\frac{d}{2}+1}_{2,1}(\R^d;\R^n)$. Then, $(QS)$ 
admits a unique maximal solution $U$  in $\mathcal{C}([0,T^*[;{\mathbb{B}^{\frac{d}{2}+1}_{2,1}})\cap\mathcal{C}^1([0,T^*[;{\mathbb{B}^{\frac{d}{2}}_{2,1}}),$ and there exists a positive constant $c$ such that 
$$T^*\geq\frac{c}{\norme{U_0}_{\mathbb{B}^{\frac{d}{2}+1}_{2,1}}}\cdotp$$
Furthermore, 
$$T^*<\infty \Longrightarrow \int_0^{T^*}\norme{\nabla U}_{L^\infty}=\infty.$$
\end{Thm}
The proof of the existence part   of Theorem \ref{ThmExistLp} is structured as follows. First, 
we multiply the  low frequencies of the data by a cut-off function  in order to have 
data in $\B^{\frac32}_{2,1}.$ One can then use the above theorem to construct
a sequence of approximate solutions, that are shown to be global. We prove uniform estimates in the space $E_p$ for those solutions, 
pass to the limit up to subsequence by means of compactness arguments, and 
finally check that the limit is a solution of $(TM)$ with the required properties.

\subsubsection*{First step. Construction of approximate solutions}

Let $(u_0,v_0)$ be such that $u_0^\ell,v_0^\ell\in\dot{\mathbb{B}}^{\frac{1}{p}}_{p,1}$ and $u_0^h,v_0^h\in\dot{\mathbb{B}}^{\frac{3}{2}}_{2,1}$. 
Since $(u_0,v_0)$ need not be in $\mathbb{B}^{\frac{3}{2}}_{2,1},$  we set for all $n\geq1,$
$$u_{0}^n\triangleq   \chi_n\, \dot S_{J_0-5}u_0+({\rm Id}-\dot S_{J_0-5})u_0
 \andf v_{0}^n\triangleq \chi_n\, \dot S_{J_0-5}v_0+({\rm Id}-\dot S_{J_0-5})v_0
 $$
with $\chi_n\triangleq \chi(n^{-1}\cdot),$
where $\chi$ stands (for instance) for the bump function of Section \ref{s:main}. 
\smallbreak
 It is obvious that the sequences $(u_{0}^n)_{n\in\N}$ and $(v_{0}^n)_{n\in\N}$ tend  to $u_0$ and $v_0$  in the sense of distributions, when $n$ tends to infinity.
Moreover, as $u_0^\ell$ and $v_0^\ell$ are in $\dot\B^{\frac1p}_{p,1},$  the  low frequencies of the data
are in $L^\infty,$ and the spatial truncation thus guarantees that  $u_{0}^n,v_{0}^n\in {\mathbb{B}}^{\frac{3}{2}}_{2,1}.$ Hence,  Theorem \ref{ThmExistLocalLp} provides us with a unique  maximal solution $(u^n,v^n)\in\mathcal{C}([0,T_n[;\mathbb{B}^{\frac{3}{2}}_{2,1})\cap\mathcal{C}^1([0,T_n[;\mathbb{B}^{\frac{1}{2}}_{2,1}).$ 
\smallbreak
We claim that  we have  for $z_0=u_0, v_0,$ 
\begin{equation}\label{eq:uon}
\norme{z_{0}^n}^\ell_{\dot{\mathbb{B}}^{\frac{1}{p}}_{p,1}}+ \norme{z_{0}^n}^h_{\dot{\mathbb{B}}^{\frac{3}{2}}_{2,1}}\lesssim\norme{z_0}^\ell_{\dot{\mathbb{B}}^{\frac{1}{p}}_{p,1}}+\norme{z_0}^h_{\dot{\mathbb{B}}^{\frac{3}{2}}_{2,1}}.
\end{equation}
Indeed,   since $\norme{\chi_n}^\ell_{\dot{\mathbb{B}}^{\frac{1}{p}}_{p,1}}\simeq\norme{\chi}^\ell_{\dot{\mathbb{B}}^{\frac{1}{p}}_{p,1}}<\infty,$
 owing to the invariance of the norm in $\dot\B^{\frac1p}_{p,1}$ by spatial dilation (see e.g. \cite[Rem. 2.19]{HJR}), 
 we may write 
$$\begin{aligned}\norme{z_{0}^n}^\ell_{\dot{\mathbb{B}}^{\frac{1}{p}}_{p,1}} &\leq
 \norme{\chi_n\,\dot S_{J_0-5}z_0}^\ell_{\dot{\mathbb{B}}^{\frac{1}{p}}_{p,1}}+\norme{({\rm Id}-\dot S_{J_0-5})z_0}^\ell_{\dot{\mathbb{B}}^{\frac{1}{p}}_{p,1}} \\ &\lesssim 
 \norme{z_0}^\ell_{\dot{\mathbb{B}}^{\frac{1}{p}}_{p,1}}\norme{\chi_n}_{\dot{\mathbb{B}}^{\frac{1}{p}}_{p,1}}+\norme{z_0}_{\dot{\mathbb{B}}^{\frac{1}{p}}_{p,1}}  \\ &\lesssim 
 \norme{z_0}^\ell_{\dot{\mathbb{B}}^{\frac{1}{p}}_{p,1}}+ \norme{z_0}^h_{\dot{\mathbb{B}}^{\frac{3}{2}}_{2,1}}.
\end{aligned}$$
Next, we see that 
$$\norme{z_{0}^n}^h_{\dot{\mathbb{B}}^{\frac{3}{2}}_{2,1}}\leq 
\|\chi_n\, \dot S_{J_0-5}z_0\|^h_{\dot{\mathbb{B}}^{\frac{3}{2}}_{2,1}}
+\|({\rm Id}-\dot S_{J_0-5})z_0^h\|_{\dot{\mathbb{B}}^{\frac{3}{2}}_{2,1}}.$$
It is obvious that the last term may be bounded  by $\|z_0\|^h_{\dot{\mathbb{B}}^{\frac{3}{2}}_{2,1}}.$
For the other term, the important observation is that for $j\geq J_0,$ we have
$$\ddj\bigl(\chi_n\, \dot S_{J_0-5}z_0\bigr) = \sum_{j'\geq j-3} \ddj\bigl(\dot S_{j'+2}\dot S_{J_0-5}z_0\,\dot\Delta_{j'}\chi_n\bigr)\cdotp$$ 
Hence, owing to the scaling properties of the space $\dot\B^{\frac32}_{2,1},$  
$$
\|\chi_n\, \dot S_{J_0-5}z_0\|^h_{\dot{\mathbb{B}}^{\frac{3}{2}}_{2,1}}\lesssim \|\dot S_{J_0-5}z_0\|_{L^\infty}
\|\chi_n\|_{\dot\B^{\frac32}_{2,1}}\lesssim n^{-1} \|z_0\|_{\dot\B^{\frac1p}_{p,1}},$$
which eventually yields  \eqref{eq:uon}.

\subsubsection*{Second step. Uniform estimates} 

Since, for all $T>0,$ the space   $\mathcal{C}([0,T];\mathbb{B}^{\frac{3}{2}}_{2,1})\cap\mathcal{C}^1([0,T];\mathbb{B}^{\frac{1}{2}}_{2,1})$ is included in our `solution space' $E_p(T)$ (that is, $E_p$ restricted to $[0,T]$), 
one can take advantage of Proposition \ref{p:bound}    for bounding our sequence. 
From it and \eqref{eq:uon}, we get, denoting $X^n_p$ the function $X_p$ pertaining to $(u^n,v^n),$   
 \begin{equation} \label{Xp}   
X^n_p\leq C\bigl(X_{p,0}+(X^n_p)^2\bigr)\cdotp
\end{equation}
It  is clear that if \begin{equation}
    2CX^n_p(t)\leq 1, \label{Xp1}
\end{equation} then  Inequality \eqref{Xp} implies that 
   $$ X^n_p(t)\leq 2CX_{p,0}. $$
Then, thanks to a classical bootstrap argument, we can conclude that if $X_{p,0}$ is small enough then \eqref{Xp1} is true as long as the solution exists. Hence, 
there exists a constant $C$ such that 
\begin{equation}\label{Xp2}
 X^n_p(t)\leq  CX_{p,0}\quad\hbox{for all } \ n\geq1\andf t\in[0,T_n[.
\end{equation}
In order to show that the above inequality  implies that the solution is global (namely that $T_n=\infty$), one can argue by contradiction, assuming that $T_n<\infty,$ and use the blow-up criterion 
of Theorem \ref{ThmExistLocalLp}. However, we first have to justify that the nonhomogeneous Besov norm $\B^{\frac32}_{2,1}$ of the solution is under control \emph{up to time $T_n.$}
Now, applying the standard energy method to $(TM)$ yields  for all $t<T_n,$ 
 $$\norme{(u^n,v^n)(t)}_{L^2}^2\leq \norme{(u^n_0,v^n_0)}_{L^2}^2 + \int_0^t\norme{\partial_xv^n}_{L^\infty}\norme{(u^n,v^n)}_{L^2}^2.$$ 
 Since \eqref{Xp2} and the embedding of $\dot\B^{\frac1p}_{p,1}$ and $\dot\B^{\frac12}_{2,1}$ in $L^\infty$ ensure that $\d_xv^n$ is in $L^1_{T_n}(L^\infty),$
 using Gronwall lemma  gives that $(u^n,v^n)$ is in  $L^\infty_{T_n}(L^2),$ and thus in 
  $L^\infty_{T_n}({\mathbb{B}}^{\frac{3}{2}}_{2,1})$ owing, again,  to \eqref{Xp2}. 

It is now easy to conclude : for all $t_{0,n}\in[0,T_n[,$  Theorem \ref{ThmExistLocalLp} provides us with  a solution of $(TM)$ with the initial data $(u(t_{0,n}),v(t_{0,n})),$ 
 on $[t_{0,n},T+t_{0,n}]$ for some $T$ that may be bounded from below independently of   $t_{0,n}$. 
Consequently, choosing  $t_{0,n}$ such that $t_{0,n}>T_n-T$, we see that  the solution $(u^n,v^n)$ can be extended beyond $T_n$, which 
contradicts the maximality of $T_n.$
Hence $T_n=+\infty$ and the solution corresponding to the initial data $(u^n_0,v^n_0)$ is global in time and satisfies $\eqref{Xp2}$ for all time.

\subsubsection*{Third step. Convergence} 

We have to show that $(u^n,v^n)_{n\in\N}$ tends, up to subsequence, to some $(u,v)\in E_p$ in the sense of distribution, that satisfies $(TM).$ 

The proof that we here propose rests on  Ascoli Theorem and suitable compact embeddings. 
Let us explain how it goes for $(u^n)_{n\in\N},$ the convergence of $(v^n)_{n\in\N}$
being similar. {}From \eqref{Xp2} and elementary embedding, we know that~: 
\begin{itemize}
\item  $(\d_xu^n)_{n\in\N}$ and  $(\d_xv^n)_{n\in\N}$  are bounded in $L^2(\dot\B^{\frac1p}_{p,1})$,
\item $(v^n)_{n\in\N}$ is bounded in $L^\infty(\dot\B^{\frac1p}_{p,1})$.
\end{itemize}
Hence, both $(v^n\d_xu^n)_{n\in\N}$ and $(\d_xv^n)_{n\in\N}$  are bounded
in  $L^2(\dot\B^{\frac1p}_{p,1}),$ which implies that $(\partial_tu^n)_{n\in\N}$ is bounded in $L^2(\dot{\mathbb{B}}^{\frac{1}{p}}_{p,1}).$
This means that $(u^n)_{n\in\mathbb{N}}$ viewed as a sequence of functions valued in $\dot{\mathbb{B}}^{\frac{1}{p}}_{p,1}$  is locally equicontinuous on $\mathbb{R}^+$. 

Moreover $(u^{n,h})_{n\in\mathbb{N}}$ is bounded in $\mathcal{C}(\mathbb{R}^+,\dot{\mathbb{B}}^{\frac{3}{2}}_{2,1})$, $(u^{n,\ell})_{n\in\mathbb{N}}$ is bounded in $\mathcal{C}(\mathbb{R}^+,\dot{\mathbb{B}}^{\frac{1}{p}}_{p,1})$ and we know, thanks to a result  of \cite[Chap. 2]{HJR}, that the embedding 
from  $F=\{u\in\mathcal{S}',\: u^\ell\in\dot{\mathbb{B}}^{\frac{1}{p}}_{p,1}\text{ and }  u^h\in\dot{\mathbb{B}}^{\frac{3}{2}}_{2,1} \}$  to  $\dot{\mathbb{B}}^{\frac{1}{p}}_{p,1}$ is locally compact. Therefore, one can combine  Ascoli Theorem and the Cantor diagonal extraction process
 to deduce that there exists a distribution $u$   
 such that, up to subsequence  $(\phi u^n)_{n\in\mathbb{N}}$ converges to $\phi u$ 
 in  $\mathcal{C}(\mathbb{R}^+;\dot\B^{\frac1p}_{p,1})$ for all function $\phi$ compactly supported in $\mathbb{R}^+\times\mathbb{R}^n$. Then, using the Fatou property (cf. \cite{HJR}, chapter 2) we obtain that $u^\ell\in L^\infty(\dot{\mathbb{B}}^{\frac{1}{p}}_{p,1})\cap L^1(\dot{\mathbb{B}}^{\frac{1}{p}+2}_{p,1})$ and $u^h\in L^\infty(\dot{\mathbb{B}}^{\frac{3}{2}}_{2,1})\cap L^1(\dot{\mathbb{B}}^{\frac{3}{2}}_{2,1}),$ with norms bounded by the right-hand side of \eqref{Xp2}.
 One can argue similarly for establishing the weak convergence of $(v^n)_{n\in\N}$
 to some distribution $v$  fulfilling the desired properties of regularity up to
 time continuity. 
 
 Finally, passing to the limit in $(TM)$ is not an issue, since we have strong convergence (after
 localization) in norms with positive indices of regularity.

\subsubsection*{Last step. Proving that \texorpdfstring{$(u,v)\in E_p$}{TEXT}}

The only property that misses is the time continuity. 
It may be obtained by looking at $u$ and $v$ as solutions of transport equations. 
Indeed, by construction, we have
$$\d_tu+v\d_xu =-\d_xv\andf \d_tv+v\d_xv+v=-\d_xu.$$
The properties we proved so far for $u$ and $v$ ensure that 
$\d_xu$ and $\d_xv$ belong to $L^2(\dot\B^{\frac1p}_{p,1}).$ 
Hence, the standard properties for the transport equation (see e.g. \cite[Chap. 3]{HJR})
give us that $(u,v)\in\cC(\R^+;\dot\B^{\frac1p}_{p,1}).$ 

To show that $(u,v)^h\in \cC(\R^+;\dot B^{\frac32}_{2,1}),$ one can mimic
the proof for general symmetric hyperbolic systems, summing up only on 
high frequencies, as presented at \cite[p.196]{HJR} for instance.

In the end, we thus have proved that $(u,v)$ is a global solution of $(TM),$ 
that verifies the desired properties of regularity and  $X_p(t)\leq  CX_{p,0}$ for all $t\in\R^+.$


\subsection{Proof of  uniqueness}

Consider two solutions $(u_1,v_1)$ and $(u_2,v_2)$ of $(TM)$ (not necessarily small)
 in  the space $E_p,$ that correspond  to the same initial data $(u_0,v_0).$ The proof of uniqueness will 
 follow from stability estimates involving  suitable norms. The difficulty is that our functional framework 
 is not the standard one for the low frequency of the solution, so that one cannot follow the classical 
 proof  for  hyperbolic symmetric systems. 
Here  we shall estimate $(\du,\dv):=(u_2-u_1,v_2-v_1)$ in the space 
\begin{equation}\label{eq:FpT}
F_p(T)\triangleq \Bigl\{z\in\cC([0,T];\dot\B^{\frac 2p-\frac 12}_{p,1})\,:\, 
z^h\in \cC([0,T];\dot\B^{\frac 12}_{2,1})\Bigr\}
\cdotp\end{equation}
The reason for this choice is the usual loss of one derivative when proving stability estimates for quasilinear hyperbolic systems
(hence the exponent $1/2$ for high frequencies). The exponent for low frequencies looks to be the best one 
for controlling  the nonlinearities. Before starting the proof, we introduce the notation
$$\dU(t)\triangleq \|(\du,\dv)(t)\|^\ell_{\dot\B^{\frac2p-\frac12}_{p,1}}+\|(\du,\dv)(t)\|^h_{\dot\B^{\frac12}_{2,1}}.$$

\subsubsection*{Step 1. Proving that  \texorpdfstring{$(\du,\dv)\in F_p(T)$}{TEXT}} 

Remember that $\d_tu_i=-\d_xv_i -v_i\d_xu_i$ for $i=1,2.$
By interpolation in Besov spaces and H\"older inequality with respect to the time variable, 
since $\d_xu_i^\ell$ and $\d_xv_i^\ell$ are in $L^\infty(\R^+;\dot\B^{\frac1p-1}_{p,1})\cap L^1(\R^+;\dot\B^{\frac1p+1}_{p,1}),$  we get
\begin{equation}\label{eq:Lr}\d_xu_i^\ell,\d_xv_i^\ell\in L^r(\R^+;\dot\B^{\frac2p-\frac12}_{p,1})
\with \frac1r\triangleq\frac14+\frac1{2p}\cdotp\end{equation}
It is clear that the same property holds for the high frequencies of $\d_xu_i$ and 
$\d_xv_i,$  since they belong  to $ L^1(\dot\B^{\frac1p}_{p,1})\cap L^\infty(\dot\B^{\frac1p}_{p,1}).$
We also know that $v_i$ belongs to $L^\infty(\R^+;\dot\B^{\frac1p}_{p,1}).$ Therefore, from 
the product laws in Besov spaces that have been recalled in Proposition \ref{LP}, we
gather that $\d_xv_i$ and $v_i\d_xu_i$ are in  $L^r(\R^+;\dot\B^{\frac2p-\frac12}_{p,1}).$
Hence, $\d_tu_i$ is in  $L^r(\R^+;\dot\B^{\frac2p-\frac12}_{p,1}),$  and thus
\begin{equation}\label{eq:uniq1} (u_i-u_0)\in \cC^{\frac1{r'}}_{loc}(\R^+;\dot\B^{\frac2p-\frac12}_{p,1}).\end{equation}
Proving the result for $v_i$ is almost the same, except that we have to handle the damping term. 
To overcome it, we notice that 
$$\d_t(e^tv_i)=-e^tv_i\,\d_xv_i- e^t\d_xu_i.$$
Arguing as above, we get that $\d_t(e^tv_i)\in L^r_{loc}(\R^+;\dot\B^{\frac2p-\frac12}_{p,1}),$ whence
\begin{equation}\label{eq:uniq2} (e^tv_i-v_0)\in \cC^{\frac1{r'}}_{loc}(\R^+;\dot\B^{\frac2p-\frac12}_{p,1}).\end{equation}
{}From \eqref{eq:uniq1} and \eqref{eq:uniq2}, we conclude that $(\du,\dv)\in F_p(T)$ for all finite $T.$

\subsubsection*{Step 2. Estimates for the low frequencies} 

The system satisfied by $(\du,\dv)$ reads: 
\begin{equation}\label{eq:uniq3}\left\{\begin{array}{l} \d_t\du+\d_x\dv=-\dv\,\d_xu_1-v_2\,\d_x\du,\\[1ex]
\d_t\dv +\dv +\d_x\du =-\dv\,\d_xv_1-v_2\,\d_x\dv.\end{array}\right.\end{equation}
Then, we follow the computations leading to Proposition \ref{APLP} with $w=0,$
looking at  $-\dv\,\d_xu_1-v_2\,\d_x\du$ and $-\dv\,\d_xv_1-v_2\,\d_x\dv$ as  source terms, 
and working at the level of regularity $2/p-1/2$ instead of $1/p$
(since the left-hand side of \eqref{eq:uniq3} is linear with constant coefficients, 
this shift of regularity does not change the proof). 
Omitting the time integral in the left-hand side of the Inequality given by Proposition \ref{APLP}, 
we find that
$$
\|(\du,\dv)(t)\|^\ell_{\dot\B^{\frac2p-\frac12}_{p,1}}\lesssim 
\int_0^t\bigl( \|\dv\,\d_xu_1\|^\ell_{\dot\B^{\frac2p-\frac12}_{p,1}}
+\|v_2\,\d_x\du\|^\ell_{\dot\B^{\frac2p-\frac12}_{p,1}}+\|\dv\,\d_xv_1\|^\ell_{\dot\B^{\frac2p-\frac12}_{p,1}}
+\|v_2\,\d_x\dv\|^\ell_{\dot\B^{\frac2p-\frac12}_{p,1}}\bigr)\cdotp$$
In order to bound the right-hand side, we may use that  Proposition \ref{Composition} yields
\begin{equation}\label{eq:loi1}
\|a\,b\|_{\dot\B^{\frac2p-\frac12}_{p,1}}\lesssim 
\|a\|_{\dot\B^{\frac1p}_{p,1}}\,\|b\|_{\dot\B^{\frac2p-\frac12}_{p,1}}\end{equation}
as well as the following inequality  that is a consequence of \eqref{eq:prod3} in 
the appendix (take $s=\frac2p-\frac12$ therein):
\begin{equation}\label{eq:loi2}\|a\,b\|^\ell_{\dot\B^{\frac2p-\frac12}_{p,1}}\lesssim  \|a\|_{\dot\B^{\frac1p}_{p,1}\cap
\dot\B^{\frac1p+1}_{p,1}}\,\|b\|_{\dot\B^{\frac2p-\frac32}_{p,1}}.\end{equation}
In the end, choosing $a=v_2$ and $b=\d_x\du$ or $\d_x\dv,$  we get 
 \begin{equation}\label{eq:uniq5}
 \|(\du,\dv)(t)\|^\ell_{\dot\B^{\frac2p-\frac12}_{p,1}}\lesssim \int_0^t
 \bigl(\|(\d_xu_1,\d_xv_1)\|_{\dot\B^{\frac2p-\frac12}_{p,1}}
 \|\dv\|_{\dot\B^{\frac1p}_{p,1}}
 +\|v_2\|_{\dot\B^{\frac1p}_{p,1}\cap\dot\B^{\frac1p+1}_{p,1}}\|(\du,\dv)\|_{\dot\B^{\frac2p-\frac12}_{p,1}}\bigr)\cdotp
 \end{equation}

 \subsubsection*{Step 3. Estimates for the high frequencies} 
 
 We look at the system satisfied by $(\du,\dv)$ under the form: 
$$\left\{\begin{array}{l} \d_t\du+v_2\,\d_x\du+\d_x\dv=-\dv\,\d_xu_1,\\[1ex]
\d_t\dv +\dv  +v_2\,\d_x\dv+\d_x\du =-\dv\,\d_xv_1.\end{array}\right.$$
This is System $(LTM)$ except for the source terms in the right-hand side. 
Clearly, following the computations leading to \eqref{eq:gmn}, but
using the index $1/2$ instead of $3/2$ gives
\begin{multline}\label{eq:uniq6}
\|(\du,\dv)(t)\|^h_{\dot\B^{\frac12}_{2,1}}\lesssim
\int_0^t\|\d_xv_2\|_{L^\infty}\|(\du,\dv)\|^h_{\dot\B^{\frac12}_{2,1}}
\\+\int_0^t\sum_{j\geq J_0} 2^{\frac j2}\bigl(\|[v_2,\ddj]\d_x\du\|_{L^2}+\|[v_2,\ddj]\d_x\dv\|_{L^2}\bigr)
+\int_0^t\bigl( \|\dv\,\d_xu_1\|^h_{\dot\B^{\frac12}_{2,1}}+\|\dv\,\d_xv_1\|^h_{\dot\B^{\frac12}_{2,1}}\bigr)\cdotp\end{multline} Let $p^*\triangleq 2p/(p-2).$ 
Lemma \ref{CP} tells us that, for $z=\du,\dv,$ 
 $$\displaylines{
\sum_{j\geq J_0}2^{\frac j2}\norme{[v_2,\ddj]\d_xz}_{L^2}\lesssim \norme{\partial_xv_2}_{L^\infty}\norme{z}^h_{\dot{\mathbb{B}}^{\frac12}_{2,1}}+ \norme{\partial_xz}_{\dot{\mathbb{B}}^{\frac{1}{p}-1}_{p,1}}\norme{v_2}^\ell_{\dot{\mathbb{B}}^{1+\frac{1}{p^*}}_{p^*,1}}\hfill\cr\hfill
+\norme{\partial_xz}_{\dot{\mathbb{B}}^{-1}_{\infty,\infty}}\norme{v_2}^h_{\dot{\mathbb{B}}^{\frac12}_{2,1}}
+ \norme{\partial_xz}^\ell_{\dot{\mathbb{B}}^{-\frac{1}{p}}_{p,1}}\norme{\d_xv_2}^\ell_{\dot{\mathbb{B}}^{-\frac{1}{p^*}}_{p^*,1}}\cdotp}$$
Hence, using obvious embedding and the fact that
$$ \norme{\partial_xz}^\ell_{\dot{\mathbb{B}}^{-\frac{1}{p}}_{p,1}}\lesssim
 \|z\|^\ell_{\dot\B^{\frac2p-\frac12}_{p,1}},\quad
  \|v_2\|^\ell_{\dot{\mathbb{B}}^{1+\frac{1}{p^*}}_{p^*,1}}\lesssim\|v_2\|^\ell_{\dot\B^{\frac1p}_{p,1}}\andf
 \|\d_xv_2\|^\ell_{\dot{\mathbb{B}}^{-\frac{1}{p^*}}_{p^*,1}}\lesssim\|v_2\|^\ell_{\dot\B^{\frac1p}_{p,1}}$$
yields for $z=\du,\,\dv,$ 
\begin{equation}\label{eq:uniq7}
\sum_{j\geq J_0}2^{\frac j2}\norme{[v_2,\ddj]\d_xz}_{L^2}\lesssim
\bigl(\|v_2\|^\ell_{\dot\B^{\frac1p}_{p,1}}+\|v_2\|^h_{\dot\B^{\frac32}_{2,1}}\bigr)
\bigl(\|z\|^\ell_{\dot\B^{\frac2p-\frac12}_{p,1}}+\|z\|^h_{\dot\B^{\frac12}_{2,1}}\bigr)\cdotp\end{equation}
The last two terms of \eqref{eq:uniq6} may be bounded thanks to 
 Inequality \eqref{eq:prod4}: we get for $z=u_1,\, v_1,$ 
 \begin{equation}\label{eq:uniq8}
 \|\dv\,\d_xz\|_{\dot\B^{\frac12}_{2,1}}^h\lesssim \bigl(\|\dv\|^\ell_{\dot\B^{\frac2p-\frac12}_{p,1}}+\|\dv\|_{\dot\B^{\frac12}_{2,1}}^h\bigr)\bigl(\|\d_xz\|_{\dot\B^{\frac1p-1}_{p,1}}^\ell
 +\|\d_xz\|_{\dot\B^{\frac12}_{2,1}}^h\bigr)\cdotp
 \end{equation}
Plugging \eqref{eq:uniq7} and \eqref{eq:uniq8} in \eqref{eq:uniq6}, and using also the fact that
$$\|\d_xv_2\|_{L^\infty}\lesssim \|v_2\|_{\dot\B^{\frac1p}_{p,1}}^\ell + 
\|v_2\|_{\dot\B^{\frac32}_{2,1}}^h,$$ 
we end up with  
 \begin{equation}\label{eq:uniq9}
 \|(\du,\dv)(t)\|^h_{\dot\B^{\frac12}_{2,1}}\lesssim \int_0^t
 \bigl(\|(u_1,v_1,v_2)\|_{\dot\B^{\frac1p}_{p,1}}^\ell
 +\|(u_1,v_1,v_2)\|_{\dot\B^{\frac32}_{2,1}}^h\bigr)\dU.\end{equation}

  \subsubsection*{Step 4. Conclusion} 
 
 Putting \eqref{eq:uniq5} and \eqref{eq:uniq9} together yields
 $$\dU(t)\leq C\int_0^t\bigl(\|(u_1,v_1,v_2)\|_{\dot\B^{\frac1p}_{p,1}}^\ell
 +\|(u_1,v_1,v_2)\|_{\dot\B^{\frac32}_{2,1}}^h+\|(\d_xu_1,\d_xv_1)\|_{\dot\B^{\frac2p-\frac12}_{p,1}}
\bigr)\dU\,d\tau.$$
 Knowing that $(u_1,v_1)$ and $(u_2,v_2)$ are in $E_p(T)$ and remembering \eqref{eq:Lr}, we get, 
  $$ \int_0^T\bigl(\|(u_1,v_1,v_2)\|_{\dot\B^{\frac1p}_{p,1}}^\ell
 +\|(u_1,v_1,v_2)\|_{\dot\B^{\frac32}_{2,1}}^h\bigr)<\infty.$$
 Hence, applying Gronwall lemma allows to conclude that 
 $\dU\equiv0$ on $[0,T].$
 In other words,  $(u_1,v_1)$ and $(u_2,v_2)$ coincide on $[0,T]\times\R.$
 Since $T$ is arbitrary, uniqueness is proved. 
 \qed


\section{Compressible Euler system with damping}\label{s:Euler}

As  pointed out  in the introduction,  System $(TM)$ may be seen as  an approximation of  the damped isentropic compressible Euler system 
with pressure law $P(\rho)=\frac{\rho^2}{2}\cdotp$
Here we want to adapt the method of the previous sections 
to the true damped compressible Euler system~:
$$ \left\{ \begin{aligned} &\partial_t\rho+\d_x(V\rho)=0,\\ 
&\partial_t(\rho V)+\partial_x(\rho V^2)+\partial_x(P(\rho))+\lambda \rho V=0, \end{aligned} \right.\leqno(E)$$
  supplemented with initial data $(\rho_0,V_0)$ that is a perturbation of some
  constant state $(\bar\rho,0)$ with $\bar\rho>0.$ 
  The (given) pressure function $P$ is assumed  to be  smooth and such that~: 
  \begin{equation}\label{eq:pressure}\begin{cases}
  \hbox{ Case } 2<p\leq4:  P(\rho)=  a \rho^\gamma \ \hbox{for some positive } a\!\!\andf\!\! \gamma\ \hbox{in a neighborhood
  of}\ \bar\rho.\\
  \hbox{ Case } p=2:  \hbox{just }\ P'(\bar\rho)>0.
  \end{cases}\end{equation}
Note that, performing a suitable normalization reduces the study to the case
$\bar\rho=P'(\bar\rho)=1$ (hence $a=1/\gamma$ in the first case),  which will be assumed  from now on.
  \medbreak
  Consider  the new unknown $$\displaystyle n(\rho)\triangleq\int_1^\rho \frac{P'(s)}{s}\,ds.$$
Since our assumptions on the pressure guarantee that  $\rho\mapsto n(\rho)$ is a smooth diffeomorphism 
 from a neighborhood of $1$ to a neighborhood of $0,$
one can rewrite $(E)$ under the form
\begin{equation} \left\{ \begin{aligned} &\partial_tn+V\partial_x n+\partial_xV+G(n)\partial_xV=0,\\ 
&\partial_tV+V\partial_xV+\partial_xn+\lambda V=0, \end{aligned} \right.\label{CED4}
\end{equation} where $G(n)$ is defined by the relation\footnote{In what follows, we shall only use that $G$ is  a smooth function vanishing at $0.$}
$G(n(\rho))\triangleq P'(\rho)-1.$
\begin{Thm} \label{ThmEulerLp} Under hypothesis \eqref{eq:pressure}, there exist $k=k(p)\in\Z$ and $c_0=c_0(p)>0$ such that for $J_\lambda\triangleq\left\lfloor\rm log_2\lambda\right\rfloor+k,$ if we assume that $(n_0,V_0)$ fulfills the same conditions as in Theorem \ref{ThmExistLp}, then System \eqref{CED4}  admits a unique global solution $(n,V)$ verifying the same properties as the solution therein.
Furthermore, Corollary \ref{CorLambda} and Theorem \ref{ThmDecayTM} hold true (with the same decay rate).
\end{Thm}
\begin{proof} 
Performing the rescaling \eqref{eq:rescaling} reduces
the proof to the case $\lambda=1,$  and we are thus left with bounding for all $t\geq0,$
the functional 
$$\displaylines{
X_p(t)\triangleq  \norme{(n,V)}^\ell_{L^\infty_t(\dot{\mathbb{B}}^{\frac{1}{p}}_{p,1})}+\norme{(n,V)}^h_{L^\infty_t(\dot{\mathbb{B}}^{\frac{3}{2}}_{2,1})}\hfill\cr\hfill+
\norme{n}^\ell_{L^1_t(\dot{\mathbb{B}}^{\frac{1}{p}+2}_{p,1})}+\norme{(n,V)}^h_{L^1_t(\dot{\mathbb{B}}^{\frac{3}{2}}_{2,1})}
+\norme{V+\d_xn}_{L^1_t(\dot{\mathbb{B}}^{\frac{1}{p}}_{p,1})}+\norme{V}_{L^2_t(\dot{\mathbb{B}}^{\frac{1}{p}}_{p,1})}}$$
in terms of 
$$X_{p,0}\triangleq  \norme{(n_0,V_0)}^\ell_{\dot{\mathbb{B}}^{\frac{1}{p}}_{p,1}}
 +\norme{(n_0,V_0)}^h_{\dot{\mathbb{B}}^{\frac{3}{2}}_{2,1}}.$$
Remember that $\norme{V}_{L^2_t(\dot{\mathbb{B}}^{\frac{1}{p}}_{p,1})}$
and $\norme{V+\d_xn}_{L^1_t(\dot{\mathbb{B}}^{\frac{1}{p}}_{p,1})}^h$  are  bounded by the first four terms  of $X_p$ (see \eqref{eq:vL2}). 

\subsubsection*{Low frequencies estimates} 

We  follow  the method we used for $(TM),$   looking at $G(n)\partial_xV$ as a source term.
Owing to  Propositions \ref{LP} and \ref{Composition}, we have
$$\norme{G(n)\partial_{x}V}_{\dot\B^{\frac{1}{p}}_{p,1}}\lesssim  \norme{n}_{\dot\B^{\frac{1}{p}}_{p,1}}\norme{\partial_xV}_{\dot\B^{\frac{1}{p}}_{p,1}}.$$
Therefore, mimicking the proof of Proposition \ref{APLP}, we end up again  for all $t\geq0$ with 
\begin{multline}\label{eq:EulerBF}
\norme{(n,V)(t)}_{\dot{\mathbb{B}}^{\frac{1}{p}}_{p,1}}^\ell
+\int_0^t\norme{n}_{\dot{\mathbb{B}}^{\frac{1}{p}+2}_{p,1}}^\ell+\int_0^t\norme{V+\d_xn}_{\dot{\mathbb{B}}^{\frac{1}{p}}_{p,1}}^\ell \\
\leq C\biggl(\norme{(n_0,V_0)}^\ell_{\dot{\mathbb{B}}^{\frac{1}{p}}_{p,1}}
+\int_0^t\norme{(n,V,\d_xn)}_{\dot{\mathbb{B}}^{\frac{1}{p}}_{p,1}}
\norme{V}_{\dot{\mathbb{B}}^{\frac{1}{p}+1}_{p,1}}\biggr)\cdotp  \end{multline}

\subsubsection*{High frequencies estimates.}

One has the following proposition:
\begin{Prop} \label{HfEuler}
Let $(n,V)$ be a smooth solution of \eqref{CED4} on the  interval $[0,T],$ under assumption \eqref{eq:pressure}.  
Define $p^*$ by the relation $1/p+1/p^*=1/2.$ 
There exists a constant $C$ depending only on  the threshold $J_0$
between the low and high frequencies such that for all $t\in[0,T],$
$$\displaylines{
\norme{(n,V)(t)}^h_{\dot{\mathbb{B}}^{\frac{3}{2}}_{2,1}}+\int_0^t\norme{(n,V)}^h_{\dot{\mathbb{B}}^{\frac{3}{2}}_{2,1}}\lesssim\norme{(n_0,V_0)}^h_{\dot{\mathbb{B}}^{\frac{3}{2}}_{2,1}} + \int_0^t\norme{(\partial_xn,\partial_xV)}_{\dot{\mathbb{B}}^{\frac{1}{2}}_{2,1}}\norme{(n,V)}^h_{\dot{\mathbb{B}}^{\frac{3}{2}}_{2,1}}\hfill\cr\hfill
+\int_0^t\biggl(\norme{V}^\ell_{\dot{\mathbb{B}}^{1+\frac{1}{p^*}}_{p^*,1}}\norme{(\partial_xn,\partial_xV)}_{\dot{\mathbb{B}}^{\frac{1}{p}-1}_{p,1}}+\norme{(\partial_xn,\partial_xV)}^\ell_{\dot{\mathbb{B}}^{1-\frac{1}{p}}_{p,1}}\norme{\d_xV}^\ell_{\dot{\mathbb{B}}^{-\frac{1}{p^*}}_{p^*,1}}\biggr)\hfill\cr\hfill
+\int_0^t\|\d_xV\|_{L^\infty}\|G(n)\|^h_{\dot\B^{\frac32}_{2,1}}
+\int_0^t\biggl(\norme{(n,G(n))}^\ell_{\dot{\mathbb{B}}^{1+\frac{1}{p^*}}_{p^*,1}}\norme{\partial_xV}_{\dot{\mathbb{B}}^{\frac{1}{p}-1}_{p,1}}+\norme{\d_xV}^\ell_{\dot{\mathbb{B}}^{1-\frac{1}{p}}_{p,1}}\norme{\d_xG(n)}^\ell_{\dot{\mathbb{B}}^{-\frac{1}{p^*}}_{p^*,1}}\biggr)\cdotp}$$
\end{Prop}
\begin{proof}
We localize  System \eqref{CED4} by means of $\ddj,$ getting 
\begin{equation} \left\{ \begin{matrix}\partial_tn_j+\dot{S}_{j-1}V\partial_xn_j+\partial_xV_j+\dot{S}_{j-1}G(n)\partial_xV_j=R^1_j+ R'^1_j,\\ \partial_t
V_j+\dot{S}_{j-1}V\partial_xV_j+\partial_xn_j+V_j=R^2_j \end{matrix} \right.
\end{equation} where $$\begin{aligned} R^1_j&\triangleq \dot{S}_{j-1}V\partial_xn_j-\ddj(V\partial_xn),\qquad
R'^1_j\triangleq \dot{S}_{j-1}G(n)\partial_xV_j-\ddj(G(n)\partial_xV) \\ \andf
R^2_j&\triangleq\dot{S}_{j-1}V\partial_xV_j-\ddj(V\partial_xV).\end{aligned}$$
The only difference with  $(TM)$ is the appearance of $\dot{S}_{j-1}G(n)\partial_xV_j$ in the first equation and of the  commutator $R'^1_j$.
To handle the former term, one has to  add  a suitable weight 
in  the definition of the functional we used for $(TM)$:  for $j\geq J_0$ and $\eta=\eta(J_0)>0$ (to be chosen small enough), we set
\begin{equation*} 
\wt\cL_j^2\triangleq\int_\mathbb{R}(\partial_xn_j)^2+(1+\dot{S}_{j-1}G(n))(\partial_xV_j)^2 +\eta\int_{\mathbb{R}}V_j\partial_xn_j\cdotp
\end{equation*}
Differentiating in time this quantity and  performing several integration by parts yields:
 \begin{multline}\label{eq:Lj1} \frac{d}{dt}\wt\cL_j^2+\wt\cH_j^2 +\!\int_\R\d_x\dot S_{j-1}V\,\bigl((\d_xn_j)^2\!+\!(1+\dot S_{j-1}G(n))(\d_xV_j)^2\bigr)
 \\ -\int_\R\partial_x\dot S_{j-1}G(n)\,S_{j-1}V\,(\d_xV_j)^2
 =\int_\mathbb{R}(\partial_xV_j)^2\,\partial_t\dot{S}_{j-1}G(n)
 \\+2\int_\R\!\bigl(\d_x(R^1_j+R'^1_j)\,\d_xn_j+(1+S_{j-1}G(n))\d_xR_j^2\,\d_xV_j\bigr)
 +\eta\int_\R\bigl(\d_x(R^1_j+R'^1_j)\,V_j+ R_j^2\,\d_{x}n_j\bigr)\cdotp\end{multline}
 with $\displaystyle\wt\cH_j^2\triangleq \eta\|\d_xn_j\|_{L^2}^2+(2-\eta)\int_\R(1+\dot S_{j-1}G(n))(\d_xV_j)^2+\eta\int_\R V_j\d_xn_j.$
 \medbreak
 To continue, let us assume that 
 \begin{equation}\label{eq:smallnV}
 \|n\|_{L^\infty}+\|V\|_{L^\infty}\ll1\quad\hbox{on }\ [0,T].\end{equation}
 Then, since $G(0)=0,$ we have, using the mean value theorem 
 and the uniform boundedness of operator $\dot S_{j-1}$ in all Lebesgue spaces:
 $$\|\dot{S}_{j-1}G(n)\|_{L^\infty}\lesssim \|n\|_{L^\infty}\ll1,$$
 and thus, if $\eta$ is small enough,  
 \begin{equation}\label{eq:equivLj}
 \wt\cL_j^2\simeq \norme{(\partial_xn_j,\partial_xV_j)}_{L^2}^2\andf \wt\cH_j^2\simeq \norme{(\partial_xn_j,\partial_xV_j)}_{L^2}^2\quad
 \hbox{for all }\ j\geq J_0.
 \end{equation}
 Let us also observe that 
 $$\begin{aligned}
\partial_tG(n)&=G'(n)\partial_tn\\
&=-G'(n)(V\partial_xn+(1+G(n))\partial_xV).
\end{aligned}$$
Owing to  assumption \eqref{eq:smallnV} and to  the mean value theorem, we thus get
\begin{equation}\label{eq:dtG} \|\partial_tG(n)\|_{L^\infty} \lesssim \|\d_x V\|_{L^\infty} + \|V\|_{L^\infty}\|\d_xn\|_{L^\infty}.
\end{equation}
 Proceeding analogously, we obtain 
 \begin{equation}\label{eq:dxG}\norme{\d_xG(n)}_{L^\infty}\lesssim \|\d_xn\|_{L^\infty}.\end{equation}
Hence, from inequality \eqref{eq:Lj1} and \eqref{eq:equivLj}, we  get for some small enough $c$ 
and large enough $C,$
 \begin{equation}\label{eq:Lj2} \frac{d}{dt}\wt\cL_j^2+c\wt\cL_j^2\leq
 C\bigl(\|(\d_xV,\d_xn)\|_{L^\infty}\wt\cL_j + 2^j\|(R_j^1,R'^1_j,R^2_j)\|_{L^2}\bigr)\wt\cL_j
 \quad\hbox{for all }\ j\geq J_0.\end{equation}
At this point, taking advantage of  Lemma  \ref{SimpliCarre} yields
 \begin{equation}\label{eq:cLj}
 \wt\cL_j(t)+c\int_0^t\wt\cL_j \leq\wt\cL_j(0) +C \int_0^t\norme{(\partial_xV,\d_xn)}_{L^\infty}\cL_j
+C2^j\int_0^t\norme{(R^1_j,R'^1_j,R^2_j)}_{L^2}.\end{equation}
   Now,  multiplying \eqref{eq:cLj} by  $2^{\frac j2},$ using \eqref{eq:equivLj}
and summing up on $j\geq J_0$ gives us
\begin{multline}\label{eq:Egmn}
\norme{(n,V)(t)}^h_{\dot{\mathbb{B}}^{\frac{3}{2}}_{2,1}}+\int_0^t\norme{(n,V)}^h_{\dot{\mathbb{B}}^{\frac{3}{2}}_{2,1}}\lesssim\norme{(n_0,V_0)}^h_{\dot{\mathbb{B}}^{\frac{3}{2}}_{2,1}}\\+\int_0^t\norme{(\partial_xV,\d_xn)}_{L^\infty}\norme{(n,V)}^h_{\dot{\mathbb{B}}^{\frac{3}{2}}_{2,1}}+\int_0^t\sum_{j\geq J_0}2^{\frac{3j}{2}}\norme{(R^1_j,R'^1_j,R^2_j)}_{L^2}.
\end{multline}
 The terms $R_j^1$ and $R_j^2$ may be bounded exactly   as in the proof Proposition \ref{PropHfLp}.
 As regards    $R'^1_j,$  Lemma \ref{CP} gives us
 $$\displaylines{ \sum_{j\geq J_0} 2^{\frac32j}\|R'^1_j\|_{L^2} 
 \lesssim \|\d_xG(n)\|_{L^\infty}\|\d_x V\|^h_{\dot\B^{\frac12}_{2,1}}
 +\|\d_xV\|_{\dot\B^{\frac1p-1}_{p,1}}\|G(n)\|^\ell_{\dot\B^{1+\frac1{p^*}}_{p^*,1}} \hfill\cr\hfill
 +\|\d_xV\|_{L^\infty} \|G(n)\|^h_{\dot\B^{\frac32}_{2,1}}
 +\|\d_xV\|^\ell_{\dot\B^{1-\frac1p}_{p,1}}\|\d_xG(n)\|^\ell_{\dot\B^{-\frac1{p^*}}_{p^*,1}}.}
 $$
Using \eqref{eq:dxG} completes the proof of the proposition. \end{proof}
   

\subsubsection*{Global-in-time a priori estimate} We claim that granted with Inequalities \eqref{eq:EulerBF} and the above proposition,
we have,  whenever  Condition 
 \eqref{eq:smallnV} is satisfied on $[0,T],$
  \begin{equation}\label{eq:XpEuler}
  X_p(t) \lesssim  X_{p,0} + X_p^2(t) \quad\hbox{for all }\ t\in[0,T].\end{equation}
 Inequality  \eqref{eq:EulerBF} is   exactly the same as for $(TM).$
Hence,  the terms in $X_p(t)$ corresponding to the low frequencies of $(n,V)$ are  bounded by $X_p^2(t).$
Note also that $\|v\|_{L_t^2(\dot\B^{\frac1p}_{p,1})}$ may be bounded according to \eqref{eq:vL2}, and thus eventually 
by $X_p^2(t).$ 
\smallbreak
  In order to handle the high frequencies,  we shall proceed differently depending on whether
  $P(\rho)=\rho^\gamma/\gamma$ or $P$ is a general  pressure law with $P'(1)=1.$ 
  In fact, to handle  the latter case, we need to assume that $p=2.$

  \subsubsection*{1. Case $P(\rho)= \rho^\gamma/\gamma$ with $\gamma>0$}
 Then, $G(n)=(\gamma-1) n$ and  the inequality of  
  Proposition \ref{HfEuler} reduces to 
  $$\displaylines{
\norme{(n,V)(t)}^h_{\dot{\mathbb{B}}^{\frac{3}{2}}_{2,1}}+\int_0^t\norme{(n,V)}^h_{\dot{\mathbb{B}}^{\frac{3}{2}}_{2,1}}\lesssim\norme{(n_0,V_0)}^h_{\dot{\mathbb{B}}^{\frac{3}{2}}_{2,1}} + \int_0^t\norme{(\partial_xn,\partial_xV)}_{\dot{\mathbb{B}}^{\frac{1}{2}}_{2,1}}\norme{(n,V)}^h_{\dot{\mathbb{B}}^{\frac{3}{2}}_{2,1}}\hfill\cr\hfill
+\int_0^t\biggl(\norme{V}^\ell_{\dot{\mathbb{B}}^{1+\frac{1}{p^*}}_{p^*,1}}\norme{(\partial_xn,\partial_xV)}_{\dot{\mathbb{B}}^{\frac{1}{p}-1}_{p,1}}+\norme{(\partial_xn,\partial_xV)}^\ell_{\dot{\mathbb{B}}^{1-\frac{1}{p}}_{p,1}}\norme{\d_xV}^\ell_{\dot{\mathbb{B}}^{-\frac{1}{p^*}}_{p^*,1}}\biggr)
\hfill\cr\hfill
+\int_0^t\biggl(\norme{n}^\ell_{\dot{\mathbb{B}}^{1+\frac{1}{p^*}}_{p^*,1}}\norme{\partial_xV}_{\dot{\mathbb{B}}^{\frac{1}{p}-1}_{p,1}}+\norme{\partial_xV}^\ell_{\dot{\mathbb{B}}^{1-\frac{1}{p}}_{p,1}}
\norme{\d_xn}^\ell_{\dot{\mathbb{B}}^{-\frac{1}{p^*}}_{p^*,1}}\biggr)\cdotp}$$
Compared to our study of $(TM),$ only the last line is new. 
However, one can use the fact that 
$$\begin{aligned}
\int_0^t\norme{n}^\ell_{\dot{\mathbb{B}}^{1+\frac{1}{p^*}}_{p^*,1}}\norme{\partial_xV}_{\dot{\mathbb{B}}^{\frac{1}{p}-1}_{p,1}}&\lesssim \norme{n}^\ell_{L_t^2(\dot{\mathbb{B}}^{1+\frac{1}{p}}_{p,1})}
\norme{V}_{L_t^2(\dot{\mathbb{B}}^{\frac{1}{p}}_{p,1})}\\
\int_0^t\norme{\partial_xV}^\ell_{\dot{\mathbb{B}}^{1-\frac{1}{p}}_{p,1}}
\norme{\d_xn}^\ell_{\dot{\mathbb{B}}^{-\frac{1}{p^*}}_{p^*,1}}
&\lesssim\norme{V}^\ell_{L_t^1(\dot{\mathbb{B}}^{1+\frac{1}{p}}_{p,1})}
\norme{n}^\ell_{L_t^\infty(\dot{\mathbb{B}}^{\frac{1}{p}}_{p,1})}\qquad\hbox{as }\ 1-\frac1{p^*}\geq\frac1p\cdotp
\end{aligned}$$
The terms on the right may be bounded by $X^2_p(t).$ Hence we have \eqref{eq:XpEuler}.

    \subsubsection*{2. Case of a general pressure law with $P'(1)=1$} For $p=2,$ Proposition \ref{HfEuler}
    together with the embeddings $\dot\B^{\frac12}_{2,1}\hookrightarrow\dot\B^0_{\infty,1}\hookrightarrow L^\infty$ and 
     $\dot\B^{\frac32}_{2,1}\hookrightarrow\dot\B^1_{\infty,1}$
         give us 
$$\displaylines{
\norme{(n,V)(t)}^h_{\dot{\mathbb{B}}^{\frac{3}{2}}_{2,1}}+\int_0^t\norme{(n,V)}^h_{\dot{\mathbb{B}}^{\frac{3}{2}}_{2,1}}\lesssim\norme{(n_0,V_0)}^h_{\dot{\mathbb{B}}^{\frac{3}{2}}_{2,1}} + \int_0^t\norme{(\partial_xn,\partial_xV)}_{\dot{\mathbb{B}}^{\frac{1}{2}}_{2,1}}\norme{(n,V)}^h_{\dot{\mathbb{B}}^{\frac{3}{2}}_{2,1}}\hfill\cr\hfill
+\int_0^t\norme{V}^\ell_{\dot{\mathbb{B}}^{\frac32}_{2,1}}\norme{(n,V)}_{\dot{\mathbb{B}}^{\frac{1}{2}}_{2,1}}
+\int_0^t\|\d_xV\|_{\dot\B^{\frac12}_{2,1}}\|G(n)\|_{\dot\B^{\frac32}_{2,1}}
+\int_0^t\norme{(n,G(n))}^\ell_{\dot{\mathbb{B}}^{\frac32}_{2,1}}\norme{V}_{\dot{\mathbb{B}}^{\frac{1}{2}}_{2,1}}\cdotp}$$
Since, by Proposition \ref{Composition} and Cauchy-Schwarz inequality, we have
$$\begin{aligned}
\int_0^t\|\d_xV\|_{\dot\B^{\frac12}_{2,1}}\|G(n)\|_{\dot\B^{\frac32}_{2,1}}&\lesssim
\int_0^t\|\d_xV\|_{\dot\B^{\frac12}_{2,1}}\|n\|_{\dot\B^{\frac32}_{2,1}}
\lesssim \|\d_xV\|_{L_t^2(\dot\B^{\frac12}_{2,1})}\|n\|_{L_t^2(\dot\B^{\frac32}_{2,1})},\\
 \int_0^t\norme{(n,G(n))}^\ell_{\dot{\mathbb{B}}^{\frac32}_{2,1}}\norme{V}_{\dot{\mathbb{B}}^{\frac{1}{2}}_{2,1}}&\lesssim
  \norme{n}_{L_t^2(\dot{\mathbb{B}}^{\frac32}_{2,1})}\norme{V}_{L_t^2(\dot{\mathbb{B}}^{\frac{1}{2}}_{2,1})},
  \end{aligned}
  $$
  one can conclude that \eqref{eq:XpEuler} is satisfied. 

\subsubsection*{Uniqueness} As  for $(TM),$ we look 
at the system satisfied by the difference $(\dn,\dV):=(n_2-n_1,V_2-V_1)$ between two solutions, namely:
\begin{equation}\label{eq:uniqE1}\left\{\begin{array}{l} \d_t\dn+\d_x\dV+V_2\,\d_x\dn+G(n_2)\partial_x\dV
=-\dV\,\d_xn_1-(G(n_2)-G(n_1))\partial_xV_1,\\[1ex]
\d_t\dV +\dV +\d_x\dn +V_2\,\d_x\dV=-\dV\,\d_xV_1,\end{array}\right.\end{equation}
and estimate $(\dn,\dV)$ for all $T>0$ in the space $F_p(T)$ defined in \eqref{eq:FpT}.
Compared to the proof of uniqueness for $(TM)$ we have to handle the two terms containing the function $G.$  
\smallbreak
Let us first explain how to estimate the low frequencies.  We have to bound
the additional terms $G(n_2)\partial_x\dV$ and $(G(n_2)-G(n_1))\partial_xV_1$ in $\dot\B^{\frac2p-\frac12}_{p,1}.$
Now, according to  \eqref{eq:loi1} and \eqref{eq:loi2}, we have 
$$\begin{aligned}
\|(G(n_2)-G(n_1))\partial_xV_1\|^\ell_{\dot\B^{\frac2p-\frac12}_{p,1}}&\lesssim  \|G(n_2)-G(n_1)\|_{\dot\B^{\frac1p}_{p,1}}
\|\d_xV_1\|_{\dot\B^{\frac2p-\frac12}_{p,1}},\\
\|G(n_2)\partial_x\dV\|^\ell_{\dot\B^{\frac2p-\frac12}_{p,1}}&\lesssim \|n_2\|_{\dot\B^{\frac1p}_{p,1}\cap\dot\B^{\frac1p+1}_{p,1}}\|\d_x\dV\|_{\dot\B^{\frac2p-\frac32}_{p,1}}.
\end{aligned}
$$
From the relation  
\begin{equation}\label{eq:GG}
G(n_2)-G(n_1)= \dn\,\int_0^1 G'(n_1+\tau \dn)\,d\tau,\end{equation}
and  Propositions \ref{LP} and \ref{Composition}, we find out:
$$\|G(n_2)-G(n_1)\|_{\dot\B^{\frac1p}_{p,1}}\lesssim \|\dn\|_{\dot\B^{\frac1p}_{p,1}}.$$
Therefore, we eventually have
 \begin{multline}\label{eq:uniqE2} 
 \|(\dn,\dV)(t)\|^\ell_{\dot\B^{\frac2p-\frac12}_{p,1}}\lesssim \int_0^t
 \bigl(\|(\d_xn_1,\d_xV_1)\|_{\dot\B^{\frac2p-\frac12}_{p,1}}
 \|\dV\|_{\dot\B^{\frac1p}_{p,1}}
 +\|V_2\|_{\dot\B^{\frac1p}_{p,1}\cap\dot\B^{\frac1p+1}_{p,1}}\|(\dn,\dV)\|_{\dot\B^{\frac2p-\frac12}_{p,1}}\bigr)\\+\int_0^t
 \bigl(\|\d_xV_1\|_{\dot\B^{\frac2p-\frac12}_{p,1}}
 \|\dn\|_{\dot\B^{\frac1p}_{p,1}}
 +\|n_2\|_{\dot\B^{\frac1p}_{p,1}\cap\dot\B^{\frac1p+1}_{p,1}}\|\dV\|_{\dot\B^{\frac2p-\frac12}_{p,1}}\bigr)\cdotp
 \end{multline}
Let us next estimate   the high frequencies  of $(\dn,\dV)$  in $\dot\B^{\frac12}_{2,1}.$ Applying operator $\ddj$ to \eqref{eq:uniqE1}, we 
get for all $j\geq J_0,$ 
$$
\left\{\begin{array}{l} \d_t\dn_j+\d_x\dV_j+\dot S_{j-1}V_2\,\d_x\dn_j+\dot S_{j-1}G(n_2)\partial_x\dV_j
\\\hspace{5cm}=-\ddj\bigl(\dV\,\d_xn_1+(G(n_2)-G(n_1))\partial_xV_1\bigr)+\dR_j^1+\dR'^1_j,\\[1ex]
\d_t\dV_j +\dV_j +\d_x\dn_j +\dot S_{j-1}V_2\,\d_x\dV_j=-\ddj(\dV\,\d_xV_1)+\dR_j^2,\end{array}\right.$$
with $\dR_j^1\triangleq \dot S_{j-1}V_2\d_x\dn_j-\ddj(V_2\d_x\dn),$
$\dR'^1_j\triangleq \dot S_{j-1}G(n_2)\d_x\dV_j-\ddj(G(n_2)\d_x\dV)$
and $\dR_j^2\triangleq \dot S_{j-1}V_2\d_x\dV_j-\ddj(V_2\d_x\dV).$
\smallbreak
Arguing as in the proof of Proposition \ref{HfEuler}, we consider the functional 
$$
\int_\mathbb{R}(\partial_x\dn_j)^2+(1+\dot{S}_{j-1}G(n_2))(\partial_x\dV_j)^2 + \eta \int_{\mathbb{R}}\dV_j\,\partial_x\dn_j$$
and follow the computations therein, with regularity exponent $1/2$ instead of $3/2.$ We get 
 $$\displaylines{
\|(\dn,\dV)(t)\|^h_{\dot\B^{\frac12}_{2,1}}\!\lesssim\!
\int_0^t\|(\partial_xn_2,\d_xV_2)\|_{L^\infty}\|(\dn,\dV)\|^h_{\dot\B^{\frac12}_{2,1}}
+\int_0^t\!\sum_{j\geq J_0}\! 2^{\frac j2}\bigl(\|\dR_j^1\|_{L^2}+\|\dR'^1_j\|_{L^2}+\|\dR_j^2\|_{L^2}\bigr)
\hfill\cr\hfill
+\int_0^t\bigl( \|\dV\,\d_xn_1\|^h_{\dot\B^{\frac12}_{2,1}}+\|\dV\,\d_xV_1\|^h_{\dot\B^{\frac12}_{2,1}}+\|(G(n_2)\!-\!G(n_1))\d_xV_1\|^h_{\dot\B^{\frac12}_{2,1}}\bigr)\cdotp}
$$
The terms with $\dR^1_j$ and $\dR^2_j$ may be bounded as in the proof of uniqueness for $(TM).$
Regarding $\dR'^1_j,$ we use   Lemma \ref{CP} with $s=1/2,$ and get
 $$\displaylines{ \sum_{j\geq J_0} 2^{\frac j2}\|\dR'^1_j\|_{L^2} 
 \lesssim \|\d_xG(n_2)\|_{L^\infty}\|\dV\|^h_{\dot\B^{\frac12}_{2,1}}
 +\|\d_x\dV\|_{\dot\B^{\frac1p-1}_{p,1}}\|G(n_2)\|^\ell_{\dot\B^{1+\frac1{p^*}}_{p^*,1}} \hfill\cr\hfill
 +\|\d_x\dV\|_{L^\infty} \|G(n_2)\|^h_{\dot\B^{\frac12}_{2,1}}
 +\|\d_x\dV\|^\ell_{\dot\B^{1-\frac1p}_{p,1}}\|\d_xG(n_2)\|^\ell_{\dot\B^{-\frac1{p^*}}_{p^*,1}}.}$$
To continue the proof, we have two distinguish two cases depending on whether $P(\rho)=\rho^\gamma/\gamma$ and $2\leq p\leq 4,$ or $P$
is a general pressure law with $P'(1)=1,$ and $p=2.$
In the first case, we have $G(n)=(\gamma-1)n,$ so that 
$G(n_2)-G(n_1)=(\gamma-1)\dn.$ Now, in light of  \eqref{eq:prod4}, we have
$$\|\dn\,\d_xV_1\|^h_{\dot\B^{\frac12}_{2,1}}\lesssim\bigl(\|\dn\|^\ell_{\dot\B^{\frac2p-\frac12}_{p,1}}+\|\dn\|^h_{\dot\B^{\frac12}_{2,1}}\bigr)
\bigl(\|\d_xV_1\|^\ell_{\dot\B^{\frac1p-1}_{p,1}}+\|\d_xV_1\|_{\dot\B^{\frac12}_{2,1}}^h\bigr)\cdotp$$
As all the terms with $G(n_2)$ in the estimate for $R'^1_j$ are proportional to $n_2,$ we  arrive at
 \begin{multline}\label{eq:uniqE4}
 \|(\dn,\dV)(t)\|^h_{\dot\B^{\frac12}_{2,1}}\\\lesssim \int_0^t
 \bigl(\|(n_1,n_2,V_1,V_2)\|_{\dot\B^{\frac1p}_{p,1}}^\ell
 +\|(n_1,n_1,V_1,V_2)\|_{\dot\B^{\frac32}_{2,1}}^h+\|(\d_xn_1,\d_xV_1)\|_{\dot\B^{\frac2p-\frac12}_{p,1}}\bigr)\|(\dn,\dV)\|_{F_p}.\end{multline}
 In the case $p=2$ with $P'(1)=1,$ then one may proceed essentially as in the proof of Proposition  \ref{HfEuler} to bound
 the terms with $G(n_2)$  in the estimate for $R'^1_j,$ and one can use 
Proposition \ref{Composition} combined with product laws  and Relation \eqref{eq:GG} to eventually arrive at
$$\|(G(n_2)-G(n_1))\partial_xV_1\|_{\dot\B^{\frac12}_{2,1}}\lesssim  \|\dn\|_{\dot\B^{\frac12}_{2,1}}\|\partial_xV_1\|_{\dot\B^{\frac12}_{2,1}}.$$
Consequently, \eqref{eq:uniqE4} still holds true. 
\smallbreak
In all cases, putting \eqref{eq:uniqE2} and \eqref{eq:uniqE4} together yields
 $$\displaylines{\|(\dn,\dV)\|_{F_p(t)}\hfill\cr\hfill\leq C\int_0^t\bigl(\|(n_1,n_2,V_1,V_2)\|_{\dot\B^{\frac1p}_{p,1}}^\ell
 +\|(n_1,n_2,V_1,V_2)\|_{\dot\B^{\frac32}_{2,1}}^h+\|(\d_xn_1,\d_xV_1)\|_{\dot\B^{\frac2p-\frac12}_{p,1}}\bigr)\|(\dn,\dV)\|_{F_p},}$$
 and using Gronwall lemma  completes the proof of  uniqueness.
 \subsubsection*{Decay estimates} Here, we  assume that  $p=2$ and follow the same  approach as for $(TM).$
 \subsubsection*{Step 1: estimating the solution in $\dot\B^{-\sigma_1}_{2,\infty}$} This is 
 only a matter of  handling  the additional term $G(n)\partial_xv.$ 
 Applying $\ddj$ to the system satisfied by $(n,V)$ yields
 $$
\left\{\begin{array}{l}
\d_tn_j+\d_xV_j=-v\d_xn_j-G(n)\partial_xV+[V,\ddj]\d_xn,\\
\d_tV_j+\d_xn_j+V_j=-V\d_xV_j+[v,\ddj]\d_xV.
\end{array}\right.$$ 
So, considering  $G(n)\partial_xV$ as a source term, we get
 $$\displaylines{\|(n_j,V_j)(t)\|_{L^2} + \int_0^t\|V_j\|_{L^2} \leq \|(n_j,V_j)(0)\|_{L^2}
+\int_0^t\|\d_xV\|_{L^\infty}\|(n_j,V_j)\|_{L^2} \hfill\cr\hfill+\int_0^t\|[V,\ddj]\d_xn\|_{L^2}
+\int_0^t\|[V,\ddj]\d_xV\|_{L^2}+\norme{\ddj(G(n)\partial_xV)}_{L^2}\norme{n_j}_{L^2}.}$$
 We have
 $$\|G(n)\partial_xV\|_{\dot\B^{-\sigma_1}_{2,\infty}}\leq \|G(n)\|_{\dot\B^{-\sigma_1}_{2,\infty}}\|\partial_xV\|_{\dot\B^{\frac{1}{2}}_{2,1}}. $$ 
 In order to bound $G(n)$ in $\dot\B^{-\sigma_1}_{2,\infty},$ one cannot readily use Proposition \ref{Composition} since  $-\sigma_1$  may be negative.
 However, from Taylor formula, we know  that there exists a smooth function $H$ vanishing at $0$ such that 
 $$ G(n)= G'(0)\,n + H(n)\,n.$$
 Hence, combining product and composition estimates gives
 $$\|G(n)\|_{\dot\B^{-\sigma_1}_{2,\infty}}
 \lesssim \|n\|_{\dot\B^{-\sigma_1}_{2,\infty}}\bigl(1 + \|n\|_{\dot\B^{\frac12}_{2,1}}\bigr)\cdotp$$
 In the regime we consider, $\|n\|_{\dot\B^{\frac12}_{2,1}}$ is small. Hence we conclude that  
 $$\|(n,V)(t)\|_{\dot\B^{-\sigma_1}_{2,\infty}}\leq \|(n_0,V_0)\|_{\dot\B^{-\sigma_1}_{2,\infty}}
+ C \int_0^t\|\d_xV\|_{\dot\B^{\frac12}_{2,1}} \|(n,V)\|_{\dot\B^{-\sigma_1}_{2,\infty}},$$
which ensures after using Gronwall lemma and the bound of $\|\d_xV\|_{L^1(\dot\B^{\frac12}_{2,1})}$ in terms  of $X_{2,0},$ that 
$$\forall t\in\R^+,\;     \|(n,V)(t)\|_{\dot\B^{-\sigma_1}_{2,\infty}}\leq C\|(n_0,V_0)\|_{\dot\B^{-\sigma_1}_{2,\infty}}.$$
 \subsubsection*{Step 2: Lyapunov functional}  
We aim at  exhibiting a Lyapunov functional that is equivalent to $\|(n,V)\|_{\dot\B^{\frac12}_{2,1}\cap\dot\B^{\frac32}_{2,1}}.$ 
The high frequency part of the solution has already been treated efficiently with  $\wt\cL_j.$ To bound  the low frequency part,
consider  the evolution equation for   $z\triangleq V+\d_xn$:
$$\d_tz+V\d_xz +z=-\d^2_{xx}V-\d_xV\d_xn-\d_x(G(n)\d_xV).$$
Following the computations we did for $(TM)$ leads to 
$$\displaylines{
\|z(t)\|^\ell_{\dot\B^{\frac12}_{2,1}}+\int_0^t\|z\|^\ell_{\dot\B^{\frac12}_{2,1}}\leq \|z_0\|^\ell_{\dot\B^{\frac12}_{2,1}}+
\int_0^t\|\d^2_{xx}V\|_{\dot\B^{\frac12}_{2,1}}^\ell\hfill\cr\hfill+
C\int_0^t\|\d_xV\|_{\dot\B^{\frac12}_{2,1}}\|z\|_{\dot\B^{\frac12}_{2,1}} + C\int_0^t\|\d_xV\|_{\dot\B^{\frac12}_{2,1}}\|\d_xn\|_{\dot\B^{\frac12}_{2,1}} 
+\int_0^t\|\d_x(G(n)\d_xV)\|^\ell_{\dot\B^{\frac12}_{2,1}}.}
$$
The last term may be  bounded. by  $\|G(n)\d_xV\|_{\dot\B^{\frac12}_{2,1}}.$  Then, using  Propositions \ref{LP} and \ref{Composition}, one
 ends up with
\begin{equation}\label{eq:Z}
\|z(t)\|^\ell_{\dot\B^{\frac12}_{2,1}}+\int_0^t\|z\|^\ell_{\dot\B^{\frac12}_{2,1}}\leq \|z_0\|^\ell_{\dot\B^{\frac12}_{2,1}}+
\int_0^t\|\d^2_{xx}V\|_{\dot\B^{\frac12}_{2,1}}^\ell+
C\int_0^t\|\d_xV\|_{\dot\B^{\frac12}_{2,1}}\|(z,n,\d_xn)\|_{\dot\B^{\frac12}_{2,1}}.\end{equation}
Next, using the fact that
$$
\d_tn+V\d_xn-\d^2_{xx}n=-G(n)\d_xV-\d_xz,$$
we get
\begin{equation}\label{eq:n} 
\|n(t)\|^\ell_{\dot\B^{\frac12}_{2,1}}+\int_0^t\|n\|^\ell_{\dot\B^{\frac52}_{2,1}}\leq \|n_0\|^\ell_{\dot\B^{\frac12}_{2,1}}+
\int_0^t\|\d_{x}z\|_{\dot\B^{\frac12}_{2,1}}^\ell+
C\int_0^t\|\d_xV\|_{\dot\B^{\frac12}_{2,1}}\|n\|_{\dot\B^{\frac12}_{2,1}}.\end{equation}
The high frequency part of the solution may be bounded according to \eqref{eq:cLj}.
Hence, setting
$$\wt\cL\triangleq \sum_{j\leq J_0} 2^{\frac j2} \|(\ddj n,\ddj z)\|_{L^2}  + \sum_{j\geq J_0} 2^{\frac j2}\wt\cL_j
\andf \wt\cH\triangleq \|V+\d_xn\|_{\dot\B^{\frac12}_{2,1}}^\ell+\|V\|_{\dot\B^{\frac32}_{2,1}}^h+\|n\|_{\dot\B^{\frac52}_{2,1}}^\ell
+\|n\|_{\dot\B^{\frac32}_{2,1}}^h$$  and bounding  $R_j^1,$ $R'^1_j$ and $R_j^2$ as in the proof of Proposition 
\ref{HfEuler}, we discover that 
if taking $J_0$  negative enough, then all the linear terms in \eqref{eq:Z} and \eqref{eq:n} may be absorbed by $\wt\cH,$ so that 
we have for some suitably small positive $c,$
\begin{equation}\label{eq:wtcL0}
\wt\cL(t)+c\int_0^t\wt\cH\leq\wt\cL(0) +C\int_0^t\|\d_xV\|_{\dot\B^{\frac12}_{2,1}}\cL 
+C\int_0^t\|n\|_{\dot\B^{\frac32}_{2,1}}\|V\|_{\dot\B^{\frac12}_{2,1}}.
\end{equation}
Above, we used that $\wt\cL\simeq \|(n,V)\|_{\dot\B^{\frac12}_{2,1}\cap\dot\B^{\frac32}_{2,1}}$
and  that $\wt\cH\gtrsim \|\d_xV\|_{\dot\B^{\frac12}_{2,1}}.$  Now,  since furthermore
$\wt\cH\gtrsim \|z\|_{\dot\B^{\frac12}_{2,1}}$ and $\wt\cL\gtrsim \|z\|_{\dot\B^{\frac12}_{2,1}},$  one may write
$$\begin{aligned}
 \|n\|_{\dot\B^{\frac32}_{2,1}}\|V\|_{\dot\B^{\frac12}_{2,1}}&\leq  \|n\|_{\dot\B^{\frac32}_{2,1}}^2
 + \|n\|_{\dot\B^{\frac32}_{2,1}}\|z\|_{\dot\B^{\frac12}_{2,1}}\\
 &\lesssim \|n\|_{\dot\B^{\frac12}_{2,1}}^\ell \|n\|_{\dot\B^{\frac52}_{2,1}}^\ell
  + (\|n\|_{\dot\B^{\frac32}_{2,1}}^h)^2
  + \|n\|^\ell_{\dot\B^{\frac32}_{2,1}}\|z\|_{\dot\B^{\frac12}_{2,1}}
  + \|n\|^h_{\dot\B^{\frac32}_{2,1}}\|z\|_{\dot\B^{\frac12}_{2,1}}\\&\lesssim
  \wt\cL\wt\cH+\wt\cL\wt\cH+\wt\cL\wt\cH+\wt\cH\wt\cL.\end{aligned}
 $$
 Hence, if $\wt\cL(0)$ is small enough  then, combining \eqref{eq:wtcL0} with a bootstrap argument 
  yields
 $$\wt\cL(t)+\frac c2\int_0^t\wt\cH\leq\wt\cL(0)\quad\hbox{for all }\ t\geq0.$$
  \subsubsection*{Step 3: Proof of decay estimates}
  {}From this point,  one can repeat word for word the proof  of decay estimates for  the low frequencies 
 of the solutions to  (TM). 
 \medbreak
 For the high frequencies,  starting from \eqref{eq:Lj2}, using  Lemma \ref{CP}
 and integrating gives 
 $$
\norme{(n,V)(t)}^h_{\dot{\mathbb{B}}^{\frac{3}{2}}_{2,1}}\lesssim e^{-ct}\norme{(n_0,V_0)}^h_{\dot{\mathbb{B}}^{\frac{3}{2}}_{2,1}}+\int_0^te^{-c(t-\tau)}\bigl(\norme{V}_{\dot{\mathbb{B}}^{\frac{3}{2}}_{2,1}}\norme{(n,V)}_{\dot{\mathbb{B}}^{\frac{3}{2}}_{2,1}}
+\norme{V}_{\dot{\mathbb{B}}^{\frac{1}{2}}_{2,1}}\norme{n}_{\dot{\mathbb{B}}^{\frac{3}{2}}_{2,1}}\bigr)\cdotp
$$
Compared to   \eqref{eq:decayhf},   there is one more term.   However,  as for $(TM),$ Steps 1 and 2 together imply  that 
$$\|(n,V)(t)\|_{\dot\B^{\frac12}_{2,1}\cap \dot\B^{\frac32}_{2,1}}\lesssim \langle t\rangle^{-\alpha_1}.$$  
Hence, one may easily conclude that 
$$\|(n,V)(t)\|_{\dot\B^{\frac32}_{2,1}}^h\lesssim \langle t\rangle^{-2\alpha_1}.$$  
This completes the proof of the theorem (up to the proof of 
existence, which is totally analogous as for  $(TM)$).  
  \end{proof}


\section{A more general 1D model}\label{s:general}
In this section, we consider a more general class of one dimensional systems, namely
\begin{equation} \left\{ \begin{matrix}\partial_tu + \alpha\partial_xv+V^1\partial_xu+W^1\partial_xv=0,\\ \partial_tv+\beta\partial_xu+V^2\partial_xu+W^2\partial_xv+\lambda v+\kappa\lambda v^q=0 \end{matrix} \right. \label{SystGen1D}
\end{equation}
where\footnote{In the case $q=3$ and $\kappa>0,$   $\kappa v^q$ is a classical representation of a drag term.}
 $\kappa$ is a real parameter, $q\geq 2,$  an integer, $V^1=V^1(v)$ and $V^2=V^2(v)$ are smooth functions vanishing at $0$, 
$W^1=W^1(u,v)$ and $W^2=W^2(u,v)$ are smooth functions vanishing at $(0,0),$ and $\alpha$, $\beta$, $\lambda$ are strictly positive constants.
\begin{Thm} \label{ThmGenL2}  Let the data $(u_0,v_0)$ satisfy the assumptions of  Theorem \ref{ThmExistLp} 
with  $J_\lambda\triangleq\left\lfloor\rm log_2\lambda\right\rfloor$ and $p=2.$  Then, System \eqref{SystGen1D} admits a unique global solution $(u,v)$ verifying the same properties as the solution therein. Furthermore,  Corollary \ref{CorLambda} and Theorem \ref{ThmDecayTM} hold true. \end{Thm}
\begin{Rmq} If $V^1,$ $V^2,$ $W^1$ and $W^2$ are `general' smooth functions, then it is unlikely  that a $L^p$ theory 
may be worked out.  We need a very specific structure of the nonlinear terms
in order  that the $L^p$ estimates of the low frequencies fit with the $L^2$ regularity of the high frequencies.\end{Rmq}
\begin{Rmq}
We do not know how to handle  terms like $u\partial_xu$ in any equations of the system (this is the reason why we assumed that $V^1$ and $V^2$ only depend on $v$).  In fact, although the system is locally well-posed if $V^1$ and $V^2$ also depend on $u,$ 
the time integrability of $u$ is not good enough for  global estimates. 
\end{Rmq}
\subsubsection*{Elements of proof}
We just explain how to find a Lyapunov function and to control the norm in $\dot\B^{-\sigma_1}_{2,\infty}$ of  a smooth solution 
$(u,v)$ of  \eqref{SystGen1D} on $[0,T],$ in terms of
the data. Proving existence and uniqueness 
is essentially the same as for the systems we treated before (uniqueness is  easier somehow
since we assumed $p=2$). Although the system under consideration is no longer 
symmetric if $\alpha+W^1\not=\beta+V^2$, it is symmetrizable (see  \cite[Chap. 10]{Benzoni-Serre}).
\smallbreak
Note that performing a suitable rescaling  reduces
our problem to the case 
\begin{equation}\label{eq:abl} \alpha=\beta=\lambda=1.\end{equation}  Indeed, if we set
$$(u,v)(t,x) =(\sqrt\alpha\,\tilde u,\sqrt\beta\,\tilde v)\Bigl(\lambda t,\frac{\lambda}{\sqrt{\alpha\beta}}\,x\Bigr),$$
then $(u,v)$ satisfies \eqref{SystGen1D} if and only if $(\tilde u,\tilde v)$ satisfies
a similar system with  \eqref{eq:abl},  parameter $\kappa \beta^{\frac{q-1}2}$  and slightly modified functions $V_1,$ $V_2,$ $W_1$ 
and $W_2$ (the modification depending only on $\alpha$ and $\beta$).  
So we will  assume \eqref{eq:abl} in the rest of this section.
\subsubsection*{A priori estimates}
We adapt  the method we used for $(TM)$ \emph{in the case $p=2$}.
 The terms $V^1\d_xu$ and $W^2\d_xv$  are  a slight generalization of $v\d_xu$ and $v\d_xv$ 
 and may be treated similarly. To handle  $W^1\d_xv$ and  $V^2\d_xu,$ we need to introduce suitable weights 
 in the definition  of the Lyapunov.   Finally,  $v^q$  may be seen as a harmless nonlinear source term. 
\smallbreak 
Let us start the computations : we assume that we are given a smooth function $(u,v)$  of \eqref{SystGen1D} on  some time interval 
$[0,T]$  such that for some suitably small $\eta>0,$
\begin{equation}\label{eq:smalluv} \sup_{t\in[0,T]}\|(u,v)(t)\|_{\dot\B^{\frac12}_{2,1}} \leq \eta,\end{equation} 
and, still denoting $u_j=\ddj u$ and $v_j=\ddj v,$  we set for all $j\in\Z,$  
$$
\cL_j\triangleq 
\biggl(\norme{(u_j,v_j)}^2_{L^2} +\int_{\mathbb{R}}v_j\partial_xu_j +\int_\mathbb{R}(1+V^2)(\partial_xu_j)^2+\int_\R(1+W^1)(\partial_xv_j)^2\biggr)^{1/2}\cdotp$$
We shall use repeatedly that \eqref{eq:smalluv} implies that 
\begin{equation}\label{eq:small}
\sup_{t\in[0,T]}\max\bigl(\|u(t)\|_{L^\infty},\,\|v(t)\|_{L^\infty},\,\|V^1(t)\|_{L^\infty},\,\|V^2(t)\|_{L^\infty},\,\|W^1(t)\|_{L^\infty},\,\|W^2(t)\|_{L^\infty}\bigr)\ll1,
\end{equation}
which  in particular entails that 
\begin{equation}\label{eq:equivLLj}
\cL_j\simeq\|(u_j,v_j,\d_xu_j,\d_xv_j)\|_{L^2}.
\end{equation}
Now,  applying $\ddj$ to  \eqref{SystGen1D}-\eqref{eq:abl} yields for all $j\in\Z,$ 
\begin{equation}\label{eq:ujvj}
 \left\{ \begin{aligned}&\partial_tu_j +(1+W^1)\partial_xv_j+V^1\partial_xu_j=R_j^1,\\ 
 &\partial_tv_j+(1+V^2)\partial_xu_j+W^2\partial_xv_j+v_j=R_j^2-\kappa\ddj(v^q) 
\end{aligned}\right. \end{equation}
 with
 $$ R_j^1\triangleq[V^1,\ddj]\d_xu+[W^1,\ddj]\d_xv\andf R_j^2\triangleq[V^2,\ddj]\d_xu+[W^2,\ddj]\d_xv.$$
  In order to compute the time derivative of $\cL_j^2,$ we need the following obvious identities:
  $$\displaylines{  \frac12\frac d{dt}\|(u_j,v_j)\|_{L^2}^2+\|v_j\|_{L^2}^2-\frac12\int_\R\bigl((u_j)^2\d_xV^1+(v_j)^2\d_xW^2\bigr)
  +\int_\R\bigl(W^1 u_j\d_xv_j+V^2v_j\d_xu_j\bigr) \hfill\cr\hfill=\int_\R \bigl(R_j^1 u_j  + R_j^2 v_j-\kappa\bigl(\ddj v^q\bigr)v_j\bigr),}$$
  $$\displaylines{\frac d{dt}\int_\R v_j\d_xu_j +\|\d_xu_j\|_{L^2}^2-\|\d_xv_j\|_{L^2}^2 +\int_\R v_j\d_xu_j\hfill\cr\hfill
  +\int_\R\bigl((W^2-V^1)\d_xu_j\,\d_xv_j +V^2(\d_xu_j)^2-W^1(\d_xv_j)^2\bigr)
  =\int_\R\bigl(R_j^2\d_xu_j-R_j^1\d_xv_j-\kappa\bigl(\ddj v^q\bigr)\d_xu_j\bigr)\cdotp}$$
  $$\displaylines{\frac12\frac d{dt}\int_\mathbb{R}(1+V^2)(\partial_xu_j)^2+\int_\R(1+V^2)\d_xu_j\d_x(V^1\d_xu_j)
  +\int_\R(1+V^2)\d_xu_j\d_x((1+W^1)\d_xv_j)\hfill\cr\hfill
  =\int_\R(1+V^2)\d_xu_j\d_xR_j^1 +\frac12\int_\R \d_tV^2(\d_xu_j)^2.}$$
    $$\displaylines{\frac12\frac d{dt}\int_\mathbb{R}(1+W^1)(\partial_xv_j)^2+\int_\R(1+W^1)\d_xv_j\d_x(W^2\d_xv_j)
  +\int_\R(1+W^1)\d_xv_j\d_x((1+V^2)\d_xu_j)\hfill\cr\hfill
+\int_\R(1+W^1)(\d_xv_j)^2  =\int_\R(1+W^1)\d_xv_j\bigl(\d_xR_j^2-\kappa\d_x\ddj v^q\bigr) +\frac12\int_\R \d_tW^1(\d_xv_j)^2.}$$
The fundamental observation that justifies our using those very  weights in the definition of $\cL_j$ is that 
the third integrals in the last  two relations  compensate. Consequently, 
denoting $$\cH_j^2\triangleq 
\|v_j\|_{L^2}^2+\frac12\int_\R v_j\d_x u_j+\frac12\|\d_xu_j\|_{L^2}^2 +\int_\R \Bigl(W^1+\frac12\Bigr)(\d_xv_j)^2,$$
and using the fact that
$$\int_\R V^2 v_j\d_xu_j = -\int_\R V^2 u_j\d_x v_j-\int_\R u_j v_j \d_xV^2,$$
we arrive at
$$\displaylines{\frac12\frac d{dt}\cL_j^2+\cH_j^2 = \frac12\int((v_j)^2\d_xW^2+(u_j)^2\d_xV^1)
+\int_\R u_j v_j \d_xV^2+\int_\R(V^2-W^1)u_j\d_xv_j
\hfill\cr\hfill+\frac12\int_\R\bigl((V^1-W^2)\d_xu_j\,\d_xv_j -V^2(\d_xu_j)^2+W^1(\d_xv_j)^2\bigr)
+\frac12\int_\R(\d_xu_j)^2\bigl( V^1\d_xV^2-(1+V^2)\d_xV^1\bigr)
\hfill\cr\hfill+\frac12\int_\R(\d_xv_j)^2\bigl(W^2\d_xW^1-(1+W^1)\d_xW^2\bigr)
+\frac12\int_\R\bigl((\d_xu_j)^2\d_tV^2+(\d_xv_j)^2\d_tW^1\bigr)
\hfill\cr\hfill+\int_\R\bigl(u_j-\frac12\d_xv_j\bigr)R_j^1+\int_\R\bigl(v_j+\frac12\d_xu_j\bigr)R_j^2
+\int_\R\bigl((1+V^2)\d_xu_j\d_xR_j^1
+(1+W^1)\d_xv_j\d_xR_j^2\bigr)\hfill\cr\hfill
-\kappa\int_\R\Bigl(\bigl(v_j+\frac12\d_xu_j\bigr)\ddj v^q+(1+W^1)\d_x\ddj v^q\d_xv_j\Bigr)\cdotp}$$
Since 
$$\d_t V^2=-(V^2)'\bigl((1+V^2)\d_xu+W^2\d_xv+v+\kappa v^q\bigr),$$ remembering \eqref{eq:small}, we have
$$
\|\d_t V^2\|_{L^\infty} \lesssim \|\d_xu\|_{L^\infty} +\|(u,v)\|_{L^\infty} \|\d_x v\|_{L^\infty} +\|v\|_{L^\infty}$$
and, similarly, 
$$\|\d_tW^1\|_{L^\infty} \lesssim \|\d_xu\|_{L^\infty} + \|\d_x v\|_{L^\infty} +\|v\|_{L^\infty}.$$
Observe also that 
$$\|\d_xV^i\|_{L^\infty}\lesssim \|\d_xv\|_{L^\infty}\andf
\|\d_xW^i\|_{L^\infty}\lesssim\|(\d_xu,\d_xv)\|_{L^\infty}\quad\hbox{for }\ i=1,2,$$
whence, in particular
$$\int_\R u_j v_j \d_xV^2 \lesssim \|\d_xv\|_{L^\infty}\|u_j\|_{L^2}\|v_j\|_{L^2}.$$
Therefore,
$$\displaylines{
\frac12\frac d{dt}\cL_j^2+\cH_j^2\lesssim \|(u,v)\|_{L^\infty}\|\d_xv_j\|_{L^2}\|(u_j,\d_xu_j,\d_xv_j)\|_{L^2}
+\|v\|_{L^\infty}\|\d_xu_j\|^2_{L^2}\hfill\cr\hfill+\|\d_xu\|_{L^\infty}\|(v_j,\d_xu_j,\d_xv_j)\|_{L^2}^2
+\|\d_xv\|_{L^\infty} \|(u_j,v_j,\d_xu_j,\d_xv_j)\|_{L^2}^2\hfill\cr\hfill
+ \|(R_j^1,R_j^2)\|_{L^2} \|(u_j,v_j,\d_xu_j,\d_xv_j)\|_{L^2}
+  \|(\d_xR_j^1,\d_xR_j^2)\|_{L^2} \|(\d_xu_j,\d_xv_j)\|_{L^2}\hfill\cr\hfill
+\|(v_j,\d_xu_j)\|_{L^2}\|\ddj v^q\|_{L^2}+\|\d_xv_j\|_{L^2}\|\d_x\ddj v^q\|_{L^2}.}
$$
Then, remembering  \eqref{eq:equivLLj} and  using lemma \ref{SimpliCarre}, we discover that for all $j\in\Z$ 
and $t\in[0,T],$
\begin{multline}\label{eq:Ljjj}
\cL_j(t) +c\min(1,2^{2j})\int_0^t\cL_j\leq \cL_j(0)
\\+C\int_0^t \bigl(\|\d_x v\|_{L^\infty}\cL_j+\|\d_xu\|_{L^\infty}\|v_j\|_{L^2}
+ \|(v,\d_x u)\|_{L^\infty}\|\d_xu_j\|_{L^2}+\|(u,v,\d_xu)\|_{L^\infty}\|\d_xv_j\|_{L^2}\bigr)\\
 + \int_0^t\|(\ddj v^q,\d_x\ddj v^q)\|_{L^2}
+C\int_0^t\|(R_j^1,R_j^2,\d_xR_j^1,\d_xR_j^2)\|_{L^2}.
\end{multline}
To bound the commutator terms, let us use  \eqref{eq:com1} that yields 
$$\|R_j^1\|_{L^2}\lesssim c_j 2^{-\frac j2}\bigl(\|\d_xV^1\|_{\dot\B^{\frac12}_{2,1}}\|u\|_{\dot\B^{\frac12}_{2,1}}
+ \|\d_xW^1\|_{\dot\B^{\frac12}_{2,1}}\|v\|_{\dot\B^{\frac12}_{2,1}}\bigr)\ \hbox{ for all }\ j\in\Z.$$
Clearly, since $v$ is small in $\dot\B^{\frac12}_{2,1},$ $V^1=V^1(v)$   and $V^1(0)=0,$ Proposition \ref{Composition} entails that
\begin{equation}\label{eq:dxV1}\|\d_xV^1\|_{\dot\B^{\frac12}_{2,1}}\lesssim\|\d_xv\|_{\dot\B^{\frac12}_{2,1}}.\end{equation}
In order to bound the term with $W^1,$ we use the fact that
there exist two smooth functions $G=G(u,v)$ and $H=H(u,v)$ vanishing at $(0,0)$ and such that
$$
\d_xW^1=\d_uW^1(0,0)\d_xu + \d_vW^1(0,0)\d_x v + G(u,v)\d_xu+H(u,v)\d_xv.$$
Consequently, using the stability of the space $\dot\B^{\frac12}_{2,1}$ by product
and results in \cite[Section 5.5]{RS} for bounding $G(u,v)$ and $H(u,v),$ we get
\begin{equation}\label{eq:dxW1}\|\d_xW^1\|_{\dot\B^{\frac12}_{2,1}} \lesssim \|(\d_xu,\d_xv)\|_{\dot\B^{\frac12}_{2,1}}
\bigl(1+\|(u,v)\|_{\dot\B^{\frac12}_{2,1}}\bigr)\cdotp\end{equation}
So finally, remembering \eqref{eq:smalluv}, we have
\begin{equation}\label{eq:Rj1}
\|R_j^1\|_{L^2}\lesssim c_j 2^{-\frac j2}\bigl(\|u\|_{\dot\B^{\frac12}_{2,1}}\|\d_xv\|_{\dot\B^{\frac12}_{2,1}}
+\|v\|_{\dot\B^{\frac12}_{2,1}} \|(\d_xu,\d_xv)\|_{\dot\B^{\frac12}_{2,1}}\bigr)\cdotp
\end{equation}
Bounding $R_j^2$ works exactly the same. 
Next, in light of \eqref{eq:com2}, we have
$$\|\d_xR_j^1\|_{L^2}\lesssim c_j 2^{-\frac j2}\bigl(\|\d_xV^1\|_{\dot\B^{\frac12}_{2,1}}\|\d_xu\|_{\dot\B^{\frac12}_{2,1}}+ \|\d_xW^1\|_{\dot\B^{\frac12}_{2,1}}\|\d_xv\|_{\dot\B^{\frac12}_{2,1}}\bigr),$$
and a similar inequality for $\d_xR_j^2.$ Hence repeating the above arguments for bounding
$\d_xV^1,$ $\d_xV^2,$ $\d_xW^1$ and $\d_xW^2,$ we end up with 
\begin{equation}\label{eq:dxRj1}
\|\d_xR_j^i\|_{L^2}\lesssim c_j 2^{-\frac j2}\|\d_xv\|_{\dot\B^{\frac12}_{2,1}}
 \|(\d_xu,\d_xv)\|_{\dot\B^{\frac12}_{2,1}},\quad i=1,2.
\end{equation}
Now, reverting to  \eqref{eq:Ljjj}, using the embedding $\dot\B^{\frac12}_{2,1}\hookrightarrow L^\infty$
and that 
\begin{equation}\label{eq:equivH}
\cH_j\simeq \|(v_j,\d_xu_j,\d_xv_j)\|_{L^2}
\end{equation}
 we get,  denoting 
$$\cL\triangleq \sum_{j\in\Z} 2^{\frac j2}\cL_j,$$
two  positive constants $c$ and $C$ such that   
\begin{multline}\label{eq:v1}
\cL(t) +c\sum_j \min(1,2^{2j}) 2^{\frac j2} \int_0^t \cL_j
\leq \cL(0) +C\int_0^t \|\d_xv\|_{\dot\B^{\frac12}_{2,1}}\cL\\ +C\int_0^t\|v\|_{\dot\B^{\frac12}_{2,1}}
\|\d_xu\|_{\dot\B^{\frac12}_{2,1}} +C\int_0^t\|\d_xu\|_{\dot\B^{\frac12}_{2,1}}^2+C\int_0^t\|(v^q,\d_xv^q)\|_{\dot\B^{\frac12}_{2,1}}.\end{multline}
As for $(TM)$, we need better properties of integrability for $v$ in order to close 
the above estimate. The situation  is a bit  more  complex  since the second line  above was not present. 
Nevertheless, it  is still  possible  to exhibit a control of $z\triangleq v+\d_xu$ in $L^1(\R_+;\dot\B^{\frac12}_{2,1})$ 
(which, as we saw in \eqref{eq:vL2} yields a bound for  $v$ in $L^2(\R_+;\dot\B^{\frac12}_{2,1})$).
Indeed, we have 
\begin{equation}\label{eq:dtz}
\d_tz+ z+V^1\d_xz=(V^1-W^2)\d_xv-V^2\d_xu-\d^2_{xx}v-\d_xV^1\,\d_xu-\d_x(W^1\d_xv)-\kappa v^q\end{equation}
which, as in the proof of \eqref{eq:zzz} leads to
$$\displaylines{\|z(t)\|^\ell_{\dot\B^{\frac12}_{2,1}}+\int_0^t\|z\|^\ell_{\dot\B^{\frac12}_{2,1}}\leq \|z_0\|^\ell_{\dot\B^{\frac12}_{2,1}}
+\int_0^t\|\d^2_{xx}v\|_{\dot\B^{\frac12}_{2,1}}^\ell+C\int_0^t\|\d_xV^1\|_{\dot\B^{\frac12}_{2,1}}\|z\|_{\dot\B^{\frac12}_{2,1}}
\hfill\cr\hfill+\int_0^t\|(V^1-W^2)\d_xv+V^2\d_xu+\d_xV^1\,\d_xu\|_{\dot\B^{\frac12}_{2,1}}^\ell
+\int_0^t\|\d_x(W^1\d_xv)\|_{\dot\B^{\frac12}_{2,1}}^\ell+\kappa\int_0^t\|v^q\|_{\dot\B^{\frac12}_{2,1}}^\ell.}$$
Using product and composition estimates and remembering \eqref{eq:smalluv}, we get 
$$\begin{aligned}
\|(V^1-W^2)\d_xv\|_{\dot\B^{\frac12}_{2,1}}&\lesssim\|(u,v)\|_{\dot\B^{\frac12}_{2,1}}\|\d_xv\|_{\dot\B^{\frac12}_{2,1}},\\
\|V^2\d_xu\|_{\dot\B^{\frac12}_{2,1}}&\lesssim\|v\|_{\dot\B^{\frac12}_{2,1}}\|\d_xu\|_{\dot\B^{\frac12}_{2,1}},\\
\|\d_xV^1\,\d_xu\|_{\dot\B^{\frac12}_{2,1}}&\lesssim\|\d_xv\|_{\dot\B^{\frac12}_{2,1}}\|\d_xu\|_{\dot\B^{\frac12}_{2,1}}.
\end{aligned}$$
Since only low frequencies are involved, we have
$$
\|\d_x(W^1\d_xv)\|_{\dot\B^{\frac12}_{2,1}}^\ell\lesssim \|W^1\d_xv\|_{\dot\B^{\frac12}_{2,1}}
\lesssim\|(u,v)\|_{\dot\B^{\frac12}_{2,1}}\|\d_xv\|_{\dot\B^{\frac12}_{2,1}}.
$$
Hence, using also \eqref{eq:dxV1}, we get 
\begin{multline}\label{eq:v2}
\|z(t)\|_{\dot\B^{\frac12}_{2,1}}+\int_0^t\|z\|_{\dot\B^{\frac12}_{2,1}}\leq \|z_0\|_{\dot\B^{\frac12}_{2,1}}+\int_0^t\|\d^2_{xx}v\|_{\dot\B^{\frac12}_{2,1}}^\ell\\
+C\int_0^t\|\d_xv\|_{\dot\B^{\frac12}_{2,1}}\|(u,v,z,\d_xu)\|_{\dot\B^{\frac12}_{2,1}}
+C\int_0^t \|v\|_{\dot\B^{\frac12}_{2,1}}\|\d_xu\|_{\dot\B^{\frac12}_{2,1}}+\int_0^t\|v^q\|_{\dot\B^{\frac12}_{2,1}}.
\end{multline}
In order to close the estimates for the solution, it suffices to add up \eqref{eq:v1} to 
$\varepsilon\cdot$\eqref{eq:v2}  with suitably small $\varepsilon.$  More precisely,  setting 
$$\wt\cL\triangleq \cL +\varepsilon \|z\|_{\dot\B^{\frac12}_{2,1}}^\ell\andf 
\wt\cH\triangleq c\sum_{j\in\Z}\min(1,2^{2j})2^{\frac j2}\cL_j +\varepsilon  \|z\|_{\dot\B^{\frac12}_{2,1}}^\ell,$$
we get for all $t\in[0,T]$ if $\varepsilon$ has been chosen small enough,  
\begin{multline}\label{eq:wtcL}\wt\cL(t) +\frac12\int_0^t\wt\cH\leq\wt\cL(0)+C\int_0^t
\bigl(\|\d_xv\|_{\dot\B^{\frac12}_{2,1}}\wt\cL+ \|\d_xu\|_{\dot\B^{\frac12}_{2,1}}^2
+ \|v\|_{\dot\B^{\frac12}_{2,1}}\|\d_xu\|_{\dot\B^{\frac12}_{2,1}}\bigr)\\+C\int_0^t\|(v^q,\d_xv^q)\|_{\dot\B^{\frac32}_{2,1}}.\end{multline}
Let us emphasize that 
$$
\wt\cL\simeq \|(u,v,\d_xu,\d_xv)\|_{\dot\B^{\frac12}_{2,1}} \andf
\wt\cH\simeq  \|u\|_{\dot\B^{\frac52}_{2,1}}^\ell + \|u\|_{\dot\B^{\frac32}_{2,1}}^h + \|v+\d_xu\|^\ell_{\dot\B^{\frac12}_{2,1}}
+\|v\|_{\dot\B^{\frac32}_{2,1}}^h.$$
Hence in particular,  we have $\|\d_xv\|_{\dot\B^{\frac12}_{2,1}}\lesssim\wt\cH$ and, as explained
in the previous section (just replace $n$ by $u$ and $V$ by $v$),   
\begin{equation}\label{eq:justabove}  \|v\|_{\dot\B^{\frac12}_{2,1}}\|\d_xu\|_{\dot\B^{\frac12}_{2,1}}\lesssim \wt\cL\wt\cH.\end{equation}
Furthermore, 
$$\begin{aligned} \|\d_xu\|_{\dot\B^{\frac12}_{2,1}}^2&\lesssim
 (\|\d_xu\|^\ell_{\dot\B^{\frac12}_{2,1}})^2 + 
 (\|\d_xu\|_{\dot\B^{\frac12}_{2,1}}^h)^2\\
 &\lesssim  \|u\|_{\dot\B^{\frac12}_{2,1}}^\ell \|u\|_{\dot\B^{\frac52}_{2,1}}^\ell
 + (\|u\|_{\dot\B^{\frac32}_{2,1}}^h)^2 \lesssim\wt\cL\wt\cH.\end{aligned} $$
Finally,   Lemma \ref{LP} and \eqref{eq:justabove}  together  ensure that 
$$\begin{aligned} \norme{v^q}_{\dot{\mathbb{B}}^{\frac{1}{2}}_{2,1}}&\lesssim  \norme{v}^q_{\dot{\mathbb{B}}^{\frac{1}{2}}_{2,1}}\\
&\lesssim \|v\|_{\dot{\mathbb{B}}^{\frac{1}{2}}_{2,1}}^{q-1}\bigl(\|z\|_{\dot\B^{\frac12}_{2,1}}+\|\d_xu\|_{\dot\B^{\frac12}_{2,1}}\bigr)\\
&\lesssim \wt\cL^{q-1}\wt\cH +\wt\cL^{q-2}\|v\|_{\dot\B^{\frac12}_{2,1}}\|\d_xu\|_{\dot\B^{\frac12}_{2,1}}\\
&\lesssim   \wt\cL^{q-1}\wt\cH
\end{aligned}
$$
and 
$$\norme{\d_xv^q}_{\dot{\mathbb{B}}^{\frac{1}{2}}_{2,1}}\lesssim
    \norme{\d_xv}_{\dot{\mathbb{B}}^{\frac{1}{2}}_{2,1}}  \norme{v}^{q-1}_{\dot{\mathbb{B}}^{\frac{1}{2}}_{2,1}}\lesssim   \wt\cL^{q-1}\wt\cH.
 $$ 
Consequently, Inequality \eqref{eq:wtcL} reduces to 
$$\wt\cL(t) +\frac12\int_0^t\wt\cH\leq\wt\cL(0)+ C\int_0^t(\wt\cL+\wt\cL^{q-1})\wt\cH.$$
Now, applying a bootstrap argument, one may conclude that 
there exists a small constant $\eta$ such that if $\wt\cL(0)\leq\eta,$ then
\begin{equation}\label{eq:final}
\forall t\in[0,T],\; \wt\cL(t) +\frac14\int_0^t\wt\cH\leq\wt\cL(0).\end{equation}
This gives the desired control on the norm of the solution and, in addition, that $\wt\cL$ is a Lyapunov functional. 

\subsubsection*{Decay estimates} Granted with a Lyapunov functional that has the same properties 
as in the previous sections, in order to get the whole family of decay estimates, 
 it suffices to establish a uniform in time  bound in  $\dot\B^{-\sigma_1}_{2,\infty}$ for the solution.  
 The starting point is  that, for all $j\in\Z,$ 
$$\left\{
\begin{aligned}&\d_tu_j+\d_xv_j+V^1\d_xu_j=[V^1,\ddj]\d_x-\ddj(W^1\d_xv),\\
&\d_tv_j+\d_xu_j+v_j=-\ddj(W^2\d_xv)-\ddj(V^2\d_xu)-\kappa \ddj v^q.\end{aligned}\right.$$
Applying an energy method, using Lemma \ref{SimpliCarre} and  Inequality \eqref{eq:com3} eventually delivers: 
$$\displaylines{\|(u,v)(t)\|_{\dot\B^{-\sigma_1}_{2,\infty}} \leq \|(u_0,v_0)\|_{\dot\B^{-\sigma_1}_{2,\infty}}
+C\int_0^t \|\d_xV^1\|_{\dot\B^{\frac12}_{2,1}}\|(u,v)\|_{\dot\B^{-\sigma_1}_{2,\infty}}
 \hfill\cr\hfill + \int_0^t\bigl(\|W^1\d_xv\|_{\dot\B^{-\sigma_1}_{2,\infty}}+  \|W^2\d_xv\|_{\dot\B^{-\sigma_1}_{2,\infty}}
+ \|V^2\d_xu\|_{\dot\B^{-\sigma_1}_{2,\infty}} + \kappa \|v^q \|_{\dot\B^{-\sigma_1}_{2,\infty}}\bigr)\cdotp}$$ 
Using for $i=1,2,$  the decomposition 
$$W^i(u,v)=\bigl(\d_uW^i(0,0) + G^i(u,v)\bigr) u + \bigl(\d_vW^i(0,0) + H^i(u,v)\bigr) v$$
where $G^i$ and $H^i$ are smooth functions vanishing at $(0,0),$ we get thanks to results in \cite[Section 5.5]{RS}
and  Proposition \ref{LP} 
$$  \|W^i\d_xv\|_{\dot\B^{-\sigma_1}_{2,\infty}}\lesssim \|(u,v)\|_{\dot\B^{-\sigma_1}_{2,\infty}} \|\d_xv\|_{\dot\B^{\frac12}_{2,1}}.$$
Proposition \ref{LP} also implies that 
$$ \|v^q \|_{\dot\B^{-\sigma_1}_{2,\infty}}\lesssim  \|v^2 \|_{\dot\B^{-\sigma_1}_{2,\infty}}\|v\|_{\dot\B^{\frac12}_{2,1}}^{q-2}.$$
In order to estimate the term with $v^2,$ we use that $v=z-\d_xu$ and get the decomposition:
$$v^2=  v^h (v+v^\ell)  + z^\ell (v^\ell-\d_xu^\ell) +(\d_xu^\ell)^2.$$
By Proposition \ref{LP} and interpolation, we thus have
$$ \begin{aligned} \|v^2 \|_{\dot\B^{-\sigma_1}_{2,\infty}}&\lesssim \|v^h\|_{\dot\B^{\frac12}_{2,1}} \|v+v^\ell\|_{\dot\B^{-\sigma_1}_{2,\infty}} 
+ \bigl(\|v^\ell \|_{\dot\B^{-\sigma_1}_{2,\infty}} +\|\d_xu^\ell \|_{\dot\B^{-\sigma_1}_{2,\infty}}\bigr)\|z^\ell\|_{\dot\B^{\frac12}_{2,1}}
+ \|\d_xu^\ell\|_{\dot\B^{\frac12(\frac12-\sigma_1)}_{2,1}}^2\\
&\lesssim\|v\|^h_{\dot\B^{\frac32}_{2,1}} \|v\|_{\dot\B^{-\sigma_1}_{2,\infty}} 
+ \|(u,v) \|_{\dot\B^{-\sigma_1}_{2,\infty}} \|z\|^\ell_{\dot\B^{\frac12}_{2,1}}
+ \|u\|^\ell_{\dot\B^{-\sigma_1}_{2,\infty}} \|u\|^\ell_{\dot\B^{\frac52}_{2,1}}.\end{aligned}
$$
Hence we have 
$$ \|v^q \|_{\dot\B^{-\sigma_1}_{2,\infty}}\lesssim \|v\|_{\dot\B^{\frac12}_{2,1}}^{q-2} \, \|(u,v) \|_{\dot\B^{-\sigma_1}_{2,\infty}}\,\wt\cH.$$
Finally, using the decomposition 
$$V^2(v)\d_xu= V^2(v)\d_x u^h + V^2(z)\d_x u^\ell - \biggl(\int_0^1 V^2(z-\tau\d_xu)\biggr) \Bigl(\d_xu^\ell\d_x u^\ell
+\d_xu^\ell\d_x u^h\Bigr),$$
we get by similar computations that 
$$ \|V^2\d_xu\|_{\dot\B^{-\sigma_1}_{2,\infty}} \lesssim  \|(u,v) \|_{\dot\B^{-\sigma_1}_{2,\infty}}\,\wt\cH.$$
Therefore, in the end, we get 
$$\|(u,v)(t)\|_{\dot\B^{-\sigma_1}_{2,\infty}} \leq \|(u_0,v_0)\|_{\dot\B^{-\sigma_1}_{2,\infty}}
+C\int_0^t \wt\cH \bigl(1+   \|v\|_{\dot\B^{\frac12}_{2,1}}^{q-2}\bigr)  \|(u,v) \|_{\dot\B^{-\sigma_1}_{2,\infty}},$$
which, combined with \eqref{eq:final} and Gronwall lemma implies that 
$$\sup_{t\in[0,T]} \|(u,v)(t)\|_{\dot\B^{-\sigma_1}_{2,\infty}} \lesssim  \|(u_0,v_0)\|_{\dot\B^{-\sigma_1}_{2,\infty}}.$$
At this stage, completing the proof of decay estimates is left to the reader. \qed



\appendix
\section{}

Here we gather a few technical results that have been used repeatedly in the paper. 
The first one is a rather standard lemma pertaining 
to some differential inequality.
\begin{Lemme}\label{SimpliCarre}
Let  $p\geq 1$ and $X : [0,T]\to \mathbb{R}^+$ be a continuous function such that $X^p$ is differentiable
almost everywhere. We assume that there exists 
 a constant $B\geq 0$ and  a measurable function $A : [0,T]\to \mathbb{R}^+$ 
such that 
 $$\frac{1}{p}\frac{d}{dt}X^p+BX^p\leq AX^{p-1}\quad\hbox{a.e.  on }\ [0,T].$$ 
 Then, for all $t\in[0,T],$ we have
$$X(t)+B\int_0^tX\leq X_0+\int_0^tA.$$
\end{Lemme}
\begin{Proof}
The case $p=1$ being obvious,  assume that $1<p<\infty$. 
Then, we set $X_\varepsilon\triangleq(X^p+\varepsilon^p)^{1/p}$ for $\varepsilon>0,$ and observe that
$$\frac{1}{p}\frac{d}{dt}X_\varepsilon^p+BX_\varepsilon^p\leq AX^{p-1}_\varepsilon+B\varepsilon^p\quad\hbox{a.e.  on }\ [0,T]. $$
Dividing both sides by  the positive function $X_\varepsilon^{p-1}$, we get
$$\frac{d}{dt}X_\varepsilon+BX_\varepsilon\leq A+B\varepsilon\left(\frac{\varepsilon}{X_\varepsilon}\right)^{p-1},$$
whence, as $\varepsilon/X_\varepsilon\leq 1$, 
$$\frac{d}{dt}X_\varepsilon+BX_\varepsilon\leq A+B\varepsilon. $$
Then,  integrating in time and taking the limit as $\varepsilon$ tends to $0$ yields the desired inequality. 
\end{Proof}
The following result from \cite{Da01} has been used in the proof of Proposition \ref{APLP}. 
\begin{Prop}\label{p:bernstein} If $\text{Supp}(\mathcal{F}f)\subset\{\xi\in\mathbb{R}^d : R_1\lambda\leq|\xi|\leq R_2\lambda\}$ for some $0<R_1<R_2$ then, there exists $c=c(d,R_1,R_2)>0$ such that for all $p\in [2,\infty[$, we have
$$
c\lambda^2\left(\frac{p-1}{p}\right)\int_{\mathbb{R}^d}|f|^p \leq (p-1)\int_{\mathbb{R}^d}|\nabla f|^2|f|^{p-2}
= -\int_{\mathbb{R}^d}f\Delta f|f|^{p-2}.
$$\end{Prop}
The proof of the following  inequality may be found in e.g. \cite[Chap. 2]{HJR}.
\begin{Lemme} \label{Commutateur1}
Let $1\leq p,q,r\leq\infty$ be such that $\frac{1}{p}+\frac{1}{q}=\frac{1}{r}\cdotp$ 
Let $a$ be a function with gradient  in $L^p$ and $b,$ a function in $L^q.$ 
There exists a constant $C$ such that 
$$\norme{[\dot{\Delta}_j,a]b}_{L^r}\leq C2^{-j}\norme{\nabla a}_{L^q}\norme{b}_{L^p}
\quad\hbox{for all }\ j\in\Z.$$
\end{Lemme}
The following estimates are   proved in  \cite[Chap. 2]{HJR} and \cite{RLG}, respectively.
\begin{Prop}\label{C1} Assume that $d=1$ and that  $1\leq p\leq\infty.$   The following inequalities hold: 
\begin{itemize}
\item  If $s\in \left]-\min\left(\frac{1}{p},\frac{1}{p'}\right),\frac{1}{p}+1\right]$, then
 \begin{equation}\label{eq:com1}
2^{js}\norme{[w,\dot{\Delta}_j]\d_xv}_{L^p}\leq Cc_j\norme{\d_x w}_{\dot{\mathbb{B}}^{\frac{1}{p}}_{p,1}}\norme{v}_{\dot{\mathbb{B}}^{s}_{p,1}}  \with\sum_{j\in\mathbb{Z}}c_j=1.
\end{equation}
\item  If $s\in \left[-\min\left(\frac{1}{p},\frac{1}{p'}\right),\frac{1}{p}+1\right[$, then
\begin{equation}\label{eq:com3} 
\sup_{j\in\Z} 2^{js}\|[w,\ddj]\d_xv\|_{L^p}\leq C\|\d_xw\|_{\dot\B^{\frac1p}_{p,1}} \|v\|_{\dot\B^{s}_{p,\infty}}.\end{equation}
\item If  $s\in \left]-1-\min(\frac{1}{p},\frac{1}{p'}),\frac{1}{p}\right]$, then we have
 \begin{equation}\label{eq:com2}
  \norme{\d_x([w,\dot{\Delta}_j]v)}_{L^p}\leq Cc_j2^{-js}\norme{\d_x w}_{\dot{\mathbb{B}}^{\frac{1}{p}}_{p,1}}\norme{v}_{\dot{\mathbb{B}}^{s}_{p,1}} \with\sum_{j\in\mathbb{Z}}c_j=1.
  \end{equation}
  \end{itemize}
 \end{Prop}
The following product laws in Besov spaces have been used several times. 
 \begin{Prop} \label{LP} Let $(s,p,r)\in ]0,\infty[\times[1,\infty]^2.$ Then, 
 $\dot{\mathbb{B}}^{s}_{p,r}\cap L^\infty$ is an algebra and we have
\begin{equation}\label{eq:prod1}
\norme{ab}_{\dot{\mathbb{B}}^{s}_{p,r}}\leq C\bigl(\norme{a}_{L^\infty}\norme{b}_{\dot{\mathbb{B}}^{s}_{p,r}}+\norme{a}_{\dot{\mathbb{B}}^{s}_{p,r}}\norme{b}_{L^\infty}\bigr)\cdotp
\end{equation}
If, furthermore, $-\min(d/p,d/p')<s\leq d/p,$ then the following inequality holds:
\begin{equation}\label{eq:prod2}
\|ab\|_{\dot\B^{s}_{p,1}}\leq C\|a\|_{\dot\B^{\frac dp}_{p,1}}\|b\|_{\dot\B^{s}_{p,1}}.
\end{equation}
We have, if $-\min(d/p,d/p')<s\leq d/p+1,$ 
\begin{equation}\label{eq:prod3}
\|a\,b\|^\ell_{\dot\B^{s}_{p,1}}\lesssim  
\|a\|_{L^\infty\cap \dot\B^{\frac1p+1}_{p,1}}\,\|b\|_{\dot\B^{s-1}_{p,1}}.\end{equation}
In the case $d=1$ and $2\leq p\leq 4,$ we have 
\begin{equation}\label{eq:prod4}
\|ab\|_{\dot\B^{\frac12}_{2,1}}\lesssim 
\bigl(\|a\|_{\dot\B^{\frac1p-1}_{p,1}}^\ell+\|a\|^h_{\dot\B^{\frac12}_{2,1}}\bigr)
\bigl(\|b\|^\ell_{\dot\B^{\frac2p-\frac12}_{p,1}}+\|b\|_{\dot\B^{\frac12}_{2,1}}^h\bigr)\cdotp
\end{equation}
\end{Prop}
\begin{proof} The first two inequalities are direct consequences of the results 
stated in \cite[Chap. 2]{HJR}. 
To prove the third one, we need  the  following so-called Bony decomposition for the product 
of two tempered distributions $a$ and $b$ (whenever it is defined):
$$ab=T_ab+T'_ba\with T_ab\triangleq\sum_{j\in\Z}\dot S_{j-1}a\,\ddj b\andf
T'_ba\triangleq \sum_{j\in\Z}\dot S_{j+2}b\,\ddj a.$$
Now, using Bernstein inequality and 
the results of continuity for $T$ and $T'$ stated in  \cite[Chap. 2]{HJR}, we may write:
$$\|T_ab\|_{\dot\B^s_{p,1}}^\ell\lesssim \|T_ab\|_{\dot\B^{s-1}_{p,1}}
\lesssim \|a\|_{L^\infty} \|b\|_{\dot\B^{s-1}_{p,1}}$$
and, provided, $s-1\leq d/p$ and $s>-\min(d/p,d/p'),$ 
$$\|T'_ba\|_{\dot\B^s_{p,1}}\lesssim \|a\|_{\dot\B^{\frac dp+1}_{p,1}}
 \|b\|_{\dot\B^{s-1}_{p,1}}.$$
 This gives \eqref{eq:prod3}.
 \medbreak
 For proving \eqref{eq:prod4}, we combine Bony's decomposition 
 and decomposition of $a$ and $b$ in low and high frequencies, writing
 $$ ab=T'_ab^\ell+T'_ab^h+T_{b^\ell}a^\ell+T_{b}a^h+T_{b^h}a^\ell.$$
 All the terms in the right-hand side, except for the last one, may be bounded
 by means of the standard results of continuity for operators $T$ and $T'$
   (see again \cite[Chap. 2]{HJR}). Setting $p^*=2p/(p-2),$ we get:
 $$\begin{aligned}
 \|T'_ab^\ell\|_{\dot\B^{\frac12}_{2,1}}&\lesssim\|a\|_{\dot\B^{\frac1p-1}_{p,1}}\|b^\ell\|_{\dot\B^{\frac1{p^*}+1}_{p^*,1}}\lesssim \|a\|_{\dot\B^{\frac1p-1}_{p,1}}\|b\|^\ell_{\dot\B^{\frac1p+1}_{p,1}}, \\
 \|T'_ab^h\|_{\dot\B^{\frac12}_{2,1}}&\lesssim\|a\|_{L^\infty}\|b\|^h_{\dot\B^{\frac12}_{2,1}},\\
 \|T_{b^\ell}a^\ell\|_{\dot\B^{\frac12}_{2,1}}& \lesssim \|b^\ell\|_{L^{p^*}}\|a^\ell\|_{\dot\B^{\frac12}_{p,1}},\\
 \|T_{b}a^h\|_{\dot\B^{\frac12}_{2,1}}&\lesssim\|b\|_{L^\infty}\|a^h\|_{\dot\B^{\frac12}_{2,1}}.
\end{aligned}$$ 
Finally, since $a^\ell=\dot S_{J_0+1}a$ and $b^h=({\rm Id}-\dot S_{J_0+1})b,$ we see that 
$$
T_{b^h} a^\ell=  \dot S_{J_0}b^h\,\dot\Delta_{J_0+1} a^\ell.$$
Consequently, 
$$\|T_{b^h} a^\ell\|_{\dot\B^{\frac12}_{2,1}}\lesssim
\|\dot\Delta_{J_0+1}a^\ell\|_{L^\infty} \|\dot S_{J_0}b^h\|_{L^2}\lesssim \|a\|_{L^\infty} 
 \|b\|^h_{\dot\B^{\frac12}_{2,1}}.$$
Adding up this latter inequality to the previous ones gives
$$
\|ab\|_{\dot\B^{\frac12}_{2,1}}\lesssim  \|a\|_{L^\infty}\|b\|^h_{\dot\B^{\frac12}_{2,1}}
+ \|a\|_{\dot\B^{\frac1p-1}_{p,1}}\|b\|^\ell_{\dot\B^{\frac1p+1}_{p,1}}
+\|b\|_{L^\infty}\|a^h\|_{\dot\B^{\frac12}_{2,1}}
+ \|b^\ell\|_{L^{p^*}}\|a^\ell\|_{\dot\B^{\frac12}_{p,1}}.$$
Then, using  Bernstein inequality, $2/p-1/2\leq1/p$ and the embeddings 
   $\dot\B^{\frac12}_{2,1}\hookrightarrow L^\infty$ and
   $\dot\B^{\frac2p-\frac12}_{p,1}\hookrightarrow L^{p^*}$
completes the proof of \eqref{eq:prod4}. 
\end{proof}
The following result for composition in Besov spaces may be found in \cite{HJR}. 
\begin{Prop}\label{Composition}
Let $f$ be a function in $\mathcal{C}^\infty(\mathbb{R})$ such that $f(0)=0$. Let $(s_1,s_2)\in]0,\infty[^2$ and \\$(p_1,p_2,r_1,r_2)\in[1,\infty]^4$. We assume that $s_1<\frac{d}{p_1}$ or that $s_1=\frac{d}{p_1}$ and $r_1=1$.

Then, for every real-valued function $u$ in $\dot{\mathbb{B}}^{s_1}_{p_1,r_1}\cap\dot{\mathbb{B}}^{s_2}_{p_2,r_2}\cap L^\infty$, the function $f\circ u$ belongs to $\dot{\mathbb{B}}^{s_1}_{p_1,r_1}\cap\dot{\mathbb{B}}^{s_2}_{p_2,r_2}\cap L^\infty,$ and we have in particular
$$\norme{f\circ u}_{\dot{\mathbb{B}}^{s_k}_{p_k,r_k}}\leq C\left(f',\norme{u}_{L^\infty}\right)\norme{u}_{\dot{\mathbb{B}}^{s_k}_{p_k,r_k}}\quad\hbox{for}\  k\in\{1,2\}.$$
\end{Prop}
The following result is the key to Theorem  \ref{ThmExistLp} in the general case. 
\begin{Lemme} \label{CP} Assume that $d=1.$ Let $p\in[2,4]$ and $s\in[1/2,3/2[.$  Define $p^*\triangleq 2p/(p-2).$  For all $j\in\Z,$ 
denote $\cR_j\triangleq \dot S_{j-1}w\,\d_x\ddj z-\ddj(w\,\d_xz).$

There exists  a constant $C$ depending only on  the threshold 
number $J_0$ between low and high frequencies and on $s,$   such that 
 $$\displaylines{
\sum_{j\geq J_0}\left(2^{js}\norme{\cR_j}_{L^2}\right)\leq
C\Bigl(\norme{\partial_xw}_{L^\infty}\norme{\d_xz}^h_{\dot{\mathbb{B}}^{s-1}_{2,1}}
+ \norme{\partial_xz}_{\dot{\mathbb{B}}^{\frac{1}{p}-1}_{p,1}}\norme{w}^\ell_{\dot{\mathbb{B}}^{1+\frac{1}{p^*}}_{p^*,1}}\hfill\cr\hfill
+\norme{\partial_xz}_{\dot{\mathbb{B}}^{s-\frac32}_{\infty,\infty}}\norme{w}^h_{\dot{\mathbb{B}}^s_{2,1}}
+ \norme{\partial_xz}^\ell_{\dot{\mathbb{B}}^{s-\frac12-\frac{1}{p}}_{p,1}}\norme{\d_xw}^\ell_{\dot{\mathbb{B}}^{-\frac{1}{p^*}}_{p^*,1}}\Bigr)\cdotp}$$
In the case $s=3/2,$  we have $$\displaylines{
\sum_{j\geq J_0}\left(2^{j\frac{3}{2}}\norme{\cR_j}_{L^2}\right)\leq
C\Bigl(\norme{\partial_xw}_{L^\infty}\norme{\d_xz}^h_{\dot{\mathbb{B}}^{\frac12}_{2,1}}
+ \norme{\partial_xz}_{\dot{\mathbb{B}}^{\frac{1}{p}-1}_{p,1}}\norme{w}^\ell_{\dot{\mathbb{B}}^{1+\frac{1}{p^*}}_{p^*,1}}\hfill\cr\hfill
+\norme{\partial_xz}_{L^\infty}\norme{w}^h_{\dot{\mathbb{B}}^{\frac32}_{2,1}}
+ \norme{\partial_xz}^\ell_{\dot{\mathbb{B}}^{1-\frac{1}{p}}_{p,1}}\norme{\d_xw}^\ell_{\dot{\mathbb{B}}^{-\frac{1}{p^*}}_{p^*,1}}\Bigr)\cdotp}$$
\end{Lemme}
\begin{proof}
{}From Bony decomposition recalled above, we deduce that
$$\begin{aligned}\cR_j&=-\ddj(T'_{\d_xz}w)-\sum_{|j'-j|\leq4} [\ddj,\dot S_{j'-1}w]\d_x\dot\Delta_{j'}z
- \sum_{|j'-j|\leq 1}\left(\dot{S}_{j'-1}w-\dot{S}_{j-1}w\right)\dot{\Delta}_j\dot{\Delta}_{j'}\partial_xz\\
&=\mathcal{R}^1_j + \mathcal{R}^2_j +  \mathcal{R}^3_j. 
\end{aligned}$$
To estimate $\mathcal{R}^1_j$, we decompose $w$ into low and high 
frequencies, getting
$$T'_{\d_xz}w=T'_{\d_xz}w^\ell+T'_{\d_xz}w^h.$$
Because $1/p+1/p^*=1/2,$ the classical results of continuity for paraproduct and remainder operators
(see e.g. \cite[Chap. 2]{HJR}) ensure that
$$\|T'_{\d_xz}w^\ell\|_{\dot\B^{\frac12}_{2,1}}\lesssim
\|\d_xz\|_{\dot\B^{\frac1p-1}_{p,1}}\|w^\ell\|_{\dot\B^{\frac1{p^*}+1}_{p^*,1}},$$
and we have
$$\|T'_{\d_xz}w^h\|_{\dot\B^{s}_{2,1}}\lesssim
\|\d_xz\|_{\dot\B^{s-\frac32}_{\infty,\infty}}\|w^h\|_{\dot\B^{\frac32}_{2,1}}\ \hbox{if } 0<s<3/2,\andf
\|T'_{\d_xz}w^h\|_{\dot\B^{\frac32}_{2,1}}\lesssim
\|\d_xz\|_{L^\infty}\|w^h\|_{\dot\B^{\frac32}_{2,1}}.
$$
Observing that $T'_{\d_xz}w^\ell$ contains only low frequencies so that
its norm in $\dot\B^{s}_{2,1}$ is controlled by its norm 
in  $\dot\B^{\frac12}_{2,1}$  if $s\geq1/2,$ we deduce that 
\begin{eqnarray}\label{eq:Rj1a}
\qquad\sum_{j\in\Z}\left(2^{js}\norme{\cR_j^1}_{L^2}\right)&\!\!\!\leq\!\!\!& C\Bigl(
\norme{\partial_xz}_{\dot{\mathbb{B}}^{\frac{1}{p}-1}_{p,1}}\norme{w^\ell}_{\dot{\mathbb{B}}^{1+\frac{1}{p^*}}_{p^*,1}}
+\norme{\partial_xz}_{\dot\B^{s-\frac32}_{\infty,\infty}}\norme{w^h}_{\dot{\mathbb{B}}^{\frac32}_{2,1}}\Bigr)
\  \hbox{ if }\ 1/2\leq s<3/2,\\\label{eq:Rj1b}
\qquad\sum_{j\in\Z}\left(2^{\frac32j}\norme{\cR_j^1}_{L^2}\right)&\!\!\!\leq\!\!\!& C\Bigl(
\norme{\partial_xz}_{\dot{\mathbb{B}}^{\frac{1}{p}-1}_{p,1}}\norme{w^\ell}_{\dot{\mathbb{B}}^{1+\frac{1}{p^*}}_{p^*,1}}
+\norme{\partial_xz}_{L^\infty}\norme{w^h}_{\dot{\mathbb{B}}^{\frac32}_{2,1}}\Bigr)\cdotp
\end{eqnarray}
Next, taking advantage of Lemma \ref{Commutateur1}, we see that
if $j'\geq J_0$ and $|j-j'|\leq 4,$ then we have 
$$
2^{js} \|[\ddj,\dot S_{j'-1}w]\d_x\dot\Delta_{j'}z\|_{L^2}
\lesssim \|\d_x\dot S_{j'-1}w\|_{L^\infty} \, 2^{j'(s-1)}\|\d_x\dot\Delta_{j'}z\|_{L^2}
$$
while, if $j'<J_0,$ $j\geq J_0$ and $|j-j'|\leq 4,$ 
$$2^{js} \|[\ddj,\dot S_{j'-1}w]\d_x\dot\Delta_{j'}z\|_{L^2}
\lesssim 2^{-\frac{j'}{p^*}}\|\d_x\dot S_{j'-1}w\|_{L^{p^*}} \, 2^{j'(s-\frac12-\frac1p)}\|\d_x\dot\Delta_{j'}z\|_{L^p}.
$$
Therefore, 
\begin{equation}\label{eq:Rj2}
\sum_{j\geq J_0}\left(2^{js}\norme{\cR_j^2}_{L^2}\right)\leq 
C\Bigl(\norme{\partial_xw}_{L^\infty}\norme{\partial_xz}^h_{\dot{\mathbb{B}}^{s-1}_{2,1}}
+ \norme{\partial_xz}^\ell_{\dot{\mathbb{B}}^{s-\frac12-\frac{1}{p}}_{p,1}}\norme{\d_xw}^\ell_{\dot{\mathbb{B}}^{-\frac{1}{p^*}}_{p^*,1}}\Bigr)\cdotp
\end{equation}
Finally, for all $j\geq J_0$ and $|j'-j|\leq 1,$ we have
$$\begin{aligned}
2^{js}\|(\dot{S}_{j'-1}w-\dot{S}_{j-1}w)\dot{\Delta}_j\dot{\Delta}_{j'}\partial_xz\|_{L^2}
&\leq 2^j\|\dot\Delta_{j\pm1} w\|_{L^\infty}\, 2^{j(s-1)}\|\d_x\dot\Delta_{j'}\ddj z\|_{L^2}\\
&\leq C\|\dot\Delta_{j\pm1} \d_xw\|_{L^\infty}\, 2^{j(s-1)}\|\d_x\ddj z\|_{L^2}.\end{aligned}
$$
Hence
\begin{equation}\label{eq:Rj3}
\sum_{j\geq J_0}\left(2^{js}\norme{\cR_j^3}_{L^2}\right)\leq 
C\|\d_xw\|_{L^\infty}\|\d_xz\|^h_{\dot\B^{s-1}_{2,1}}.\end{equation}
Putting \eqref{eq:Rj1a}, \eqref{eq:Rj1b}, \eqref{eq:Rj2} and \eqref{eq:Rj3} together completes
the proof.
\end{proof}

\bibliographystyle{plain}
\nocite{*}

\bibliography{Biblio}

\end{document}